\documentclass{ims9x6}

\makeindex
\usepackage[font=footnotesize,labelfont=bf]{caption}
\usepackage{float}
\usepackage{subfloat}
\usepackage[skip=1pt]{subcaption}
\usepackage{cases}
\newtheorem{assum}{Assumption}
\usepackage{empheq}
\usepackage[most]{tcolorbox}
\usepackage{cite}

\allowdisplaybreaks

\newtcbox{\mymath}[1][]{%
    nobeforeafter, math upper, tcbox raise base,
    enhanced, colframe=blue!15!black, boxrule=.0pt, left=1mm, right=1mm, #1}

\newtcolorbox{empheqboxed}{colback=gray!35, 
 colframe=white,
 width=\textwidth,
 sharpish corners,
 top=-2mm, 
 bottom=0pt
}
    
\newcommand{\coloredeq}[2]{\begin{empheq}[box={\mymath[colback=gray!13, sharp corners]}]{align}\label{#1}#2\end{empheq}}

\usepackage{mdframed}
\usepackage{lipsum}

\newmdtheoremenv{theo-frmd}{Theorem}

\newcommand{\nn}{\nonumber}
\newcommand{\R}{{\mathbb R}}

\renewcommand{\tilde}{\widetilde}
\renewcommand{\hat}{\widehat}
\renewcommand{\bar}{\overline}

\newcommand{\CR}{\mathcal{C}([0,T];\mathbb{R}^d)}

\newcommand{\EE}{\mathbb{E}}
\newcommand{\TE}{\mathcal{F}}
\newcommand{\RR}{\mathbb{R}}
\newcommand{\NN}{\mathbb{N}}
\newcommand{\PP}{\mathbb{P}}

\newcommand{\OV}{\overline{V}}
\newcommand{\OX}{\overline{X}}
\newcommand{\OY}{\overline{Y}}
\newcommand{\CH}{{\mathcal{H}}}

\newcommand{\la}{\langle}
\newcommand{\ra}{\rangle}

\newcommand{\norm}[1]{\left\lVert#1 \, \right\rVert}

\begin{document}

\chapter{MEAN-FIELD PARTICLE SWARM OPTIMIZATION}

\markboth{S.~Grassi, H.~Huang, L.~Pareschi and J.~Qiu}{Mean-field particle swarm optimization}

\author{Sara Grassi}

\address{Department of Mathematical, Physical and Computer Sciences, \\
University of Parma,\\
Parco Area delle Scienze 7/A, 43124 Parma, Italy\\
sara.grassi{@}unipr.it }

\author{Hui Huang} 

\address{Department of Mathematics and Statistics, \\
University of Calgary, \\
2500 University Dr NW, Calgary, AB T2N 1N4, Canada\\
hui.huang1{@}ucalgary.ca }

\author{Lorenzo Pareschi}

\address{Department of Mathematics and Computer Science, \\
University of Ferrara,\\
 Via Machiavelli 30, 44121 Ferrara, Italy \\
lorenzo.pareschi{@}unife.it }

\author{Jinniao Qiu} 

\address{Department of Mathematics and Statistics, \\
University of Calgary, \\
2500 University Dr NW, Calgary, AB T2N 1N4, Canada\\
jinniao.qiu{@}ucalgary.ca}

\begin{abstract}
In this work we survey some recent results on the global minimization of a non-convex and possibly non-smooth high dimensional objective function by means of particle based gradient-free methods. Such problems arise in many situations of contemporary interest in machine learning and signal processing.
After a brief overview of metaheuristic methods based on particle swarm optimization (PSO), we introduce a continuous formulation via second-order systems of stochastic differential equations that generalize PSO methods and provide the basis for their theoretical analysis. Subsequently, we will show how through the use of mean-field techniques it is possible to derive in the limit of large particles number the corresponding mean-field PSO description based on Vlasov-Fokker-Planck type equations. Finally, in the zero inertia limit, we will analyze the corresponding macroscopic hydrodynamic equations, showing that they generalize the recently introduced consensus-based optimization (CBO) methods by including memory effects. Rigorous results concerning the mean-field limit, the zero-inertia limit, and the convergence of the mean-field PSO method towards the global minimum are provided along with a suite of numerical examples. 
\end{abstract}

\vspace*{12pt}

\chaptercontents  

\section{Introduction}

The Particle Swarm Optimization (PSO) algorithm was introduced by James Kennedy, a social psychologist, and Russel Eberhart, an electrical engineer, in the mid-1990s \cite{kennedy1995particle,kennedy1997particle}. Since its introduction, due to its simplicity and versatility, the PSO method has gained a great deal of attention from the scientific community, resulting in a huge number of variants of the original algorithm \cite{kennedy2010particle,poli2007particle,shieber98,hassan04,wang2017particle}. The origin of the method can actually be traced back to an earlier time, since the basic principle of optimization by interacting agents is inspired by previous attempts to reproduce the observed behaviors of animals in their natural habitat, such as flocks of birds or schools of fish \cite{Aoki,Okubo,Vicseck,Giardina2008collective,boids1987,CFT}. These roots in the natural processes of collective animal behavior lead to the PSO algorithm's classification as belonging to Swarm Intelligence (SI), where the notion of swarm intelligence refers to the property of a system in which the coordinated behaviors of agents interacting locally with their environment cause coherent global functional patterns (e.g., self-organization, emergent behavior) to emerge \cite{Sumpter2010collective,Dorigo,cucker2007,motsch2014,CS}. 

Currently, similar to other gradient-free approaches \cite{Back:1997:HEC:548530,Blum:2003:MCO:937503.937505,Gendreau:2010:HM:1941310,8843624,larson_menickelly_wild_2019,neumaier_2004,Audet2017}, PSO is considered an efficient metaheuristic method for solving complex optimization problems and is available in several programming language libraries. Gradient-based optimizers are effective at finding local minima for high-dimensional, nonlinearly constrained convex problems; however, most gradient-based optimizers have problems dealing with noisy, discontinuous functions, and are not designed to handle discrete and mixed discrete-continuous variables. Unlike gradient-based methods in a convex search space, metaheuristic methods are not necessarily guaranteed to find true global optimal solutions, but they are capable of finding many good solutions that are sometimes sufficient in practical applications. 
 Some of the most popular stochastic metaheuristic methods include Simulated Annealing (SA) \cite{holley1988simulated,kirkpatrick1983optimization,Aarts:1989:SAB:61990}, Ant Colony Optimization (ACO) \cite{dorigo1996ant,dorigo2005ant}, Genetic Algorithms (GA) \cite{Holland:1992:ANA:531075,Goldberg1989genetic} and Differential Evolution (DE) \cite{Fogel:2006:ECT:1202305, Storn1997differential}. See also \cite{8843624} for a recent survey on other natured inspired metaheuristics. It should also be mentioned that a large number of newer  metaheuristic methods have begun to attract criticism in the research community for hiding their lack of novelty behind elaborate constructions unsupported by any theoretical analysis \cite{Sorensen}. 

In spite of its apparent simplicity, PSO poses formidable challenges for those interested in understanding swarm intelligence through theoretical analysis. To date a fully complete mathematical theory for particle swarm optimization is still lacking (see for example \cite{bruned2019weak,Schmitt2015convergence,XU2018709,Bellomo17,poli2009mean,zhang2015comprehensive} and the references therein). The algorithm explores the search space in an intelligent way thanks to a population of particles interacting with each other and updated at each step their position and velocity. Thus, from the theoretical point of view, one can take advantage of the fact that PSO is inspired by classical second order Newtonian dynamics of particle systems. This allows approaches derived from statistical mechanics and mean-field theory to be adapted to the study of the system properties in the limit of a large number of particles \cite{Jab2,Snitzman,Golse,Jab,CC1,CC4,huangmf,BCC}. 
 
Analogies with mean-field dynamics in consensus formation have recently inspired Consensus-based Optimization (CBO) methods, a novel class of particle based methods for global optimization (see \cite{pinnau2017consensus,carrillo2018analytical,carrillo2019consensus,TW,Totzeck2018ANC,ha2020convergence,ha2021convergence,chen2020consensusbased,fornasier2021consensusbased} and the recent survey \cite{totzeck2021trends}). Global optimization methods with similar features, but based on Kuramoto-Vicseck  dynamics constrained to hypersurfaces \cite{fhps20-1,fhps20-2,fornasier2021anisotropic} or on binary Boltzmann dynamics \cite{benfenati2021binary}, have been introduced and studied recently. These methods are inherently simpler than PSO methods since they were inspired by first order consensus-like dynamics typical of social interactions such as opinion formations and wealth exchanges \cite{partos13,NPT}.
In contrast to classic metaheuristic methods typically formulated through a discrete sequence of operations and for which it is quite difficult to provide rigorous convergence to global minimizers, CBO-like methods, thanks to their formulation through stochastic differential equations (SDE) permit to exploit mean-field techniques to prove global convergence for a large class of optimization problems \cite{carrillo2018analytical,carrillo2019consensus,fhps20-2,fornasier2021anisotropic}. On the other hand, CBO methods seem to be powerful and robust enough to tackle many interesting high dimensional non-convex optimization problems of interest in machine learning and sampling \cite{carrillo2019consensus,fhps20-2,chen2020consensusbased,carrillo2021consensus,benfenati2021binary,fornasier2021anisotropic,grassi2021consensus}.

In this work we review some recent results on the mean-field modeling of particle swarm optimization with the goal of providing a robust mathematical theory for PSO methods and their convergence to the global minimum, based on a continuous description of their dynamics \cite{Grassi2021PSO,huang2021mean,huang2021mean1,cipriani2021zero,HJK,grassi2021consensus}. 
A major difficulty in the mathematical description of PSO methods, and other metaheuristic algorithms, is the presence of memory mechanisms that make their interpretation in terms of differential equations particularly challenging. To this end, 
the discrete PSO method is generalized via a system of second-order SDEs in which an additional state variable takes into account the memory of the individual particle. We refer to \cite{TW} for alternative approaches to memory mechanisms in CBO system. 

Adopting the same regularization process for the global best as in CBO methods \cite{pinnau2017consensus,carrillo2018analytical}, it is then possible to pass to the mean-field limit and derive the corresponding Vlasov-Fokker-Planck equation that characterizes the behavior of the system in the limit of a large number of particles \cite{Grassi2021PSO,huang2021mean1}. The new mathematical formalism based on mean-field equations permits to study the behavior of the Vlasov-Fokker-Planck PSO model in the limit of zero inertia (see \cite{AS,DLP,sun1,sun2,choi2020quantified,carrillo2021large} for related results in other contexts). In particular, we prove that in this limit the PSO dynamics is described by simplified macroscopic models that correspond to a generalization of CBO models including memory effects and local best \cite{Grassi2021PSO,cipriani2021zero}. The convergence of the mean-field PSO model to the global minimum is then discussed and shown rigorously in absence of memory effects \cite{HJK}.  
 
Several numerical examples are reported to validate of the mean-field process and the small inertia limit, and to illustrate the role of the various parameters involved in solving high dimensional global optimization problems for various prototype test functions.
Other than the basic algorithmic aspects of implementing these generalized PSO methods, we do not discuss the practical algorithmic improvements that can be adopted to increase the success rate, like for example the use of random batch methods \cite{AlPa,carrillo2019consensus,JLJ}, particle reduction techniques \cite{fhps20-2,fornasier2021anisotropic} and parameters adaptivity \cite{poli2007particle,wang2017particle}. We refer to \cite{grassi2021consensus} for further details on these implementation aspects.

The rest of the survey is organized as follows. In Section 2 we introduce the PSO algorithms and derive the corresponding representations as SDEs using a time continuous approximation of the memory process. Next, in Section 3, thanks to a regularization of the global best and the local best we discuss the large particle limit and derive the respective Vlasov-Fokker-Planck equations describing the mean-field dynamic. A rigorous proof of the mean-field limit is also given. 
Section 4 is then dedicated to the zero-inertia limit for the mean-field system that allows to recover a CBO model with local best as the corresponding macroscopic limit. This is shown rigorously at the end of the Section. A general convergence result to the global minimum is illustrated in Section 5 in absence of memory effects.  
Several numerical examples, validating the mean-field approximation, the small inertia limit and testing the performances of the minimizers against some prototype functions in high dimension are then given in Section 6. Some concluding remarks and open research directions are reported at the end of the manuscript. 

\begin{figure}[t]
\begin{center}
\includegraphics[scale=0.35]{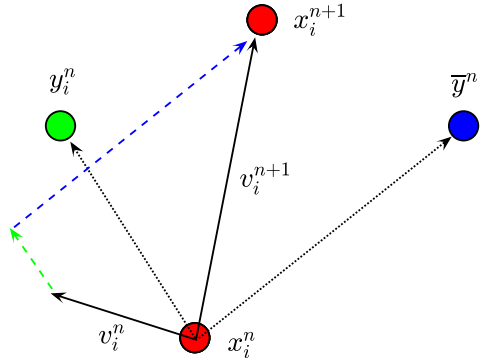}
\end{center}
\caption{Particle dynamics in the standard PSO model \eqref{eq:pso}. Green and blue dashed arrows denote the influence of the local best and global best, respectively.}
\label{Fg:pso}
\end{figure}

\section{Second order stochastic models for particle swarm optimization}
In the sequel we consider the following optimization problem
\begin{equation}\label{typrob}
x^\ast \in \argmin\limits_{x\in \RR^d}\TE(x)\,,
\end{equation}
where $\TE(x):\mathbb R^{d} \to \mathbb R$ is a given  high dimensional objective function, which we wish to minimize. In machine learning the objective function allows the algorithm designer to encode the appropriate and expected behavior for the machine learning model, such as fitting well to the training data versus some loss function. Modern applications frequently require learning algorithms to operate in extremely high dimensional spaces \cite{bishop06,vapnik91}. In other applications, the natural objective of the learning task is a possibly non-smooth and non-convex function \cite{nonconvex}. Common examples include training deep neural networks and tensor decomposition problems.
In contrast to gradient based optimizers and other metaheuristic solvers, PSO solve the minimization problem \eqref{typrob} by starting from a population of candidate solutions, represented by particles, and moving these particles in the search space according to simple mathematical relationships on particle position and speed.  The movement of each particle is influenced by its best known local position, but it is also driven to the best collective position of the swarm in the search space, which is updated when the particles find better positions (see Figure \ref{Fg:pso}). 

\subsection{The standard PSO algorithm}

The method is based on introducing $N$ particles with position $x_i\in {\RR}^d$ and speed $v_i\in {\RR}^d$, $i=1,\ldots,N$. In the \emph{standard PSO algorithm} the particle positions and velocities, starting with an initial $x_i^0$ and $v_i^0$ assigned, are updated according to the following rule \cite{kennedy1995particle}
\coloredeq{eq:pso}{
\begin{split}
x^{n+1}_i &= x_i^n + v_i^{n+1},\\
v^{n+1}_i &= v_i^n + c_1 R_1^n \left(\p_i^n-x_i^n\right)+c_2 R_2^n \left(\g^n-x_i^n\right),
\end{split}}
where the values $c_1, c_2 \in {\RR}$ are the {\em acceleration coefficients}, $\p_i^n$ is the {\em local best} position found by the $i$ particle up to that iteration, and $\g^n$ is the {\em global best} position found among all the particles up to that iteration. The terms $R_1^n$ and $R_2^n$ denote two $d$-dimensional diagonal matrices with random numbers uniformly distributed in $[0,1]$ on their diagonals. These numbers are generated at each iteration and for each particle.
Typically, the values of $x_i$ and $v_i$ are restricted within a specific search domain $X=[X_{min},X_{max}]^d$ and velocity range $V=[-V_{max},V_{max}]^d$. Different boundary conditions are usually applied in the search space $X$.

The local best $\p_i^n$ and global best $\g^n$ are defined by the following relationships
\begin{eqnarray}
\nonumber
\p_i^0&=&x_i^0,\\ 
\nonumber
\p_i^{n+1} &=& \left\{
\begin{array}{lcl}
\p_i^n  & \hbox{if}   & \TE(x_i^{n+1}) \geq \TE(y_i^n),  \\ 
x_i^{n+1}  & \hbox{if}   & \TE(x_i^{n+1}) < \TE(y_i^n),  
\end{array}
\right.\\ 
\label{eq:psoEvolution}
\\
\nonumber
\g^0&=&\hbox{argmin}\{\TE(x_1^0),\TE(x_2^0),\ldots,\TE(x_N^0)\},\\ 
\nonumber
\g^{n+1} &=& \hbox{argmin}\{\TE(y_1^{n+1}),\TE(y_2^{n+1}),\ldots,\TE(y_N^{n+1})\}. 
\end{eqnarray}
Another way to represent the local best, which will be useful in the sequel, is the following \cite{hassan04}
\be
\p_i^{n+1} = \p_i^n + \frac12\left(x_i^{n+1}-\p_i^n\right)S(x_i^{n+1},\p_i^n),
\label{eq:hass}
\ee
where
\be
S(x,y)=\left(1+\sign\left(\TE(y)-\TE(x)\right)\right).
\ee

\subsection{The stochastic differential PSO system}
In order to derive a time continuous version of the PSO algorithm \eqref{eq:pso}, we rewrite it in the form
\begin{equation}
\begin{split}
x^{n+1}_i &= x_i^n + {v_i^{n+1}},\\
v^{n+1}_i &= v_i^n + \frac{c_1}{2}\left(\p_i^n-x_i^n\right)+\frac{c_2}{2}\left(\g^n-x_i^n\right) \\
&\quad +\frac{c_1}{2}\widetilde R_1\left(\p_i^n-x_i^n\right)+\frac{c_2}{2}\widetilde R_2\left(\g^n-x_i^n\right),
\end{split}
\label{eq:pso2}
\end{equation} 
where $\widetilde R_k=(2R_k-1)$, $k=1,2$. We can interpret \eqref{eq:pso2} as a semi-implicit time discretization method for SDEs with time stepping $\Delta t=1$ where the implicit Euler scheme has been used for the first equation and the Euler-Maruyama method is used for the second one. Note that, the particular distribution of the random noise will not change the corresponding stochastic differential system provided the noise has the same mean value and variance. In the case of the PSO model \eqref{eq:pso2}, since the random terms are uniformly distributed in $[-1,1]$, the mean value is $0$ and the corresponding variance is $1/3$. 

We can then write the time continuous formulation as a {second order system of SDEs} in It\^o form defining the \emph{stochastic differential PSO system}
\coloredeq{eq:psoc}{
\begin{split}
dX^i_t &= V^i_t dt,\\ 
dV^i_t &= \lambda_1\left(\P_t^i-X^i_t\right)dt+\lambda_2\left(\G_t-X^i_t\right)dt \\
 &\quad +\sigma_1 D(\P_t^i-X^i_t)dB^{1,i}_t+\sigma_2 D(\G_t-X^i_t)dB^{2,i}_t,
\end{split} 
}
with 
\begin{equation}
\lambda_k=\frac{c_k}{2},\quad \sigma_k = \frac{c_k}{2\sqrt{3}},\quad k=1,2 
\label{eq:param}
\end{equation}
the \emph{drift and diffusion coefficients} and
\begin{equation}
D(X_t)=\diag\left\{(X_t)_1,(X_t)_2,\dots,(X_t)_d\right\},
\end{equation}
a $d$-dimensional diagonal matrix. \\ In \eqref{eq:psoc} the vectors $B^k_t=\left((B^k_t)_1,(B_t^k)_2,\dots,(B_t^k)_d\right)^T$, $k=1,2$ denote $d$ independent 1-dimensional \emph{Brownian motions} and depend on the $i$-th particle. 
One critical aspect is the definition of the best positions $\P_t^i$ and $\G_t$ which in the PSO method make use of the past history of the particles. Thanks to \eqref{eq:hass}, for a positive constant $\nu$, we can approximate $y_i^{n+1}$ with the following \emph{differential system for the local best}
\coloredeq{eq:lbest}{
d\P_t^i = \nu \left(X^i_t-\P^i_t\right)S(X^i_t,\P^i_t)dt,
}
with $Y^i_0=X^i_0$
and consequently define
\begin{equation}
\G_t = \hbox{argmin}\left\{\TE(\P^1_t),\TE(\P^2_t),\ldots,\TE(\P^N_t)\right\}.
\end{equation}
Note that, equation \eqref{eq:lbest} does not describe 
the evolution of the local best, but rather a time continuous approximation of its evolution. 

\begin{figure}[t]
\includegraphics[scale=0.3]{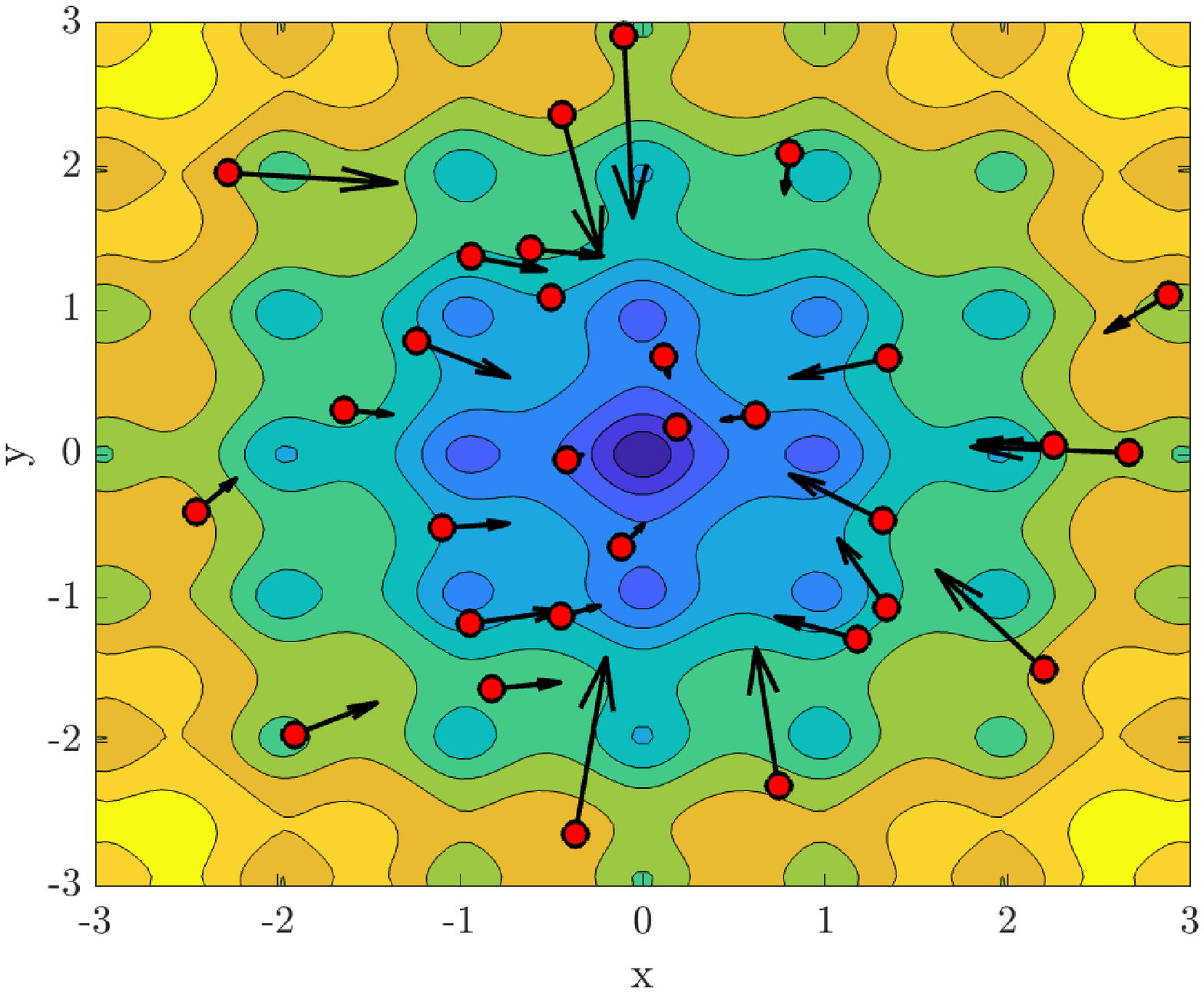}\hskip -.45cm
\includegraphics[scale=0.3]{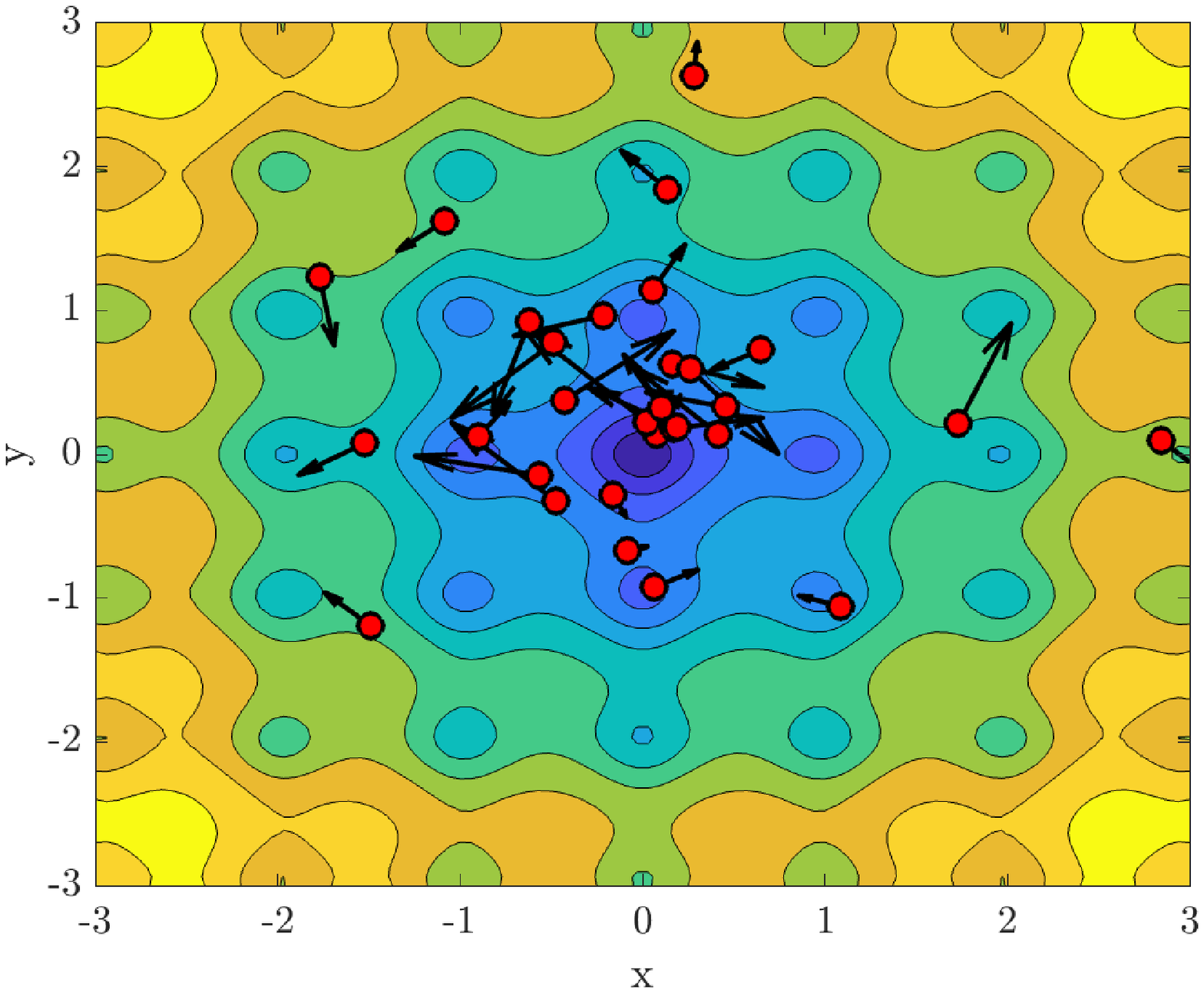}\hskip -.45cm
\caption{Snapshots of the PSO minimization process \eqref{eq:psoi} for the two-dimensional Ackley function (see Table \ref{Tabfun}) using $N=30$ particles, with $m=0$, $c_1=0.25$ and $c_2=2$.}
\end{figure}

\subsection{Stochastic differential PSO model with inertia}
To optimize the search algorithm, the value $c_k= 2$, $k=1,2$ was adopted in early PSO research. This value, which corresponds to $\lambda_k=1$ and $\sigma_k=1/\sqrt{3}$, $k=1,2$ in the SDEs form, however, may lead to unstable dynamics with particle speed increase without control. The use of hard bounds on velocity in $[-V_{\max},V_{\max}]^d$ is one way to control the velocities. However, the value
of $V_{\max}$ is problem-specific and difficult to determine. For this reason, the \emph{PSO algorithm with inertia} has been  considered \cite{shieber98}
\coloredeq{eq:psoi}{
\begin{split}
x^{n+1}_i &= x_i^n + v_i^{n+1},\\
v^{n+1}_i &= \iw  v_i^n + c_1 R_1^n \left(\p_i^n-x_i^n\right)+c_2 R_2^n \left(\g^n-x_i^n\right),
\end{split}
}
where $\iw \in (0,1]$ is the \emph{inertia weight}. The above system can be rewritten as
\begin{equation}
\begin{split}
x^{n+1}_i &= x_i^n + v_i^{n+1},\\
\iw  v^{n+1}_i &= \iw  v_i^n - (1-\iw ) v_i^{n+1} + c_1 R_1^n \left(\p_i^n-x_i^n\right)+c_2 R_2^n \left(\g^n-x_i^n\right).
\end{split}
\label{eq:psoi2}
\end{equation}
In this case, we can interpret the second equation as a semi-implicit Euler-Maruyama method, that is implicit in $v_i$ and explicit in $x_i$, hence the corresponding \emph{stochastic differential PSO system with inertia} reads
\coloredeq{eq:psoci}{
\begin{split}
dX^i_t &= V^i_t dt,\\
\iw  dV^i_t &= -\gamma V^i_t dt +\lambda_1\left(\P_t^i-X^i_t\right)dt +\lambda_2\left(\G_t-X^i_t\right) dt \\
&\quad +\sigma_1 D(\P_t^i-X^i_t)dB^{1,i}_t+\sigma_2 D(\G_t-X^i_t)dB^{2,i}_t,
\end{split}} 
where $\gamma=(1-\iw ) \geq 0$ is the \emph{friction coefficent}. Thus, the constant $\gamma$ acts effectively as a friction coefficient, and can be related to the fluidity of the medium in which particles move. System \eqref{eq:psoci} is reminescent of other second order stochastic particle system with inertia \cite{AS,DLP}. However, note that here, the inertia weight $\iw $ and the friction coefficient $\gamma$ are not independent.

In practice, in the PSO method \eqref{eq:psoi} the parameter $\gamma$ is often initially set to some low value, which corresponds to a system where particles move in a low viscosity medium and perform extensive exploration, and gradually increased to a higher value closer to one, where the system is more dissipative and would more easily concentrate into local minima.
Most PSO approaches, nowadays, are based on \eqref{eq:psoi} (or some variant) which is usually referred to as canonical PSO method to distinguish it from the original PSO method \eqref{eq:pso} (see \cite{poli2007particle}). Similarly we will refer to \eqref{eq:psoc}-\eqref{eq:lbest} as the original stochastic differential PSO (SD-PSO) system and to \eqref{eq:psoci}-\eqref{eq:lbest} as the canonical SD-PSO system.  
\begin{remark}
We underline that the PSO stochastic systems \eqref{eq:psoci}-\eqref{eq:lbest} if discretized properly yields the  PSO algorithm with inertia \eqref{eq:psoi}. This is achieved discretizing \eqref{eq:psoci} implicitly in $V^i_t$ and explicitly in $X^i_t$, and \eqref{eq:lbest} implicitly in $X^i_t$ and explicitly in $Y^i_t$. Taking $\Delta t=1$, $\nu=1/2$, the drift and diffusion terms satisfying \eqref{eq:param}, and a uniform noise permits to recover exactly \eqref{eq:psoi}. We refer to the last part of the manuscript containing the numerical examples for further details.
\end{remark}


\section{Mean-field particle swarm optimization}
In this section we introduce a modified version of the canonical stochastic differential PSO system for which we can formally compute its mean-field limit. We first consider the case in absence of memory effects and then we extend the results to the general case.
Throughout this note, our theoretical analysis assumes the cost function $\TE$ satisfies the following
\begin{assum}\label{asum}
	For the given cost function $\TE:\RR^d\rightarrow \RR$, it holds that:
	\begin{itemize}
		\item[(1)]  There exists some constant $L>0$ such $|\TE(x)-\TE(y)|\leq L(|x|+|y|)|x-y|$ for all $x,y\in\RR^d$;
		\item[(2)]	$\TE$ is  bounded from below with $-\infty<\underline{\TE}:=\inf \TE$ and there exists some constant $C_u>0$ such that
		\begin{equation*}
		\TE(x)-\underline{\TE}\leq C_u(1+|x|^2)\mbox{ for all }x\in\RR^d\,;
		\end{equation*}
		\item[(3)]  $\TE$ has quadratic growth at infinity. Namely, there  exist constants $C_l,\, M>0$ such that
		\begin{equation*}
		\TE(x)-\underline{\TE}\geq C_l|x|^2\mbox{ for all }|x|\geq M\,.
		\end{equation*}
	\end{itemize}
\end{assum}

\subsection{Regularized PSO dynamics without memory effects}
To simplify the mathematical description, let us consider a PSO approach where the dynamic is instantaneous without memory of the local best positions and the global best has been regularized as in \cite{pinnau2017consensus}. The corresponding second order system of SDEs describing the \emph{regularized SD-PSO dynamics} takes the form\footnote{The superscript $N$ is used to emphasize the dependence on the number of particles in the system.} 
\coloredeq{PSO}{
	\begin{split}
		dX_t^{i,N} & = V_t^{i,N}d t, \\
		mdV_t^{i,N} & =-\gamma V_t^{i,N}dt+\lambda (X^\alpha(\rho_t^{N})-X_t^{i,N})dt\\
		&\quad +\sigma D(X^\alpha(\rho_t^{N})-X_t^{i,N})dB_t^i\,,
	\end{split}}
where the $\RR^d$-valued functions $X_t^{i,N} $ and $V_t^{i,N}$ denote the position and velocity of the $i$-th particle at time $t$,  and $\{(B_t^i)_{t\geq0}\}_{i=1}^N$ are $N$ independent $d$-dimensional Brownian motions. 
Here the weighted average \emph{regularization of the global best} is given by
\coloredeq{ValphaE}{
	{X}^{\alpha}(\rho_t^{N}):=\frac{\int_{\RR^d}x\omega_{\alpha}^{\TE}(x)\rho_t^{N}(dx)}{\int_{\RR^d}\omega_{\alpha}^{\TE}(x)\rho_t^{N}(dx)},
}
with  the empirical measure $\rho^N:=\frac{1}{N}\sum_{i=1}^{N}\delta_{X^{i,N}}$, which is the spacial marginal of
$f^N:=\frac{1}{N}\sum_{i=1}^{N}\delta_{(X^{i,N},V^{i,N})}$.
The choice of the weight function $\omega_\alpha^\TE(x):=e^{-\alpha\TE(x)}$ in \eqref{ValphaE} comes from  the  well-known \emph{Laplace principle}, a classical result in large deviation theory, which states that for any probability measure $\rho\in\mc{P}(\RR^d)$ compactly supported, it holds
\vspace{5pt}
\begin{equation}\label{Laplace}
\lim\limits_{\alpha\to\infty}\left(-\frac{1}{\alpha}\log\left(\int_{\RR^d}e^{-\alpha\TE(x)}\rho(dx)\right)\right)=\inf\limits_{x\, \in\, \rm{supp }(\rho)} \TE(x)\,.
\end{equation}
\vspace{7pt}
Therefore, for large values of $\alpha \gg 1$ the regularized global best $X^\alpha(\rho_t^{N}) \approx X_t^*$, where
\[
X_t^* = \argmin\left\{\TE(X_t^{1,N}),\TE(X_t^{2,N}),\ldots,\TE(X_t^{N,N})\right\}.
\] 
We emphasize that the stochastic particle system \eqref{PSO} has locally Lipschitz coefficients, thus it admits strong solutions and pathwise uniqueness holds up to any finite time $T>0$, see \cite{CS,Dur}. The above system of SDEs in the sequel is considered in a general setting, without necessarily satisfying the PSO constraint \eqref{eq:param}.

As the particle number $N\to \infty$, one expects to derive the  \emph{mean-filed PSO description without local best} characterized by the following nonlinear \emph{Vlasov-Fokker-Planck equation}
\coloredeq{PDEi}{
\begin{split}
& \partial_t f + \  v \cdot \nabla_x f = \\ 
&\quad \nabla_v\cdot\left(\frac{\gamma}{\iw } v f + \frac{\lambda}{\iw } (x - X^{\alpha}(\rho) )f+\frac{\sigma^2}{2\iw ^2}D(x - X^{\alpha}(\rho))^2\nabla_v f\right)
\end{split}} 
where we have used the identity
\vspace{2pt}
\begin{equation}
\sum_{j=1}^d \frac{\partial^2}{\partial v^2_j}\left((x-X^{\alpha}(\rho))^2_j f\right) = \nabla_{v}\cdot \left(D(x - X^{\alpha}(\rho))^2\nabla_{v} f\right)
\label{eq:newid}
\end{equation}
\\
with $D(x - X^{\alpha}(\rho) )^2$ the diagonal matrix given by the square of $D(x - X^{\alpha}(\rho) )$. 
Equation \eqref{PDEi} represents the mean-field PSO (MF-PSO) model without local best and 
should be accompanied by initial (and boundary)
data, and normalization
\[
\iint_{\RR^{2d}} f(t,dx,dv) = 1.
\]
We refer to \cite{CFT,BCC,Jab,Golse,Snitzman} and the references therein, for more details and rigorous results about mean-field models of Vlasov-Fokker-Planck type. Note, however, that the presence of $X^{\alpha}(\rho)$ makes the Vlasov-Fokker-Planck equation nonlinear and nonlocal. This is nonstandard in the literature and raises several analytical and numerical questions (see \cite{carrillo2018analytical,fhps20-2}).

\subsubsection{Mean-field limit}\label{meanPSO1}
In this section, following \cite{huang2021mean1} we provide a rigorous justification of the mean-field limit of  PSO model \eqref{PSO} towards its mean-field PDE \eqref{PDEi} through a compactness argument. More precisely, we first prove  that the sequence of empirical measures $\{f^N\}_{N\geq 2}$ ($f^N=\frac{1}{N}\sum_{i=1}^{N}\delta_{(X^{i,N},V^{i,N})}$ are $\mc{P}(\CR\times\CR)$-valued random variables) is tight.  Prokhorov's theorem indicates that there exists a subsequence of $\{f^N\}_{N\geq 2}$ converging in law to a random measure $f$. Then, to identify the limit, we verify that the limit measure $f$ is a weak solution to the mean-field PSO equation \eqref{PDEi}  almost surely, while the uniqueness of the weak solution to PDE \eqref{PDEi} yields that $f$ is actually deterministic.  Our main result can be described in the following way:
\begin{theo-frmd}\label{thmmean}
	Let $\TE$ satisfy Assumption \ref{asum} and $f_0\in \mc{P}_4(\RR^{2d})$. For any $N\geq 2$, we assume that $\{(X_t^{i,N},V_t^{i,N})_{t\in[0,T]}\}_{i=1}^N$ is the unique solution to the SD-PSO system \eqref{PSO} with $f_0^{\otimes N}$-distributed initial data $\{X_0^{i,N},V_0^{i,N}\}_{i=1}^N$.  Then the limit (denoted by $f$) of the sequence of the empirical measure $f^N=\frac{1}{N}\sum_{i=1}^N\delta_{(X^{i,N},V^{i,N})}$ exists. Moreover, $f$ is the unique weak solution to the MF-PSO equation \eqref{PDEi}.
\end{theo-frmd}
To obtain the above theorem,  let us first prove the following lemma on a uniform moment estimate for the particle system \eqref{PSO}.
\begin{lemma}\label{lem2mon}
	Let $\TE$ satisfy Assumption \ref{asum} and $f_0\in \mc{P}_4(\RR^{2d})$. For any $N\geq 2$, assume that $\{(X_t^{i,N},V_t^{i,N})_{t\in[0,T]}\}_{i=1}^N$ is the unique solution to the SD-PSO system \eqref{PSO} with $f_0^{\otimes N}$-distributed initial data $\{(X_0^{i,N},V_0^{i,N})\}_{i=1}^N$. Then there exists a constant $K>0$ independent of $N$ such that
	\begin{equation}\label{est-lem2mon}
	\begin{split}
	&\sup\limits_{i=1,\cdots,N} \left\{
	\sup\limits_{t\in[0,T]}
	\EE\left[|X_t^{i,N}|^2+|X_t^{i,N}|^4+|V_t^{i,N}|^2+|V_t^{i,N}|^4\right]	\right\}\\ 
	&\quad +\sup\limits_{t\in[0,T]}\EE\left[|X^\alpha(\rho^N_t)|^2+|X^\alpha(\rho^N_t)|^4\right]
	\leq K\,. 
	\end{split}
	\end{equation}
\end{lemma}
The proof follows similar arguments as in \cite[Lemma 3.4]{carrillo2018analytical}.

We treat $(X^{i,N},V^{i,N}): \Omega\mapsto \CR\times \CR$. Then
$f^N=\sum_{i=1}^{N}\delta_{(X^{i,N},V^{i,N})}: \Omega\mapsto \mc{P}(\CR\times\CR)$ is a random measure. Let us denote $\mc{L}(f^N):=\mbox{Law}(f^N)\in \mc{P}(\mc{P}(\CR\times\CR))$.  We can prove that $\{\mc{L}(f^N)\}_{N\geq 2}$ is tight, or  we say
$\{f^N\}_{N\geq2}$ is tight, which can be done by verifying the Aldous criteria \cite{billingsley2013convergence} as presented below:
	\begin{lemma}\label{lemAldous}
	Let $\{X^n\}_{n\in \NN}$ be a sequence of random variables defined on a probability space $(\Omega,\mc{F},\PP)$ and valued in $\mc{C}([0,T];\RR^d)$. The sequence of probability distributions $\{\mu_{X^n}\}_{n\in \NN}$  of $\{X^n\}_{n\in \NN}$ is tight on $\mc{C}([0,T];\RR^d)$ if the following two conditions hold.
	
	$(Con 1)$ For all $t\in [0,T]$, the set of distributions of $X_t^n$, denoted by $\{\mu_{X_t^n}\}_{n\in\NN}$, is tight as a sequence of probability measures on $\RR^d$.
	
	$(Con 2)$ For all $\varepsilon>0$, $\eta>0$, there exists $\delta_0>0$ and $n_0\in\NN$ such that for all $n\geq n_0$ and for all discrete-valued $\sigma(X^n_s;s\in[0,T])$-stopping times $\beta$ with $0\leq \beta+\delta_0\leq T$, it holds that
	\begin{equation}
	\sup_{\delta\in[0,\delta_0]}\PP\left(|X^n_{\beta+\delta}-X^n_{\beta}|\geq \eta\right)\leq \varepsilon\,.
	\end{equation}
\end{lemma}
We can then prove:
\begin{theorem}\label{thmtight}
	Let $\TE$ satisfy Assumption \ref{asum} and $f_0\in \mc{P}_4(\RR^{2d})$. For any $N\geq 2$, we assume that $\{(X_t^{i,N},V_t^{i,N})_{t\in[0,T]}\}_{i=1}^N$ is the unique solution to the SD-PSO system \eqref{PSO} with $f_0^{\otimes N}$-distributed initial data $\{X_0^{i,N},V_0^{i,N}\}_{i=1}^N$. Then the sequence $\{\mc{L}(f^N)\}_{N\geq 2}$ is tight in $\mc{P}(\mc{P}(\CR\times\CR))$.
\end{theorem}
\begin{proof}
	According to Proposition 2.2 $(ii)$ in \cite[Proposition 2.2 (ii)]{Snitzman}, we only need to prove that $\{\mc{L}((X^{1,N},V^{1,N}))\}_{N\geq 2}$ is tight in $\mc{P}(\CR\times\CR)$ because of the exchangeability of the particle system. It is sufficient to justify conditions $(Con 1)$ and $(Con 2)$ in Lemma \ref{lemAldous}.
	
	$\bullet$ \textit{Step 1: Checking $(Con 1)$. }  
	For any $\varepsilon>0$, there exists a compact subset $U_\varepsilon:=\{(x,v):~|x|^2+|v|^2\leq \frac{K}{\varepsilon}\}$ such that by Markov's inequality
	\begin{align*}
	&\mc{L}((X_t^{1,N},V_t^{1,N}))~\big((U_\varepsilon)^c\big)=\PP\left(|X_t^{1,N}|^2+|V_t^{1,N}|^2> \frac{K}{\varepsilon}\right)\\
	&\qquad \leq \frac{\varepsilon\EE[|X_t^{1,N}|^2+|V_t^{1,N}|^2]}{K}\leq\varepsilon,\quad \forall ~N\geq 2\,,
	\end{align*}
	where we have used Lemma \ref{lem2mon} in the last inequality.
	This means that for each $t\in[0,T]$, the sequence $\{\mc{L}((X_t^{1,N},V_t^{1,N}))\}_{N\geq 2}$ is tight, which verifies  condition $(Con 1)$ in Lemma \ref{lemAldous}.
	
	$\bullet$ \textit{Step 2: Checking $(Con 2)$. }   Let $\beta$ be a $\sigma((X_s^{1,N},V_s^{1,N});s\in[0,T])$-stopping time with discrete values such that $\beta+\delta_0\leq T$.  It is easy to see that
	\begin{equation}
		\EE [|X_{\beta+\delta}^{1,N}-X_{\beta}^{1,N}|^2]\leq \delta\int_0^{T} \EE[|V_s^{1,N}|^2]ds\leq C\delta\,,
	\end{equation}
	where $C>0$ is independent of $N$ by \eqref{est-lem2mon}. Furthermore, 
	\begin{align*}
		V_{\beta+\delta}^{1,N}-V_{\beta}^{1,N}&=-\frac{\gamma}{m}\int_\beta^{\beta+\delta}V_{s}^{1,N}ds+\frac{\lambda}{m}\int_{\beta}^{\beta+\delta }(X^\alpha(\rho_s^N)-X_s^{1,N})ds\nn\\
		&\quad+\frac{\sigma}{m}\int_{\beta}^{\beta+\delta}D(X^\alpha(\rho_s^N)-X_s^{1,N})dB_s^1\,.
	\end{align*}
	Notice that
	\begin{equation}
	\label{es1}
	\begin{split}
		&\EE\left[\left|\int_{\beta}^{\beta+\delta }(X^\alpha(\rho_s^N)-X_s^{1,N})ds\right|^2\right]\leq \delta \int_0^T\EE\left[|X^\alpha(\rho_s^N)-X_s^{1,N}|^2\right]ds\\
		&\quad \leq  2 \delta T \left(\sup\limits_{t\in[0,T]}\EE\left[|X_t^{1,N}|^2\right] +\sup\limits_{t\in[0,T]}\EE\left[|X^\alpha(\rho^N_t)|^2\right]\right)\leq 2TK\delta\,,
	\end{split}
	\end{equation}
	where we have used Lemma \ref{lem2mon} in the last inequality.  Similarly we have
	\begin{align}\label{es1'}
	\EE\left[\left|\int_{\beta}^{\beta+\delta } V_s^{1,N}ds\right|^2\right] \leq TK\delta\,.
	\end{align}
	Further we apply It\^{o}'s isometry 
	\begin{align}\label{es2}
		&\EE\left[\left|\int_{\beta}^{\beta+\delta}D (X^\alpha(\rho_s^N)-X_s^{1,N})dB_s^1\right|^2\right]
		= \EE\left[\int_{\beta}^{\beta+\delta}|X^\alpha(\rho_s^N)-X_s^{1,N}|^2ds\right]
		\notag\\
		&\quad\leq
		 \delta^{\frac{1}{2}}  \EE \left[\left(\int_{0}^{T}|X^\alpha(\rho_s^N)-X_s^{1,N}|^4ds\right)^{\frac{1}{2}}\right]\nn\\
		&\quad \leq   \delta^{\frac{1}{2}}
		\left(\int_0^T\EE[|X^\alpha(\rho_s^N)-X_s^{1,N}|^4] ds\right)^\frac{1}{2}
	\leq
		 \delta^{\frac{1}{2}} T^{\frac{1}{2}}(8K)^{\frac{1}{2}}\,.
	\end{align} 
	Combining estimates \eqref{es1}--\eqref{es2} one has
	\begin{equation}
	\EE[|V_{\beta+\delta}^{1,N}-V_{\beta}^{1,N}|^2]\leq C(\gamma.\lambda,m,\sigma,T,K)\left( \delta^{\frac{1}{2}} + \delta\right)  \,.
	\end{equation}
	Hence, for any $\varepsilon>0$, $\eta>0$, there exists some $\delta_0>0$ such that for all $N\geq 2$ it holds that
	\begin{align}
	&\sup_{\delta\in[0,\delta_0]}\PP\left(|X_{\beta+\delta}^{1,N}-X_{\beta}^{1,N}|^2+|V_{\beta+\delta}^{1,N}-V_{\beta}^{1,N}|^2\geq \eta\right)\nn\\
	&\qquad\leq \sup_{\delta\in[0,\delta_0]}\frac{\EE\left[|X_{\beta+\delta}^{1,N}-X_{\beta}^{1,N}|^2+|V_{\beta+\delta}^{1,N}-V_{\beta}^{1,N}|^2\right]}{\eta}\leq \varepsilon\,.
	\end{align}
	 Hence $(Con 2)$ is verified.
\end{proof}
For any $\varphi\in \mc{C}_c^2(\R^d\times \R^d)$,   define a functional on $\mc{P}(\CR\times \CR)$ as follows
{\small \begin{align*}
	&F_{\varphi}(f)\nn
	:=\la \varphi(\textbf{x}_t,\textbf{v}_t),f(d\textbf{x},d\textbf{v})\ra-\la \varphi(\textbf{x}_0,\textbf{v}_0),f(d\textbf{x},d\textbf{v}) \ra 
	+\!\!\int_0^t\la \textbf{v}_s\cdot\nabla_x\varphi,f(d\textbf{x},d\textbf{v})\ra ds
	\nn\\
	&\qquad -\frac{\gamma}{m}\int_0^t\la \textbf{v}_s\cdot\nabla_v\varphi,f(d\textbf{x},d\textbf{v})\ra ds+\frac{\lambda}{m}\int_0^t\la (\textbf{x}_s-X^\alpha(\rho_s))\cdot \nabla_v\varphi, f(d\textbf{x},d\textbf{v})\ra ds\notag\\
	&\qquad -\frac{\sigma^2}{2m^2}\int_0^t\sum_{k=1}^{d}\la (\textbf{x}_s-X^\alpha(\rho_s))_k^2\frac{\partial^2\varphi}{\partial v_k^2},f (d\textbf{x},d\textbf{v})\ra ds
	\notag\\
	&\quad =\la \varphi(x,v) ,f_t (dx,dv)\ra-\la \varphi(x,v),f_0(dx,dv) \ra +\int_0^t\la v\cdot\nabla_x\varphi,f_s(dx,dv)\ra ds
	\nn\\
	&\qquad -\frac{\gamma}{m}\int_0^t\la v\cdot\nabla_v\varphi,f_s(dx,dv)\ra ds+\frac{\lambda}{m}\int_0^t\la (x-X^\alpha(\rho_s))\cdot \nabla_v\varphi, f_s (dx,dv)\ra ds\notag\\
	&\qquad -\frac{\sigma^2}{2m^2}\int_0^t\sum_{k=1}^{d}\la (x-X^\alpha(\rho_s))_k^2\frac{\partial^2\varphi}{\partial v_k^2},f_s (dx,dv)\ra ds \,,
\end{align*}}
for all $f\in \mc{P}(\CR\times \CR)$ and $\textbf{x},\textbf{v}\in \CR$, where $\rho_t(x)=\int_{ \RR^d }f_t(x,dv)$. 

Then we have the following estimate by the reasoning in  \cite[Proposition 3.2]{huang2021mean1}.
\begin{lemma}\label{prop}
	Let $\TE$ satisfy Assumption \ref{asum} and $f_0\in \mc{P}_4(\RR^{2d})$. For any $N\geq 2$, assume that $\{(X_t^{i,N},V_t^{i,N})_{t\in[0,T]}\}_{i=1}^N$ is the unique solution to the SD-PSO system \eqref{PSO} with $f_0^{\otimes N}$-distributed initial data $\{(X_0^{i,N},V_0^{i,N})\}_{i=1}^N$. There exists a constant $C>0$ depending only on $\sigma,\gamma,\lambda,m,K,T$, and $\norm{\nabla\varphi}_\infty$ such that
	\begin{equation}
	\EE[|F_{\varphi}(f^N)|^2]\leq \frac{C}{N}\,,
	\end{equation}
	where $f^N=\frac{1}{N}\sum_{i=1}^N\delta_{(X^{i,N},V^{i,N})}$ is the empirical measure.
\end{lemma}
By Skorokhod's lemma (see \cite[Theorem 6.7 on page
70]{billingsley2013convergence}),   using Theorem \ref{thmtight} we may find a common probability space $(\Omega,\mc{F},\PP)$ on which the processes $\{f^N\}_{N\in\mathbb N}$ converge to some process $f$ as a random variable valued in $\mc{P}(\CR\times\CR)$ almost surely. In particular, we have that for all $t\in [0,T]$ and $\phi\in C_b(\RR^d\times \RR^d)$,
\begin{equation}\label{310}
\lim_{N\rightarrow \infty} |\la \phi,f_t^N-f_t\ra| + \left|X^\alpha(\rho^N_t)-X^\alpha(\rho_t)\right|= 0,\quad \text{a.s.}
\end{equation}
Indeed, according to Assumption \ref{asum}, one has  $xe^{-\alpha \TE(x)}, e^{-\alpha \TE(x)} \in \mc{C}_b(
\R^d)$, which gives
{\small \begin{align*}
\lim_{N\rightarrow \infty} X^\alpha(\rho_t^N) =	\lim_{N\rightarrow \infty} \frac{\la xe^{-\alpha\TE(x)},\rho_t^N(dx)\ra}{\la e^{-\alpha\TE(x)}, \rho_t^N(dx)\ra}=\frac{\la xe^{-\alpha\TE(x)},\rho_t(dx)\ra}{\la e^{-\alpha\TE(x)}, \rho_t(dx)\ra}=X^\alpha(\rho_t)\quad \text{a.s.} 
\end{align*}}
\begin{lemma}{\cite[Lemma 3.3]{carrillo2018analytical}}\label{lemXa}
	Let $\TE$ satisfy Assumption \ref{asum}  and $\mu\in \mc{P}_2(\R^d)$. Then  it holds that
	\begin{equation}
	|X^\alpha(\mu)|^2 \leq b_1+b_2\int_{\RR^d}|x|^2\mu(dx)\,,
	\end{equation}
	where $b_1$ and $b_2$ depends only on $M$, $C_u$, and $C_l$.
\end{lemma}
For each $A>0$, it follows from \eqref{310} that
\begin{align*}
&\EE\left[\iint_{\R^{2d}}((|x|^4+|v|^4)\wedge A)f_t(dx,dv)\right]\\
&\quad=\EE\left[\lim_{N\rightarrow \infty}\iint_{\R^{2d}}((|x|^4+|v|^4)\wedge A)f_t^N(dx,dv)\right]\\
&\quad\leq \lim_{N\rightarrow \infty}\frac{\sum_{i=1}^{N}\EE[|X_t^{i,N}|^4+|V_t^{i,N}|^4]}{N}\leq K\,,
\end{align*}
where we have used Lemma \ref{lem2mon}. Letting $A\rightarrow \infty$, we have 
\begin{align}
\sup_{t\in[0,T]}\EE\left[\iint_{\R^{2d}}(|x|^4+|v|^4)f_t(dx,dv)\right]\leq K. \label{bd-mu}
\end{align}
Then Lemma \ref{lemXa} implies that
\begin{equation}\label{Xa1}
\EE[|X^\alpha(\rho_t)|^4]<\infty\,,
\end{equation}
for all $t\in[0,T]$.
Furthermore, it holds that 
\begin{equation}\label{convergence}
\lim_{N\rightarrow \infty} 
\EE\left[
\left|\la \phi,f_t^N-f_t\ra \right|^2 + |X^\alpha(\rho^N_t)-X^\alpha(\rho_t)|^2
\right]=0,
\end{equation}
which follows directly from the pointwise convergences of $\la \phi,f_t^N-f_t\ra$ and $X^\alpha(\rho^N_t)-X^\alpha(\rho_t)$, and the uniform estimate \eqref{est-lem2mon} in Lemma \ref{lem2mon} and \eqref{Xa1}. 

We can now prove the main result in Theorem \ref{thmmean}:
\begin{proof}{\textbf{(Theorem \ref{thmmean})}}
	Suppose the $\mc{P}(\CR\times\CR)$-valued random variable $f$ is the limit of a subsequence of the empirical measure $f^N=\frac{1}{N}\sum_{i=1}^N\delta_{(X^{i,N},V^{i,N})}$. W.l.o.g., Denote the subsequence by itself. We may continue to work on the above common probability space $(\Omega,\mc{F},\PP)$ by Skorokhod's lemma where the convergence is holding almost surely (see \eqref{310} for instance). We may first check that  $f_t$ is a.s. continuous in time. Indeed for any $\phi\in\mc{C}_b(\R^{2d})$ and $t_n\to t$ we may apply dominated convergence theorem
	\begin{align*}
	&\iint_{\CR\times\CR}\phi(\textbf{x}_{t_n},\textbf{v}_{t_n})f(d\textbf{x},d\textbf{v})\nn\\
	&\quad\to \iint_{\CR\times \CR}\phi(\textbf{x}_{t},\textbf{v}_{t}) f(d\textbf{x},d\textbf{v})\quad \text{a.s.,}
	\end{align*}
	which gives
	\begin{equation*}
	\iint_{\R^{2d} }\phi(x,v)f_{t_n}(d{x},dv)\to \iint_{\R^{2d}}\phi(x,v)f_t(d{x},dv)\quad\text{a.s.}
	\end{equation*}
	For $\varphi\in \mc{C}_c^2(\RR^{2d})$, using the convergence result in \eqref{convergence} one has
	\begin{equation}\label{est1}
	\lim_{N\rightarrow \infty} \EE\left[|(\la \varphi ,f_t^N\ra-\la \varphi,f_0^N \ra) -(\la \varphi ,f_t \ra-\la \varphi,f_0\ra)|\right]=0\,.
	\end{equation}
	Further we notice that
      \begin{align*}
		&\left|\int_0^t\la (x-X^\alpha(\rho_s^N))\cdot \nabla_v\varphi, f_s^N \ra ds -\int_0^t\la (x-X^\alpha(\rho_s))\cdot \nabla_v\varphi, f_s \ra ds\right|\notag\\
		&\quad \leq \int_0^t\left|\la (x-X^\alpha(\rho_s^N))\cdot \nabla_v\varphi, f_s^N -f_s \ra \right| ds \nn\\
		&\qquad +\int_0^t\left|\la (X^\alpha(\rho_s)-X^\alpha(\rho_s^N))\cdot \nabla_v\varphi , f_s \ra\right| ds \notag\\
		&\quad =:\int_0^t|I_1^N(s)|ds+\int_0^t|I_2^N(s)|ds\,.
		\end{align*}
	One computes
	\begin{align*}
	&\EE[|I_1^N(s)|]\nn\\
	&\quad \leq \EE[|\la x\cdot \nabla_v \varphi,f_s^N -f_s \ra|]+\EE[|X^\alpha(\rho_s^N) \cdot \la\nabla_v \varphi,f_s^N-f_s \ra|]\notag\\
	&\quad \leq \EE[|\la x\cdot \nabla_v \varphi,f_s^N -f_s \ra|]+K^{\frac{1}{2}}(\EE[| \la\nabla_v \varphi,f_s^N -f_s \ra|^2])^\frac{1}{2}\,,
	\end{align*}
	where we have used Lemma \ref{lem2mon} in the second inequality.
	Since $\varphi$ has a compact support, applying \eqref{convergence} leads to 
	\begin{equation}
	\lim\limits_{N\to \infty }\EE\left[|I_1^N(s)|\right]=0\,.
	\end{equation}
	Moreover, the uniform boundedness of $\EE\left[|I_1^N(s)|\right]$ follows directly from \eqref{bd-mu}, \eqref{Xa1}, and the estimates in Lemma \ref{lem2mon}, which by the dominated convergence theorem implies
	\begin{equation}\label{I1}
	\lim\limits_{N\to \infty }\int _0^t\EE[|I_1^N(s)|]ds=0\,.
	\end{equation}
	As for  $I_2^N$, we know that
	\begin{equation}
	\left|\la (X^\alpha(\rho_s)-X^\alpha(\rho_s^N))\cdot \nabla_v\varphi, f_s \ra\right|\leq \norm{\nabla_v \varphi}_\infty  |X^\alpha(\rho_s)-X^\alpha(\rho_s^N)|\,.
	\end{equation}
	Hence by \eqref{convergence} it yields that
	\begin{equation}
	\lim\limits_{N\to \infty }\EE[|I_2^N(s)|]=0\,.
	\end{equation}
	Again by the dominated convergence theorem, we have
	\begin{equation}
	\lim\limits_{N\to \infty }\int _0^t\EE[|I_2^N(s)|]ds=0\,.
	\end{equation}
	This combined with \eqref{I1} leads to
{\small 	\begin{equation}
\begin{split}\label{est2}
	&\lim_{N\rightarrow \infty}	\EE\left[\left|\int_0^t\la (x-X^\alpha(\rho_s^N))\cdot \nabla_v\varphi, f_s^N \ra ds\right.\right.\\
	&\qquad\qquad\left.\left. -\int_0^t\la (x-X^\alpha(\rho_s))\cdot \nabla_v\varphi, f_s \ra ds\right|\right]=0\,.
	\end{split}\end{equation}}
	Similarly we split the error
	{\small \begin{align*}
		&\left|\int_0^t\la (x-X^\alpha(\rho_s^N))_k^2\frac{\partial^2}{\partial{v_k}^2}\varphi,f_s^N \ra ds -\int_0^t\la (x-X^\alpha(\rho_s))_k^2\frac{\partial^2}{\partial{v_k}^2}\varphi,f_s \ra ds \right|\notag\\
		&\quad \leq \left|\int_0^t\la (x-X^\alpha(\rho_s^N))_k^2\frac{\partial^2}{\partial{v_k}^2}\varphi,f_s^N-f_s\ra ds\right|\nn\\
		&\qquad +\left|\int_0^t\la ((x-X^\alpha(\rho_s^N))_k^2-(x-X^\alpha(\rho_s))_k^2)\frac{\partial^2}{\partial{v_k}^2}\varphi,f_s \ra ds \right| \notag\\
		&\quad =: \int_0^t |I_3^N(s)|ds+\int_0^t |I_4^N(s)|ds\,.
		\end{align*}}
	Following the same argument as for $I_1^N$ and $I_2^N$, one has
	\begin{equation}
	\lim\limits_{N\to \infty }\int_0^t \EE[|I_3^N(s)|]ds=0 \mbox{ and } \lim\limits_{N\to \infty }\int_0^t \EE[|I_4^N(s)|]ds=0\,.
	\end{equation}
	This implies that
	{\small 	\begin{equation}
	\label{est3}
	\begin{split}
		\lim_{N\rightarrow \infty}\EE\bigg[\bigg|&\int_0^t \sum_{k=1}^d\la (x-X^\alpha(\rho_s^N))_k^2\frac{\partial^2}{\partial{v_k}^2}\varphi(x),f_s^N \ra ds\\
		&-\int_0^t\sum_{k=1}^d\la (x-X^\alpha(\rho_s))_k^2\frac{\partial^2}{\partial{v_k}^2}\varphi(x),f_s \ra ds\bigg|\bigg]=0\,.
		\end{split}\end{equation}}
	Moreover it is easy to get
	\begin{align}\label{est4}
	\lim_{N\rightarrow \infty}\EE\left[\left|\int_0^t \la v \cdot \nabla_x\varphi, f_s^N\ra ds -\int_0^t \la v\cdot \nabla_x\varphi, f_s\ra ds\right|\right]=0
	\end{align}
	and
		\begin{align}\label{est5}
	\lim_{N\rightarrow \infty}\EE\left[\left|\int_0^t \la v \cdot \nabla_v\varphi, f_s^N\ra ds -\int_0^t \la v\cdot \nabla_v\varphi, f_s\ra ds\right|\right]=0\,.
	\end{align}	
	Collecting estimates \eqref{est1}, \eqref{est2}, \eqref{est3}, \eqref{est4} and \eqref{est5} we have
	\begin{equation}
	\lim_{N\rightarrow \infty}\EE[|F_{\varphi}(f^N) -F_{\varphi}(f)|]=0 \,.
	\end{equation}
	Then we have
	\begin{align*}
	\EE[|F_{\varphi}(f)|]&\leq \EE[|F_{\varphi}(f^N) -F_{\varphi}(f)|]+\EE[|F_{\varphi}(f^N)|]\nn\\
	&\leq \EE[|F_{\varphi}(f^N) -F_{\varphi}(f)|]+\frac{C}{\sqrt{N}}\to 0\quad \mbox{as }N\to\infty\,,
	\end{align*}
	where we have used Lemma \ref{prop} in the last inequality.
	This implies that
	\begin{equation}
	F_{\varphi}(f)=0\quad \text{a.s.}
	\end{equation}
	In other words, it holds that
{\small 	\begin{align*}
	&\la \varphi(x,v) ,f_t (dx,dv)\ra-\la \varphi(x,v),f_0(dx,dv) \ra +\int_0^t\la v\cdot\nabla_x\varphi,f_s(dx,dv)\ra ds
	\nn\\
	&\quad-\frac{\gamma}{m}\int_0^t\la v\cdot\nabla_v\varphi,f_s(dx,dv)\ra ds+\frac{\lambda}{m}\int_0^t\la (x-X^\alpha(\rho_s))\cdot \nabla_v\varphi, f_s (dx,dv)\ra ds\notag\\
	&\quad-\frac{\sigma^2}{2m^2}\int_0^t\sum_{k=1}^{d}\la (x-X^\alpha(\rho_s))_k^2\frac{\partial^2\varphi}{\partial v_k^2},f_s (dx,dv)\ra ds=0\,,
	\end{align*}}
	for any $\varphi\in \mc{C}_c^2(\RR^{2d})$. 
	
	Until now we have proved that $f$ a.s. is a weak solution to PDE \eqref{PDEi}. Finally combining the uniqueness of weak solution to \eqref{PDEi}  (see for example in \cite{BCC}) and the arbitrariness of the subsequence of $\{f^N\}_{N\geq 2}$,  the (deterministic) weak solution $f$ to PDE \eqref{PDEi} must be the limit of the whole sequence $\{f^N\}_{N\geq 2}$. We completed the proof.
	\end{proof}

\subsection{Regularized PSO dynamic with memory and local best}
Next, we consider the second order system of SDEs corresponding to the \emph{regularized SD-PSO method with local best} 
\coloredeq{eq:psocir}{
\begin{split}
dX^{i,N}_t &= V^{i,N}_t dt,\\
d\P_t^{i,N} &= \nu \left(X^{i,N}_t-\P^{i,N}_t\right)S^\beta(X^{i,N}_t,\P^{i,N}_t)dt,\\
\iw dV^{i,N}_t &= -\gamma V^{i,N}_tdt +\lambda_1\left(\P_t^{i,N}-X^{i,N}_t\right)dt \\
&\quad +\lambda_2\left(Y^\alpha(\bar \rho_t^N)-X^{i,N}_t\right)dt\\
&\quad +\sigma_1 D(\P_t^{i,N}-X^{i,N}_t)dB^{1,i}_t\\
&\quad +\sigma_2 D(Y^\alpha(\bar \rho_t^N)-X^{i,N}_t)dB^{2,i}_t,
\end{split}} 
where, similarly to the previous case, we  introduced the following \emph{regularized global best}  
\coloredeq{ValphaE2}{
{Y}^{\alpha}(\bar \rho_t^{N}):=\frac{\int_{\RR^d}y\omega_{\alpha}^{\mc{E}}(y)\bar\rho_t^{N}(dy)}{\int_{\RR^d}\omega_{\alpha}^{\mc{E}}(y)\rho_t^{N}(dy)},
}
with  the empirical measure $\bar \rho^N:=\frac{1}{N}\sum_{i=1}^{N}\delta_{Y^{i,N}}$, which is the $Y$-marginal of
$f^N=\frac{1}{N}\sum_{i=1}^{N}\delta_{(X^{i,N},Y^{i,N},V^{i,N})}$.

Furthermore, in the right hand side of \eqref{eq:psocir} we have replaced the $\sign(x)$ function with a \emph{sigmoid}, for example the hyperbolic tangent $\tanh(\beta x)$ for $\beta\gg 1$, and consider 
\be
S^\beta(x,y)=1+\tanh\left(\beta(\TE(y)-\TE(x))\right).
\ee 
Thanks to these regularizations, also the stochastic particle system \eqref{eq:psocir} has locally Lipschitz coefficients and therefore it admits strong solutions and pathwise uniqueness holds for any finite time $T>0$. Even in this case, the system of SDEs \eqref{eq:psocir} is generalized without restricting the search parameters to the PSO constraint \eqref{eq:param}.  


In order to derive a mean-field description of system \eqref{eq:psocir}, we can follow the same arguments as in Section \ref{meanPSO1}. The only difference is that we have an additional variable $Y$, which can be treated easily because of the regularity of the function $S^\beta$.
 Namely we can prove the tightness of the empirical measures $\{f^N\}_{N\geq2}$ by verifying the Aldous criteria ( Lemma \ref{lemAldous}). Then there exists a subsequence of $\{f^N\}_{N\geq 2}$ converging in law to a deterministic measure $f\in\mc{P}(\CR\times\CR\times\CR)$, which is the unique weak solution to the following \emph{mean-field PSO system with local best} characterized by the nonlinear Vlasov-Fokker-Planck equation
 \coloredeq{PDEii}{
\begin{split}
&\partial_t f + v \cdot \nabla_x f + \nabla_y \cdot \left(\nu(x-y)S^\beta(x,y)f\right)= 
\\
&\qquad \nabla_v\cdot\left(\frac{\gamma}{\iw} v f + \frac{\lambda_1}{\iw} (x - y)f
+ \frac{\lambda_2}{\iw} (x - Y^\alpha(\bar\rho) )f\right.\\
&\qquad \left.+\left(\frac{\sigma_2^2}{2{\iw^2}}D(x - Y^\alpha(\bar\rho))^2+\frac{\sigma_1^2}{2{\iw}^2}D(x - y)^2\right)\nabla_v f\right)\,,
\end{split}}
where $\bar\rho(t,y)=\int_{ \RR^{2d}}f(t,dx,y,dv)$.

This can be summarized in the following theorem
\begin{theo-frmd}
	Let $\TE$ satisfy Assumption \ref{asum} and $f_0\in \mc{P}_4(\RR^{3d})$. For any $N\geq 2$, we assume that $\{(X_t^{i,N},Y_t^{i,N},V_t^{i,N})_{t\in[0,T]}\}_{i=1}^N$ is the unique solution to the SD-PSO system \eqref{eq:psocir} with $f_0^{\otimes N}$-distributed initial data $\{X_0^{i,N},Y_0^{i,N},V_0^{i,N}\}_{i=1}^N$.  Then the limit (denoted by $f$) of the sequence of the empirical measure $f^N=\frac{1}{N}\sum_{i=1}^N\delta_{(X^{i,N},Y^{i,N},V^{i,N})}$ exists. Moreover, $f$ is the unique weak solution to MF-PSO equation \eqref{PDEii}.
\end{theo-frmd}

\section{Zero-inertia limit and consensus-based optimization}
In this section we consider the asymptotic behavior of the previous Vlasov-Fokker-Planck equations modelling the PSO dynamic in the small inertia limit, i.e. $m\to 0$. We will derive the corresponding macroscopic equations which permit to recover and generalize the recently introduced consensus-based optimization (CBO) methods \cite{carrillo2019consensus}. We refer to \cite{DLP,choi2020quantified} for a theoretical background concerning the related problem of the overdamped limit of nonlinear Vlasov-Fokker-Planck systems.
 

\subsection{The case without memory effects}
Let us first consider the simplified setting in absence of local best. 
Now we write down the  so called McKean-Vlasov process \cite{mckean1966class} underlying PSO equation \eqref{PDEi}, which is of the form\footnote{We used the superscript $m$ to emphasize its dependence on the inertia coefficient $m$.}
{ \begin{subequations}\label{MVeq}
	\begin{eqnarray}
	d\OX_t^m&=& \OV_t^mdt\,, \label{eqX}
	\\
	\nn
	d\OV_t^m&=&-\frac{\gamma}{m}\OV_t^mdt+\frac{\lambda}{m}(X^{\alpha}(\rho_t^m)-\OX_t^m)dt\\[-.2cm]
	\label{eqV} 
	\\[-.2cm]\nn
	&& +\frac{\sigma}{m}D(X^{\alpha}(\rho_t^m)-\OX_t^m)dB_t\,, 
	\end{eqnarray}
\end{subequations}}
where 
\begin{equation}\label{14}
X^{\alpha}(\rho_t^m)=\frac{\int_{\RR^d}x\omega_{\alpha}^{\mc{E}}(x)\rho^m(t,dx)}{\int_{\RR^d}\omega_{\alpha}^{\mc{E}}(x)\rho^m(t,dx)}, \quad \rho^m(t,x)=\int_{\RR^d}f^m(t,x,dv)\,,
\end{equation}
and the initial data $(\OX_0,\OV_0)$ is the same as in \eqref{PSO}. Here  $f^m(t,x,v)$ is  the distribution of $(\OX_t^m,\OV_t^m)$ at time $t$ , which makes the set of equations \eqref{MVeq} nonlinear.  A direct application of the It\^{o}-Doeblin formula yields that the law $f_t^m:=f^m(t,\cdot,\cdot)$ at time $t$ is a weak solution to \eqref{PDEi}.

To illustrate the limiting procedure, let us observe that for $m \to 0^+$ from the equation \eqref{eqV} we formally have 
\[
\OV^0_tdt = \lambda\left(X^{\alpha}(\rho_t^0)-\OX^0_t\right)dt+\sigma D(X^{\alpha}(\rho_t^0)-\OX^0_t)dB_t,
\] 
where we used the fact that $\gamma = 1 - m \to 1$. Substituting the above identity into the  equation \eqref{eqX} and omitting the superscripts gives the first order CBO system \cite{carrillo2019consensus}
\coloredeq{MVCBO}{
d\OX_t=\lambda(X^{\alpha}(\rho_t)-\OX_t)dt +\sigma D(X^{\alpha}(\rho_t)-\OX_t)dB_t\,.
}
Therefore, the CBO models based on a multiplicative noise can be understood as reduced order approximations of SD-PSO dynamics. 

\subsubsection{Formal derivation in the mean-field case}
In the sequel we will develop these arguments in the case of the nonlinear Vlasov-Fokker-Planck equation \eqref{PDEi} describing the evolution of the distribution of \eqref{MVeq}. 
We re-write the scaled Vlasov-Fokker-Planck system in the form
\begin{equation}\label{PDEis}
\begin{split}
\partial_t f + v \cdot \nabla_x f + 
\frac1{m}\nabla_v\cdot\left(m  v f + \lambda (X^{\alpha}(\rho)-x)f\right)= L_{m}(f)
\end{split}
\end{equation}
where we used the fact that $\gamma=1-m$ and define
\[
\begin{split}
L_{m}(f)&=\frac1{m}\nabla_v\cdot\left(v f + \frac{\sigma^2}{2m} D(x - X^{\alpha}(\rho) )^2\nabla_v f\right)\\
&= \frac1{m} \sum_{j=1}^d \frac{\sigma^2}{2} (x_j - X^{\alpha}_j(\rho))^2 \frac{\partial}{\partial v_j}\left(\frac{2 f v_j}{\sigma^2(x_j - X^{\alpha}_j(\rho))^2} + \frac1{m}\frac{\partial f}{\partial v_j}\right).
\end{split}
\]
{Note that the last equality is a direct consequence of identity \eqref{eq:newid}}.
Let us now introduce the local Maxwellian with unitary mass and zero momentum
\[
\begin{split}
{\mathcal M}_m(x,v,t)&= \prod_{j=1}^d M_{m}(x_j,v_j,t), \\
M_{m}(x_j,v_j,t) &= \frac{m^{1/2}}{\pi^{1/2}\sigma |x_j-X^\alpha_j(\rho)|} 
\exp\left\{-\frac{m v_j^2}{\sigma^2(x_j-X^\alpha_j(\rho))^2}\right\},
\end{split}
\]
then we have
\[
L_{m}(f) = \frac1{m^2 }\sum_{j=1}^d \frac{\sigma^2}{2} (x_j - X^{\alpha}_j(\rho))^2 \frac{\partial}{\partial v_j}\left(f\frac{\partial}{\partial v_j}\log\left(\frac{f}{M_{m}(x_j,v_j,t)}\right)\right).
\]
Therefore $L_{m}(f)$ is of order $1/m^2 $ and we can write for small values of $m \ll 1$
\begin{equation}
f(x,v,t)=\rho(x,t){\mathcal M}_m(x,v,t).
\label{eq:maxw1}
\end{equation}
{Let us now integrate equation \eqref{PDEis} with respect to $v$, and multiply the same equation by $v$ and ingrate again with respect to $v$,
we get}
\[
\begin{split}
\frac{\partial \rho}{\partial t} &+ \nabla_x \cdot (\rho u) = 0\\
\\[-.5cm] \nn
\frac{\partial \rho u}{\partial t} &+ \int_{\RR^d} v\left(v\cdot \nabla_x f\right)\,dv = -\frac{\gamma}{m} \rho u + \frac1{m} \lambda (X^{\alpha}(\rho)-x) \rho
\end{split}
\]
where
\[
\rho u = \int_{\RR^d} f(x,v,t) v\,dv.
\]
{Now assuming \eqref{eq:maxw1} we can compute for
$m \ll 1$ the $j$-th component of the second term in the right hand side of
last equation as}
\[{\small
\begin{split}
\int_{\RR^d} v_j\left(v\cdot \nabla_x \left(\rho(x,t){\mathcal M}_m(x,v,t)\right) \right)\,dv &=  \sum_{j=1}^d \frac{\partial}{\partial x_j} \left(\rho(x,t) \int_{\RR^d} v_j (v_j {\mathcal M}_m(x,v,t))\,dv\right)\\
&=  \frac{\partial}{\partial x_j} \left(\rho(x,t) \int_{\RR} v^2_j {M}_m(x_j,v_j,t)\,dv_j\right)\\
&= \frac{\sigma^2}{2m}\frac{\partial}{\partial x_j} \left(\rho(x,t) (x_j-X^\alpha_j(\rho))^2\right)
\end{split}}
\]
which provides the \emph{macroscopic PSO system without local best}
\coloredeq{eq:macro}{
\begin{split}
\frac{\partial \rho}{\partial t} &+ \nabla_x \cdot (\rho u) = 0,\\
\\[-.5cm]
\frac{\partial (\rho u)_j}{\partial t} &+ \frac{\sigma^2}{2m} \frac{\partial}{\partial x_j} \left(\rho(x,t) (x_j-X^\alpha_j(\rho))^2\right) =\\
&= \displaystyle -\frac{1-m}{m} (\rho u)_j + \frac1{m} \lambda (X_j^{\alpha}(\rho)-x_j) \rho.
\end{split}}
\\
Formally, as $m\to 0^+$, from the second equation in \eqref{eq:macro} we get 
\[
(\rho u)_j = \lambda (X_j^{\alpha}(\rho)-x_j) \rho -\frac{\sigma^2}{2} \frac{\partial}{\partial x_j} \left(\rho(x,t) (x_j-X^\alpha_j(\rho))^2\right),
\]
which substituted in the first equation yields the \emph{mean-field CBO system} \cite{carrillo2019consensus}
\coloredeq{eq:CBOp}{
\frac{\partial \rho}{\partial t} + \nabla_x \cdot \lambda (X^{\alpha}(\rho)-x) \rho = \frac{\sigma^2}{2}\sum_{j=1}^d \frac{\partial^2}{\partial x^2_j} \left(\rho(x,t) (x_j-X^\alpha_j(\rho))^2\right).
}
Therefore, in the small inertia limit we expect the macroscopic density in the PSO system \eqref{PDEi} to be well approximated by the solution of the CBO equation \eqref{eq:CBOp}. We emphasize that system \eqref{eq:macro} represents a novel mean-field optimization model with an intermediate level of description between the mean-field PSO system \eqref{PDEis} and the mean-field CBO system \eqref{eq:CBOp}.

\subsubsection{Rigorous derivation}
\label{seczero}
In this section, we present a rigorous derivation of the zero-inertia limit \cite{cipriani2021zero}. More precisely we prove that as $m\to 0^+$, the processes $\{\OX^m\}$ satisfying the SDEs \eqref{MVeq}  converge weakly to the solution  $\OX$ to the SDE \eqref{MVCBO} in the continuous path space $\mc{C}([0,T];\RR^d)$, and a convergence rate is obtained.  The main theorem can be stated as below:
\begin{theo-frmd}\label{thmzero-limit} 
	Let Assumption \ref{asum} hold and $(X_t^m,V_t^m)_{t\in[0,T]}$ satisfy the system \eqref{MVeq}. Then as $m\rightarrow 0^+$, the sequence of  stochastic processes $\{\OX^m\}_{0< m\leq \frac{1}{2}}$  converge weakly to $\overline X$, which is the unique solution to the following SDE:
	\begin{equation}\label{MVCBO1}
	\begin{split}
	\OX_t= &\ \OX_0
	+\lambda\int_0^t(X^{\alpha}(\rho_s)-\OX_s)ds \\
	&+\sigma\int_0^t D(X^{\alpha}(\rho_s)-\OX_s)dB_s\,.
	\end{split}
	\end{equation}
	Moreover it holds that
	\begin{equation}\label{est-convegence-m}
	\sup_{t\in[0,T]}\EE[|\OX^m_t-\OX_t|^2]\leq C \, m\,,
	\end{equation}
	where the constant $C$ depends only on $\EE[|\OX_0|^4]$, $\EE[|\OV_0|^4]$, $M$, $C_u$, $C_l$, $\lambda$, $\sigma$, $d$, and $T$.
\end{theo-frmd}
\begin{remark}
	It follows from the definition of Wasserstein distance that
	\begin{equation}
	\sup_{t\in[0,T]}W_2^2(\rho^m_t,\rho_t)\leq  \sup_{t\in[0,T]}\EE[|\OX^m_t-\OX_t|^2]\leq C\, m \,,
	\end{equation}
	which in a way is consistent with the result obtained in \cite[Theorem 1.3]{choi2020quantified}, where the authors obtained a quantified overdamped limit  (with the same rate $m$) of the singular Vlasov-Poisson-Fokker-Planck system to the aggregation-diffusion equation. 
\end{remark}
The following theorem gives the well-posedness of the mean-field PSO dynamic  \eqref{MVeq} whose proof is  analogous to \cite[Theorem 2.3]{huang2021mean} or \cite[Theorem 3.1]{carrillo2018analytical}, and thus omitted.
\begin{theorem}\label{thm-huang2021mean}
	Let Assumption \ref{asum} hold.  If $(\OX^m_0,\OV^m_0)=(\OX_0,\OV_0)$ is distributed according to $f_0$ with $f_0\in\mc{P}_4(\RR^{2d})$, then for each $T>0$and $m\in(0,1]$, the nonlinear SDE \eqref{MVeq} admits a unique strong solution up to time $T$ with the initial data $(\OX^m_0,\OV^m_0)$ and it holds further that
	\begin{equation}\label{secmen}
	\sup\limits_{t\in[0,T]}\EE\left[|\OX_t^m|^4+|\OV_t^m|^4\right]\leq e^{CT} \cdot \EE\left[ |\OX_0|^4+|\OV_0|^4\right]\,,
	\end{equation}
	where $C$ depends only on $\lambda,m,\sigma$, $M$, $C_u$, and $C_l$.
\end{theorem}
Solving \eqref{eqV} for $\OV_t^m$ gives  
\begin{align*}
\OV_t^m&=e^{-\frac{\gamma}{m}t}\OV_0+\frac{\lambda}{m}\int_0^te^{-\frac{\gamma}{m}(t-s)}(X^{\alpha}(\rho_s^m)-\OX_s^m)ds\\
&\quad +\frac{\sigma}{m}\int_0^te^{-\frac{\gamma}{m}(t-s)}D(X^{\alpha}(\rho_s^m)-\OX_s^m)dB_s\,,
\end{align*}
which implies that
\begin{equation}
\label{onlyX}
\begin{split}
\OX_t^m&=\OX_0+\int_0^t\OV_\tau d\tau=\OX_0+\int_0^t  e^{-\frac{\gamma}{m}\tau}\OV_0d\tau\\
&\quad +\frac{\lambda}{m}\int_0^t\int_0^\tau e^{-\frac{\gamma}{m}(\tau-s)}(X^{\alpha}(\rho_s^m)-\OX_s^m)dsd\tau\\
& \quad +\frac{\sigma}{m}\int_0^t\int_0^\tau e^{-\frac{\gamma}{m}(\tau-s)}D(X^{\alpha}(\rho_s^m)-\OX_s^m)dB_s d\tau\,.
\end{split}
\end{equation}
Then $\OX_t^m$ has the law $\rho^m_t$ for each $t\geq 0$. 

Each continuous stochastic process $\OX^m$ may be seen as a $\mc{C}([0,T];\RR^d)$-valued random function and it induces a probability measure (or law, denoted by $\rho^m$) on $\mc{C}([0,T];\RR^d)$. We shall use the weak convergence in the space of probability measures on $\mc{C}([0,T];\RR^d)$. In what follows,  we write $\OX^m \rightharpoonup \OX$ or $\rho^m \rightharpoonup \rho$ with $\rho$ being the law of $\OX$,  if $\left\{\rho^m\right\}_{m>0}$, as a sequence of probability measures,  converges weakly to $ \rho$, i.e.,
for each bounded continuous functional $\Phi$ on $\mc{C}([0,T];\RR^d)$ , there holds $\lim_{m\rightarrow 0^+}\EE\left[\Phi(\OX^m)\right]= \EE\left[\Phi(\OX)\right]$. The weak convergence $\OX^m \rightharpoonup \OX$ is stronger than and actually implies the convergence of $\{\rho^m_t \}_{m>0}$ to $ \rho_t$ with $\rho_t$ being the law of $\OX_t$ for each $t\geq 0$, while the converse need not hold. Moreover, due to the separability and completeness of the space $\mc{C}([0,T];\RR^d)$,  Prohorov's theorem implies that the relative compactness is equivalent to the tightness; see \cite{billingsley2013convergence} for more details.


\begin{theorem}\label{thmtight1}
	Let Assumption \ref{asum} hold and $(X_t^m,V_t^m)_{t\in[0,T]}$ satisfy the system \eqref{MVeq}. For each countable subsequence $\{m_k\}_{k\in\mathbb N} \subset [0,\frac{1}{2}]$ with $\lim_{k\rightarrow \infty}m_k=0$, the sequence of probability distributions $\{\rho^{m_k}\}_{k\in\mathbb N}$ of $\{\OX^{m_k}\}_{k\in\mathbb N}$  is tight.
\end{theorem}
\begin{proof}
	By Lemma \ref{lemAldous}, it is sufficient to justify conditions $(Con 1)$ and $(Con 2)$ in Aldous tightness criteria .
	
	$\bullet$ \textit{Step 1: Checking $(Con 1)$. }  
	First, for $0< m\leq \frac{1}{2}$, recalling \eqref{onlyX}, we have by Fubini's theorem (see \cite[Theorem 4.33]{da2014stochastic} for the stochastic version)
	\begin{align}\label{Fubini}
	\OX_t^m&=\OX_0+\int_0^t  e^{-\frac{\gamma}{m}\tau}\OV_0d\tau
	+\frac{\lambda}{m}\int_0^t\int_0^\tau e^{-\frac{\gamma}{m}(\tau-s)}(X^{\alpha}(\rho_s^m)-\OX_s^m)dsd\tau\notag\\
	& \quad +\frac{\sigma}{m}\int_0^t\int_0^\tau e^{-\frac{\gamma}{m}(\tau-s)}D(X^{\alpha}(\rho_s^m)-\OX_s^m)dB_s d\tau \notag\\
	&=\OX_0+\int_0^t  e^{-\frac{\gamma}{m}\tau}\OV_0d\tau
	+\frac{\lambda}{m}\int_0^t\int_s^t e^{-\frac{\gamma}{m}(\tau-s)}d\tau (X^{\alpha}(\rho_s^m)-\OX_s^m)ds\notag\\
	& \quad +\frac{\sigma}{m}\int_0^t\int_s^t e^{-\frac{\gamma}{m}(\tau-s)}d\tau D(X^{\alpha}(\rho_s^m)-\OX_s^m)dB_s \notag\\
	&=\OX_0+\frac{m}{\gamma}(1-e^{-\frac{\gamma}{m}t})\OV_0
	+\frac{\lambda}{\gamma}\int_0^t(1-e^{-\frac{\gamma}{m}(t-s)}) (X^{\alpha}(\rho_s^m)-\OX_s^m)ds\notag\\[-.2cm] 
	\\[-.25cm]\nn
	& \quad +\frac{\sigma}{\gamma}\int_0^t(1-e^{-\frac{\gamma}{m}(t-s)}) D(X^{\alpha}(\rho_s^m)-\OX_s^m)dB_s\,.
	\end{align}
	Note here the assumption on $0< m\leq \frac{1}{2}$ ensures that $\gamma=1-m\in[\frac{1}{2},1)$, so $\frac{1}{\gamma}$ is well defined.
	It follows from H\"{o}lder's inequality that
		\begin{align}\label{Xes}
		|\OX_t^m|^4&\leq 64|\OX_0|^4+\frac{64m^4}{\gamma^4}|\OV_0|^4+\frac{64\lambda^4t^3}{\gamma^4}\int_0^t|X^{\alpha}(\rho_s^m)-\OX_s^m|^4ds \nn \\
		&\quad+\frac{64\sigma^4}{\gamma^4}\left|\int_0^{t}(1-e^{-\frac{\gamma}{m}(t-s)}) D(X^{\alpha}(\rho_s^m)-\OX_s^m)dB_s\right|^4,
		\end{align}
	where we have used the fact that for any sequence $\{a_i\}_{i=1}^n\geq 0$ and $p\geq 2$, there holds
	\begin{equation*}
	\left(\sum_{i=1}^{n}a_i\right)^p\leq n^{p-1}\sum_{i=1}^{n}a_i^p\,.
	\end{equation*}
	Using the moment inequality for stochastic integrals as in \cite[Theorem 7.1]{mao2007stochastic}  yields that
	\begin{align*}
	&\EE\left[\left|\int_0^{t}(1-e^{-\frac{\gamma}{m}(t-s)}) D(X^{\alpha}(\rho_s^m)-\OX_s^m)dB_s\right|^4\right]\nn\\
	&\quad \leq d^3\EE\left[\sum_{k=1}^{d}\left|\int_0^{t}(1-e^{-\frac{\gamma}{m}(t-s)}) (X^{\alpha}(\rho_s^m)-\OX_s^m)_kdB_s^ke_k\right|^4\right]\nn\\
	&\quad \leq 36d^3t\int_0^{t}\EE\left[  \sum_{k=1}^{d}|(X^{\alpha}(\rho_s^m)-\OX_s^m)_k|^4\right] ds\nn\\
	&\quad \leq 36d^3t\int_0^{t}\EE\left[  |X^{\alpha}(\rho_s^m)-\OX_s^m|^4\right] ds\,.
	\end{align*}
	Thus,
	\begin{align*}
	\EE[|\OX_t^m|^4]&\leq 64\EE[|\OX_0|^4]+\frac{64m^4}{\gamma^4}\EE[|\OV_0|^4]\\
	&\quad+\frac{64(\lambda^4t^3+36d^3t\sigma^4)}{\gamma^4}\int_0^t\EE[|X^{\alpha}(\rho_s^m)-\OX_s^m|^4]ds\,.
	\end{align*}
	Notice that
	\begin{equation}
	\label{Jensen}
	\begin{split}
	&\EE[|X^{\alpha}(\rho_t^m)-\OX_t^m|^4]\leq 8 |X^{\alpha}(\rho_t^m)|^4+8\EE[|\OX_t^m|^4]\\
	&\quad \leq 8(b_1+b_2\EE[|\OX_t^m|^2])^2+8\EE[|\OX_t^m|^4]\\
	&\quad \leq c_1+c_2\EE[|\OX_t^m|^4]
	\,,
	\end{split}
	\end{equation}
	where we have used Lemma \ref{lemXa}  in the second inequality, and $c_1,c_2$ depend only on $C_u$, $M$ and $C_l$. Thus we have
	\begin{align*}
	\EE[|\OX_t^m|^4]
	&\leq 64\EE[|\OX_0|^4]+\frac{64m^4}{\gamma^4}\EE[|\OV_0|^4]+c_3\nn\\
	&\quad+\frac{64c_2(\lambda^4t^3+36d^3t\sigma^4)}{\gamma^4}\int_0^t\EE[|\OX_s^m|^4]ds\,.
	\end{align*}
	Using Gronwall's inequality leads to
{\small 	\begin{equation}\label{esOX2}
\begin{split}
	\EE[|\OX_t^m|^4] \leq &\left(64\EE[|\OX_0|^4]+\frac{64m^4}{\gamma^4}\EE[|\OV_0|^4]+c_3\right)\cdot\\
	&\quad\cdot\exp\left(\frac{64c_2(\lambda^4T^3+36d^3T\sigma^4)}{\gamma^4}T\right),
	\end{split}
	\end{equation}}
for all $t\in[0,T]$.
	Recalling  $0\leq m\leq \frac{1}{2}$ and $\frac{1}{\gamma}=\frac{1}{1-m}\leq 2$, from estimate \eqref{esOX2} we obtain the boundedness:
	\begin{equation}\label{Sm}
	\EE[|\OX_t^m|^4] \leq C(\EE[|\OX_0|^4],\EE[|\OV_0|^4],M,C_u,C_l,\lambda,d,\sigma,T)\,.
	\end{equation}
	This yields that
	\begin{equation} \label{uniform-bd}
	\begin{split}
	&\sup_{m\in(0,1]} \sup_{t\in[0,T]}\EE[|\OX_t^m|^4] \\
	&\qquad \leq C(\EE[|\OX_0|^4],\EE[|\OV_0|^4],M,C_u,C_l,\lambda,\sigma,d,T)=:C_1
	\end{split}
	\end{equation}
	where the constant $C_1>0$ is independent of $m$. Therefore, for any $\varepsilon>0$, there exists a compact subset $K_\varepsilon:=\{x:~|x|^4\leq \frac{C_1}{\varepsilon}\}$ such that by Markov's inequality
	\begin{equation}
	\rho_t^m((K_\varepsilon)^c)=\PP(|X_t^m|^4> \frac{C_1}{\varepsilon})\leq \frac{\varepsilon\EE[|X_t^m|^4]}{C_1}\leq\varepsilon,\quad \forall ~0<m\leq 1\,.
	\end{equation}
	This means that for each $t\in[0,T]$, each countable subset of $\{\rho_t^m\}_{0<m\leq 1}$ is tight, which verifies  condition $(Con 1)$ in Lemma \ref{lemAldous}.
	\\ \\
	$\bullet$ \textit{Step 2: Checking $(Con 2)$. }   Let $\beta$ be a $\sigma(X^m_s;s\in[0,T])$-stopping time with discrete values such that $\beta+\delta_0\leq T$. Without any loss of generality, we may assume that the concerned countable subsequence $\{m_k\}_{k\in\mathbb N} \subset [0,1]$ satisfies $m_k\leq \frac{1}{2}$ for all $k\in\NN$; thus, we may just consider the case of $0<m\leq \frac{1}{2}$ which indicates $\frac{1}{2}\leq \gamma<1$. 
	Recall \eqref{onlyX} and compute
	\begin{equation*}
		\begin{split}
		\OX_{\beta+\delta}^m-\OX_{\beta}^m =&\int_\beta^{\beta+\delta}\OV_\tau d\tau = \int_\beta^{\beta+\delta}  e^{-\frac{\gamma}{m}\tau}\OV_0d\tau \\
		&+\frac{\lambda}{m}\int_\beta^{\beta+\delta}\int_0^\tau e^{-\frac{\gamma}{m}(\tau-s)}(X^{\alpha}(\rho_s^m)-\OX_s^m)dsd\tau
		\\
		&  +\frac{\sigma}{m}\int_\beta^{\beta+\delta}\int_0^\tau e^{-\frac{\gamma}{m}(\tau-s)}D(X^{\alpha}(\rho_s^m)-\OX_s^m)dB_s d\tau 		\end{split}
	\end{equation*}
	\begin{equation*}
		\begin{split}
		=&\int_\beta^{\beta+\delta}  e^{-\frac{\gamma}{m}\tau}\OV_0d\tau\\
		&+\frac{\lambda}{m}\int_0^\beta\int_\beta^{\beta+\delta} e^{-\frac{\gamma}{m}(\tau-s)}d\tau (X^{\alpha}(\rho_s^m)-\OX_s^m)ds\\
		&+\frac{\lambda}{m}\int_\beta^{\beta+\delta}\int_s^{\beta+\delta} e^{-\frac{\gamma}{m}(\tau-s)}d\tau (X^{\alpha}(\rho_s^m)-\OX_s^m)ds\\
		& +\frac{\sigma}{m}\int_0^\beta\int_\beta^{\beta+\delta} e^{-\frac{\gamma}{m}(\tau-s)}d\tau D(X^{\alpha}(\rho_s^m)-\OX_s^m)dB_s \\
		&+\frac{\sigma}{m}\int_\beta^{\beta+\delta}\int_s^{\beta+\delta} e^{-\frac{\gamma}{m}(\tau-s)}d\tau D(X^{\alpha}(\rho_s^m)-\OX_s^m)dB_s \,.
		\end{split}
	\end{equation*}
	Then it yields
	\begin{equation}\label{diff}
	\begin{split}
		& \OX_{\beta+\delta}^m-\OX_{\beta}^m	=\frac{m}{\gamma}(e^{-\frac{\gamma}{m}\beta}-e^{-\frac{\gamma}{m}(\beta+\delta)})\OV_0 \\
		&+\frac{\lambda}{\gamma}\int_0^\beta (e^{-\frac{\gamma}{m}(\beta-s)}-e^{-\frac{\gamma}{m}(\beta+\delta-s)}) (X^{\alpha}(\rho_s^m)-\OX_s^m)ds\\
		&+\frac{\lambda}{\gamma}\int_\beta^{\beta+\delta} (1-e^{-\frac{\gamma}{m}(\beta+\delta-s)}) (X^{\alpha}(\rho_s^m)-\OX_s^m)ds\\
		&+\frac{\sigma}{\gamma}\int_0^\beta (e^{-\frac{\gamma}{m}(\beta-s)}-e^{-\frac{\gamma}{m}(\beta+\delta-s)})D(X^{\alpha}(\rho_s^m)-\OX_s^m)dB_s\\
		&+\frac{\sigma}{\gamma}\int_\beta^{\beta+\delta} (1-e^{-\frac{\gamma}{m}(\beta+\delta-s)})D(X^{\alpha}(\rho_s^m)-\OX_s^m)dB_s\,.
		\end{split}
		\end{equation}	
	Note that there holds $|e^{-x}-e^{-y}|\leq |x-y|\wedge 1 $ for all $x,y\in[0,\infty)$. Basic computations further indicate that for each $q\geq 1$ and $\tau\in[0,T]$,
	\begin{align*}
	\int_0^{\tau} &
	\left| e^{-\frac{\gamma(\tau-s)}{m}} - e^{-\frac{\gamma(\tau+\delta-s)}{m}}   \right|^q \,ds
	\leq
	\int_0^{\tau}
	\left( e^{-\frac{\gamma(\tau-s)}{m}} - e^{-\frac{\gamma(\tau+\delta-s)}{m}}  \right) \,ds\nn\\
	&=
	\frac{m}{\gamma} \left( 1- e^{-\frac{\gamma \delta}{m}}  \right)
	-\frac{m}{\gamma} \left( e^{-\frac{\gamma \tau}{m}} -e^{-\frac{\gamma (\tau+\delta)}{m}} \right)
	\\
	& \leq  \frac{m}{\gamma} \cdot \frac{\gamma\delta}{m} 
	=\delta,
	\end{align*}
	and in particular, 
	$$
	\int_{\beta}^{\beta+\delta} \left(1-e^{-\frac{\gamma(\beta+\delta-s)}{m}}\right)^q ds \leq \int_{\beta}^{\beta+\delta} 1 \,ds =\delta.
	$$
	Then, it is obvious that
	{\small \begin{align*}
		\EE\left[\left|\frac{m}{\gamma}(e^{-\frac{\gamma}{m}\beta}-e^{-\frac{\gamma}{m}(\beta+\delta)})\OV_0\right|^2\right]
		\leq \frac{m^2}{\gamma^2}\cdot \frac{\gamma^2\delta^2}{m^2}
		\left(\EE[|\OV_0|^4]\right)^{\frac{1}{2}}
		\leq \delta^2\left(\EE[|\OV_0|^4]\right)^{\frac{1}{2}}.
		\end{align*}}
	Next, it follows that
	{\small \begin{align*}
		&\EE\left[\left| \int_0^\beta (e^{-\frac{\gamma}{m}(\beta-s)}-e^{-\frac{\gamma}{m}(\beta+\delta-s)}) (X^{\alpha}(\rho_s^m)-\OX_s^m)ds\right|^2\right]
		\notag\\
		&\quad \leq\   
		 \EE\left[\int_0^{\beta} |e^{-\frac{\gamma}{m}(\beta-s)}-e^{-\frac{\gamma}{m}(\beta+\delta-s)}|^2 ds\cdot \int_0^{\beta} |X^{\alpha}(\rho_s^m)-\OX_s^m|^2ds\right]\notag\\
		&\quad\leq \ 
		\delta \cdot T \sup_{s\in[0,T]}  \left(\EE\left[    |X^{\alpha}(\rho_s^m)-\OX_s^m|^4 \right] \right)^{1/2},
		\notag
		\end{align*}}
	and analogously,
	{\small \begin{align*}
		&\EE\left[\left| \int_\beta^{\beta+\delta} (1-e^{-\frac{\gamma}{m}(\beta+\delta-s)}) (X^{\alpha}(\rho_s^m)-\OX_s^m)ds \right|^2\right] \nn\\
		&\quad\leq \ 
		\EE\left[
		\int_{\beta}^{\beta+\delta} \left(1-e^{-\frac{\gamma(\beta+\delta-s)}{m}}\right)^2 ds
		\cdot \int_\beta^{\beta+\delta} |X^{\alpha}(\rho_s^m)-\OX_s^m|^2ds
		\right]
		\\
		&\quad\leq \    \delta\cdot
		\EE\left[\int_\beta^{\beta+\delta} |X^{\alpha}(\rho_s^m)-\OX_s^m|^2ds\right] \\
		 &\quad\leq \  \delta \cdot T \sup_{s\in[0,T]}  \left(\EE\left[    |X^{\alpha}(\rho_s^m)-\OX_s^m|^4 \right]\right)^{1/2}\,.
		\end{align*}}
	Further, applying It\^{o}'s isometry  gives
{\small 	\begin{align}
	&\EE\left[\left|\int_0^\beta (e^{-\frac{\gamma}{m}(\beta-s)}-e^{-\frac{\gamma}{m}(\beta+\delta-s)})D(X^{\alpha}(\rho_s^m)-\OX_s^m)d B_s\right|^2\right]\notag\\
	&\quad\leq \  d  \EE\left[\int_0^{\beta} |e^{-\frac{\gamma}{m}(\beta-s)}-e^{-\frac{\gamma}{m}(\beta+\delta-s)}|^2|X^{\alpha}(\rho_s^m)-\OX_s^m|^2ds\right]
	\notag\\
	&\quad\leq \  d\left(\EE\left[\int_0^{\beta} |e^{-\frac{\gamma}{m}(\beta-s)}-e^{-\frac{\gamma}{m}(\beta+\delta-s)}|^4 ds \right] \right)^{1/2}\cdot\nn\\
&\quad \ \cdot \left(\EE\left[\int_0^{\beta} |X^{\alpha}(\rho_s^m)-\OX_s^m|^4ds\right]\right)^{1/2}\notag\\
	&\quad\leq \  d  \delta^{1/2} \left(T \sup_{s\in[0,T]}  \EE\left[    |X^{\alpha}(\rho_s^m)-\OX_s^m|^4 \right]\right)^{1/2}, \nonumber
	\end{align}}
	and analogously,
	\begin{align}
	&\EE\left[\left|\int_\beta^{\beta+\delta} (1-e^{-\frac{\gamma}{m}(\beta+\delta-s)}) D(X^{\alpha}(\rho_s^m)-\OX_s^m)dB_s\right|^2\right]
	\notag\\
	&\quad\leq \  d  \delta^{1/2} \left(T \sup_{s\in[0,T]}  \EE\left[    |X^{\alpha}(\rho_s^m)-\OX_s^m|^4 \right]\right)^{1/2} . \nonumber
	\end{align}
	Therefore, summing up the above estimates and recalling  $0<m\leq m_0=\frac{1}{2}$, $\frac{1}{\gamma}\leq 2$, and the relations \eqref{Jensen} and \eqref{uniform-bd}, we arrive at
{\small 	\begin{align}
	&\EE[|\OX_{\beta+\delta}^m-\OX_{\beta}^m|^2] \nn\\
	&\quad\leq \ \frac{5}{\gamma^2} 
	\delta^2
	(\EE[|\OV_0|^4])^{\frac{1}{2}} 
	+\frac{10}{\gamma^2} 
	\left(\lambda^2  \delta T + \sigma^2d \left(\delta T\right)^{1/2} \right)
	\sup_{s\in[0,T]}  \left(\EE\left[    |X^{\alpha}(\rho_s^m)-\OX_s^m|^4 \right] \right)^{1/2}
	\notag\\
	&\quad\leq \ 
	C\left(\EE[|\OX_0|^4],\EE[|\OV_0|^4],M,C_u,C_l,\lambda,\sigma,d,T\right)
	\left(\delta^{\frac{1}{2}} + \delta+ \delta^2 \right). \nonumber
	\end{align}}
	Hence, for any $\varepsilon>0$, $\eta>0$, there exists some $\delta_0>0$ such that for all $0<m\leq \frac{1}{2}$ it holds that
	\begin{equation}
	\sup_{\delta\in[0,\delta_0]}\PP(|\OX_{\beta+\delta}^m-\OX_{\beta}^m|^2\geq \eta)\leq \sup_{\delta\in[0,\delta_0]}\frac{\EE[|\OX_{\beta+\delta}^m-\OX_{\beta}^m|^2]}{\eta}\leq \varepsilon\,.
	\end{equation}
	This justifies condition $Con 2$ in Lemma \ref{lemAldous}.  
\end{proof}
Next we shall identify the limit process, before which we recall a lemma on the stability estimate of the nonlinear term $X^\alpha(\rho)$.
\begin{lemma}{\cite[Lemma 3.2]{carrillo2018analytical}}\label{lemsta} 
	Assume that $\rho,\hat \rho\in\mc{P}_4(\RR^{d})$. Then the following stability estimate holds
	\begin{equation}\label{lemstaeq}
	|X^\alpha(\rho)-X^\alpha(\widehat \rho)|\leq CW_2(\rho,\widehat \rho)\,,
	\end{equation}
	where $W_2$ is the $2$-Wasserstein distance, and $C$ depends only on $\alpha,L$,  $\int_{\RR^d}|x|^4\rho(dx)$, and $\int_{\RR^d}|x|^4\hat\rho(dx)$.
\end{lemma}
Finally let us prove Theorem \ref{thmzero-limit}:
\begin{proof}{\textbf{(Theorem \ref{thmzero-limit})}}
	By Theorem \ref{thmtight1},  each subsequence $\{\OX^{m_k}\}_{k\in\mathbb N}$ with $m_0\leq 1/2$ and $m_k$ converging to $0$ as $k\rightarrow \infty$  admits a subsequence (denoted w.l.o.g. by itself) that converges weakly.
	By Skorokhod's lemma (see \cite[Theorem 6.7 on page
	70]{billingsley2013convergence}) and the existence and uniqueness of strong solution to SDE \eqref{MVeq},   we may find a common probability space $(\Omega,\mc{F},\PP)$ on which the joint processes $\{(\OX^{m_k},B)\}_{k\in\mathbb N}$ converge to some process $(\widehat X,B)$ as random variables valued in $\mc{C}([0,T];\RR^{2d})$ almost surely. Here $B$ is an identical $d$-dimensional Wiener process on $(\Omega,\mc{F},\PP)$. In particular, we have 
	\begin{align}
	\mathbb P\left( \lim_{k\rightarrow \infty} \sup\limits_{t\in[0,T]}|\OX^{m_k}_t-\widehat X_t| =0 \right)=1\,. \label{converge-as}
	\end{align}
	We shall verify that the limit $\widehat X$ is indeed the unique solution $\OX$ to SDE \eqref{MVCBO1}.
		
	Recalling the existence and uniqueness of the strong solution $\OX^{m_k}$ to SDE \eqref{Fubini} in Theorem \ref{thm-huang2021mean},  we have
{\small 	\begin{align}
\label{eq-X-mk}
	\OX_t^{m_k}&=\OX_0+\frac{m_k}{\gamma}(1-e^{-\frac{\gamma}{{m_k}}t})\OV_0
	+\frac{\lambda}{\gamma}\int_0^t(1-e^{-\frac{\gamma}{{m_k}}(t-s)}) (X^{\alpha}(\rho_s^{m_k})-\OX_s^{m_k})ds\\
	& \quad +\frac{\sigma}{\gamma}\int_0^t(1-e^{-\frac{\gamma}{{m_k}}(t-s)}) D(X^{\alpha}(\rho_s^{m_k})-\OX_s^{m_k})dB_s\notag\,. 
	\end{align}}
	By the estimates in \eqref{uniform-bd} and Fatou's lemma 
	there exists a constant $C_2 $ being independent of $m_k$ such that
	\begin{align}\label{estimate-L4-bd}  
	\sup_{k\in\mathbb N} & \sup_{t\in[0,T]}\EE \left[|\OX_t^{m_k}|^4\right]  +
	\sup_{t\in[0,T]} \EE\left[\left|\widehat X_t\right|^4 \right]\nn\\[-.25cm]
	\\[-.25cm]
	 & \leq  C_2:=C(\EE[|\OX_0|^4],\EE[|\OV_0|^4],C_{\alpha,\TE},\lambda,\sigma,d,T)<\infty.\nn
	\end{align}
	As a straightforward consequence of the above boundedness, it holds that
	\begin{align}
	\sup_{k\in \NN, t\in[0,T]} \PP(|\OX_t^{m_k} -\widehat X_t| > A)
	\leq \frac{2^4C_2}{A^4},\quad \forall\, A>0.
	\end{align}
	Thus, the dominated convergence theorem gives that for each $A>0$,
	\begin{align}
	&\lim_{k\rightarrow \infty}\EE\left[\int_0^T |\OX_t^{m_k}-\widehat X_t|^2\,dt \right] 
	\nonumber\\
	&\quad\leq
	\limsup_{k\rightarrow \infty}\bigg(\EE\left[\int_0^T |\OX_t^{m_k}-\widehat X_t|^2 \wedge A^2\,dt \right]\nn\\
	&\qquad+ \EE \left[\int_0^T |\OX_t^{m_k}-\widehat X_t|^2 1_{\{|\OX_t^{m_k}-\widehat X_t|>A\}}\,dt \right]  \bigg)
	\nonumber \\
	&\quad\leq
	\limsup_{k\rightarrow \infty}\EE\left[\int_0^T |\OX_t^{m_k}-\widehat X_t|^2 \wedge A^2\,dt \right]\nn\\
	&\qquad+T\cdot \sup_{k\in\mathbb N} \sup_{t\in[0,T]} \left(\EE \left[|\OX_t^{m_k}-\widehat X_t|^4\right] \right)^{1/2}   \left|  \PP(|\OX_t^{m_k} -\widehat X_t| > A) \right|^{1/2}
	\nonumber
	\\
	&\quad\leq
	\limsup_{k\rightarrow \infty}\EE\left[\int_0^T |\OX_t^{m_k}-\widehat X_t|^2 \wedge A^2\,dt \right]
	+  \frac{2^4 \, C_2 T}{A^2}
	\nonumber
	\\
	&\quad= \frac{2^4 \, C_2 T}{A^2}, \nonumber
	\end{align}
	which by the arbitrariness of $A>0$ indicates that 
	\begin{equation} \label{limit-zero}
	\lim_{k\rightarrow \infty}\EE\left[\int_0^T |\OX_t^{m_k}-\widehat X_t|^2\,dt \right] 
	=0.
	\end{equation}
	Letting $\rho(t,dx)$  be the probability distribution of $\widehat X_t$ for $t\in[0,T]$, Lemma \ref{lemXa} gives
	$$
	|{X}^{\alpha}(\rho_t)|\leq (b_1+b_2\EE[|\widehat X_t|^2])^{\frac{1}{2}} \leq (b_1+b_2C_2^{\frac{1}{2}})^\frac{1}{2}=:C_3,
	$$
	and thus
	\begin{equation}\label{Xalbound}
	\sup_{k\in\mathbb N} \sup_{t\in[0,T]}|{X}^{\alpha}(\rho_t^{m_k})|\leq  C_3, \quad \text{ and }\quad
	\sup_{t\in[0,T]} |{X}^{\alpha}(\rho_t)|\leq  C_3\,.
	\end{equation}	
	Then we compare the SDEs \eqref{MVCBO1} and \eqref{eq-X-mk} term by term. By Lemma \ref{lemsta}, we have
	\begin{equation*}
	|X^{\alpha}(\rho_t^{m_k})-X^\alpha( \rho_t)|^2\leq C W_2^2(\rho_t^{m_k}, \rho_t)\leq C \EE[|\OX_t^{m_k}-\widehat X_t|^2],
	\end{equation*}
	and thus by using the fact that  $\gamma=1-m_k$, one has
{\small  \begin{align}
	&\EE \bigg[ \bigg|
	\frac{\lambda}{\gamma}
	\int_0^t(1-e^{-\frac{\gamma}{{m_k}}(t-s)}) (X^{\alpha}(\rho^{m_k}_s)-\OX_s^{m_k})ds- \lambda \int_0^t  (X^{\alpha}(\rho_s)-\widehat X_s)ds \bigg|^2\bigg]\nn\\
	 &\quad \leq  2 \EE \bigg[  \bigg|\frac{\lambda}{1-m_k}\int_0^t(1-e^{-\frac{1-m_k}{{m_k}}(t-s)}) \cdot (X^{\alpha}(\rho^{m_k}_s)- X^{\alpha}(\rho_s) + \widehat X_s-\OX_s^{m_k})ds \bigg|^2 \bigg]
	\nonumber \\
	&\qquad +2 \EE \left[  \left|{\lambda} 
	\int_0^t \left(\frac{1-e^{-\frac{1-m_k}{{m_k}}(t-s)}}{1-m_k} -1\right) (X^{\alpha}(\rho_s)- \widehat X_s)ds \right|^2  
	\right]
	\nonumber \\
	&\quad \leq 
	C  \EE \left[ \int_0^t  \left|   \widehat X_s-\OX_s^{m_k} \right|^2 ds \right] +C{\lambda^2}  \int_0^t \left|\frac{1-e^{-\frac{1-m_k}{{m_k}}(t-s)}}{1-m_k} -1\right|^2ds
	\nn\\
	\ &\qquad\cdot \EE \left[  \int_0^T \left|X^{\alpha}(\rho_s)- \widehat X_s\right|^2  ds 
	\right] 
	\nonumber \\
	&\quad \leq 
	C \EE \left[ \int_0^t  \left|   \widehat X_s-\OX_s^{m_k} \right|^2 ds \right]
	+ C 
	\int_0^t \left|\frac{1-e^{-\frac{1-m_k}{{m_k}}(t-s)}-(1-m_k)}{1-m_k} \right|^2ds  
	\nonumber \\
	&\quad \leq 
	C \EE \left[ \int_0^t  \left|   \widehat X_s-\OX_s^{m_k} \right|^2 ds \right]
	+ C 
	\int_0^t \left( \left|{m_k}\right|^2  +  e^{-\frac{2(1-m_k)}{{m_k}}(t-s)} \right) ds 
	\nonumber \\
	&\quad \leq 
	C \EE \left[ \int_0^t  \left|   \widehat X_s-\OX_s^{m_k} \right|^2 ds \right]
	+ C 
	\left( t \left|{m_k}\right|^2  + \frac{m_k}{2(1-m_k)}   \right)  , \label{est-drift}
	\end{align}}
	where the constants $C$s are independent of $k$.
	For the stochastic integrals, it holds analogously that
{\small 	\begin{align}
	&\EE \ \bigg[ \bigg|
	\frac{\sigma}{\gamma}
	\int_0^t(1-e^{-\frac{\gamma}{{m_k}}(t-s)}) D(X^{\alpha}(\rho^{m_k}_s)-\OX_s^{m_k})dB_s \nn\\
	&\qquad- \sigma \int_0^t  D(X^{\alpha}(\rho_s)-\widehat X_s)dB_s \bigg|^2\bigg]\nn\\
	&\quad \leq{d\sigma^2} \sum_{n=1}^d
	\EE \bigg[ \bigg|
	\frac{1}{\gamma}
	\int_0^t(1-e^{-\frac{\gamma}{{m_k}}(t-s)}) (X^{\alpha}(\rho^{m_k}_s)-\OX_s^{m_k})_ndB_s^ne_n\nn\\
	&\qquad- \int_0^t  (X^{\alpha}(\rho_s)-\widehat X_s)_ndB_s^ne_n \bigg|^2\bigg]
	\nonumber \\
	&\quad =d{\sigma^2} \sum_{n=1}^d\EE \bigg[\int_0^t \bigg| \frac{1-e^{-\frac{\gamma}{{m_k}}(t-s)}} {\gamma} (X^{\alpha}(\rho^{m_k}_s)-\OX_s^{m_k})_n \nn\\
	&\qquad - (X^{\alpha}(\rho_s)-\widehat X_s)_n\bigg|^2 ds \bigg]\,.
	\end{align}}
	Thus we have
{\small 		\begin{align}
	&\EE \bigg[ \bigg|
	\frac{\sigma}{\gamma}
	\int_0^t(1-e^{-\frac{\gamma}{{m_k}}(t-s)}) D(X^{\alpha}(\rho^{m_k}_s)-\OX_s^{m_k})dB_s\nn\\
	& \quad\quad- \sigma \int_0^t  D(X^{\alpha}(\rho_s)-\widehat X_s)dB_s \bigg|^2\bigg]
	\nonumber \\
	&\quad \leq 
	{2d\sigma^2} \sum_{n=1}^d
	\EE \bigg[ 
	\int_0^t \bigg| \frac{1-e^{-\frac{\gamma}{{m_k}}(t-s)}} {\gamma}   \cdot \nn\\
	&\qquad \left( (X^{\alpha}(\rho^{m_k}_s)-\OX_s^{m_k})_n -     (X^{\alpha}(\rho_s)-\widehat X_s)_n\right) \bigg|^2\,ds\bigg]
	\nonumber \\
	&\qquad+  {2d\sigma^2}  \sum_{n=1}^d \EE \left[ 
	\int_0^t \left|
	\left(
	\frac{1-e^{-\frac{\gamma}{{m_k}}(t-s)}} {\gamma} -1
	\right)
	(X^{\alpha}(\rho_s)-\widehat X_s)_n\right|^2 ds \right]
	\nonumber \\
	&\quad \leq 
	C
	\EE \left[ 
	\int_0^t \left|     \OX_s^{m_k}     -\widehat X_s  \right|^2\,ds\right]\nonumber \\
	&\qquad +  {2d\sigma^2} \sup_{s\in [0,t]}  \EE \left[  \left|(X^{\alpha}(\rho_s)-\widehat X_s)\right|^2 \right] \cdot 
	\int_0^t \left|
	\left(
	\frac{1-e^{-\frac{\gamma}{{m_k}}(t-s)}} {\gamma} -1
	\right)
	\right|^2 ds  
	\notag\\
	 &\quad \leq 
	C \EE \left[ \int_0^t  \left|   \widehat X_s-\OX_s^{m_k} \right|^2 ds \right]
	+ C  
	\left( t \left|{m_k}\right|^2  + \frac{m_k}{2(1-m_k)}   \right) .
	\label{ineq-converge}
	\end{align}}
	In addition, it is obvious that
	\begin{align}
	\left| \frac{m_k}{\gamma}(1-e^{-\frac{\gamma}{m_k}t})\OV_0\right|
	\leq  C m_k \left| \OV_0\right|. \label{est-V0}
	\end{align}	
	Combining the estimates \eqref{est-drift}-\eqref{est-V0}, letting $k$ tend to infinity on both sides of \eqref{eq-X-mk} and recalling $\frac{1}{2}\geq m_k \rightarrow 0^+$ and the relation \eqref{limit-zero}, we have
	\begin{align*}
	\widehat X_t=\OX_0
	+\lambda\int_0^t(X^{\alpha}(\rho_s)-\widehat X_s)ds +\sigma\int_0^t D(X^{\alpha}(\rho_s)-\widehat X_s)dB_s.
	\end{align*}
	Therefore, the limit $\widehat X$ turns out to be a solution to SDE \eqref{MVCBO1}.
	Meanwhile, in view of the continuity of $X^{\alpha}(\rho)$ in Lemma \ref{lemsta}, we can easily show that \eqref{MVCBO1}  admits a unique (strong) solution as in Theorem \ref{thm-huang2021mean}  by using Leray-Schauder fixed point theorem as in \cite[Theorem 3.1]{carrillo2018analytical}. Thus, we must have $\widehat X = \OX$ that is the unique strong solution to SDE \eqref{MVCBO1} with $\sup_{t\in[0,T]} \EE \left[ |\OX_t|^4\right]\leq C_2$. Further, due to the arbitrariness of the subsequence $\{\OX^{m_k}\}_{k\in\mathbb N}$, we conclude that as $m\rightarrow 0^+$, the sequence of  stochastic processes $\{\OX^m\}_{0< m\leq \frac{1}{2}}$  converge weakly to  the unique solution $\overline X$ to SDE \eqref{MVCBO1}.  
		
	Finally, to measure the distance between $\OX^m$ and the limit $\widehat X=\OX$, we may have similar calculations to \eqref{est-drift}-\eqref{est-V0}, subtract both sides of SDEs \eqref{MVCBO1} from those of \eqref{eq-X-mk}, and arrive at
	\begin{align}
	\EE[|\OX^{m}_t-\OX_t|^2]\leq C\int_0^t\EE[|\OX^{m}_s-\OX_s|^2]ds+ 
	C\, m,\quad t\in[0,T]. \nonumber
	\end{align}
	By Gronwall's inequality it implies that
	\begin{equation}\label{convergence-l2}
	\sup_{t\in[0,T]}\EE[|\OX^{m}_t-\OX_t|^2]\leq Cm ,
	\end{equation}
	where $C$ depends only on $\EE[|\OX_0|^4],\EE[|\OV_0|^4],C_u,M,C_l,\lambda,\sigma,d$, and $T$.   This completes the proof. 
\end{proof}

\subsection{The general case with memory}
Next, we consider the same small inertia scaling in the general case with dependence from the local best. Again, we first write down the 
nonlinear McKean-Vlasov process corresponding to  the SD-PSO system \eqref{eq:psocir}, which is of the form

\begin{subequations}
	\begin{eqnarray}
	d\OX_t^m &=& \OV_t^mdt, \label{mXeq}\\
	d \OY_{t}^{m} &=& \nu\left(\OX_{t}^{m}-\OY_{t}^{m}\right) S^{\beta}\left(\OX_{t}^{m}, \OY_{t}^{m}\right) d t, \label{mYeq}\\
	d\OV_{t}^{m} &=& -\frac{\gamma}{m} \OV_{t}^{m} d t+\frac{\lambda_{1}}{m}\left(\OY_{t}^{m}-\OX_{t}^{m}\right) d t \nn \\[-.2cm]
	\nonumber
	\\
	\label{mVeq}
	&& +\frac{\lambda_{2}}{m}\left({Y}^{\alpha}(\overline \rho_t^m)-\OX_{t}^{m}\right) d t +\frac{\sigma_{1}}{m} D\left(\OY_{t}^{m}-\OX_{t}^{m}\right) d B_{t}^{1} \\[-.2cm] 
	\nn\\
	&& +\frac{\sigma_{2} }{m}D\left({Y}^{\alpha}(\overline \rho_t^m)-\OX_{t}^{m}\right) d B_{t}^{2}\nn\,,
	\end{eqnarray}
\end{subequations}
where $B^1$ and $B^2$ are two mutually independent d-dimensional Wiener processes, and similarly to the previous section, we introduce the following regularization of the global best
position
\begin{align*}\label{global_best_regular}
{Y}^{\alpha}(\overline \rho^m_t)=\frac{\int_{\mathbb{R}^{d}} y \omega_{\alpha}(y) {\overline \rho}^m(t,dy) }{\int_{\mathbb{R}^{d}} \omega_{\alpha}(y) {\bar \rho}^m(t,dy) }, 
 \qquad {\overline \rho}^m(t,y)=\iint_{\mathbb{R}^{d} \times \mathbb{R}^{d}} f^m(t,dx, y,d v)\,.
\end{align*}
As $m\to 0^+$ we formally get from \eqref{mVeq}
\[
\begin{split}
\OV_{t}^{0} dt &=\lambda_1\left(\OY_t^0-\OX^0_t\right)dt +\lambda_2\left({Y}^{\alpha}(\overline \rho_t^0)-\OX^0_t\right)dt\\
&\quad +\sigma_1 D(\OY_t^0-\OX^0_t)dB^{1}_t+\sigma_2 D({Y}^{\alpha}(\overline \rho_t^0)-\OX^0_t)dB^{2}_t,
\end{split}
\] 
which inserted into \eqref{mXeq} and omitting the superscripts corresponds to a novel \emph{CBO system with local best}
\coloredeq{eq:cbozero}{
\begin{split}
d \OX_t&=
\lambda_1 (\OY_t-\OX_t)dt +\lambda_2 (Y^{\alpha}(\bar\rho_t)-\OX_t)dt \\
&\quad+\sigma_1D(\OY_t-\OX_t)dB_t^1+\sigma_2 D(Y^{\alpha}(\bar\rho_t)-\OX_t)dB_t^2,\\
d\OY_{t}&=\nu\left(\OX_{t}-\OY_{t}\right) S^{\beta}\left(\OX_{t}, \OY_{t}\right) dt\,.
\end{split}
}
In contrast with the model recently introduced in \cite{TW} the above first order CBO method avoids backward time integration through the use of an additional differential equation. We refer to \cite{grassi2021consensus} for further details on the above CBO system.
\subsubsection{Formal derivation in the mean-field case}
Concerning the corresponding MF-PSO limit characterized by \eqref{PDEii} for $m\to 0^+$ we can essentially perform analogous computations as in the previous section (see \cite{Grassi2021PSO}). Similarly by considering the local Maxwellian with unitary mass and zero momentum
\[
\begin{split}
{\mathcal M}_m(x,y,v,t)&= \prod_{j=1}^d M_{m}(x_j,y_j,v_j,t), \\
M_{m}(x_j,y_j,v_j,t) &= \frac{m^{1/2}}{\pi^{1/2} |\Sigma(x_j,y_j,t)|} 
\exp\left\{-\frac{m v_j^2}{\Sigma(x_j,y_j,t)^2}\right\},
\end{split}
\]
where
\[
\Sigma(x_j,y_j,t)^2 = {\sigma_2^2}(x_j - Y^\alpha_j(\bar\rho))^2+{\sigma_1^2}(x_j - y_j)^2,
\]
we can assume for $m \ll 1$
\begin{equation}
f(x,y,v,t)=\rho(x,y,t){\mathcal M}_\e(x,y,v,t).
\label{eq:maxw2}
\end{equation}
After integration of the MF-PSO equation \eqref{PDEii} with respect to $v$, we get the second order \emph{macroscopic PSO system with local best}
\coloredeq{eq:macro2}{
\begin{split}
\frac{\partial \rho}{\partial t} + \nabla_x \cdot (\rho u) + \nabla_y \cdot \left(\nu(x-y)S^\beta(x,y)\rho\right)&=0\\
\frac{\partial (\rho u)_j}{\partial t} + \frac{\sigma^2}{2 m}\frac{\partial}{\partial x_j} \left(\rho(x,t) \Sigma(x_j,y_j,t)^2\right)&= \\
-\frac{\gamma}{m} (\rho u)_j + \frac1{m} (\lambda_1 (y_j-x_j)&+\lambda_2 (Y_j^\alpha(\bar\rho) - x_j) )\rho.
\end{split}}
Formally, as $m\to 0^+$, the above system reduces to a novel \emph{mean-field CBO system with local best}
\coloredeq{eq:CBOlb}{
\begin{split}
\frac{\partial \rho}{\partial t} &+ \nabla_x \cdot \left(\lambda_1 (y-x)+\lambda_2 (Y^\alpha(\bar\rho) - x) \right)\rho\\
& + \nabla_y \cdot \left(\nu(x-y)S^\beta(x,y)\rho\right)\\
 &=\frac{1}{2}\sum_{j=1}^d \frac{\partial^2}{\partial x^2_j} \left(\rho(x,t)\left( \sigma_1^2 (x_j-y_j)^2+\sigma_2^2 (x_j-Y_j^\alpha(\bar\rho))^2\right)\right).
\end{split}}

\subsubsection{Rigorous derivation}
Since the proof of the zero-inertia limit for the PSO dynamics with memory effects follows similar arguments as developed in section \ref{seczero} and  no essential innovation is needed to be explained, we only recall the main results here.

Let us solve \eqref{mVeq} to obtain
\begin{align}\label{OXm}
\OX_t^m&=\OX_0+\frac{m}{\gamma}(1-e^{-\frac{\gamma}{m}t})\OV_0
+\frac{\lambda_1}{\gamma}\int_0^t(1-e^{-\frac{\gamma}{m}(t-s)}) (\OY_s^m-\OX_s^m)ds \nn\\ &\quad +\frac{\sigma_1}{\gamma}\int_0^t(1-e^{-\frac{\gamma}{m}(t-s)}) D(\OY_s^m-\OX_s^m)dB_s^1\notag\\
&\quad +\frac{\lambda_2}{\gamma}\int_0^t(1-e^{-\frac{\gamma}{m}(t-s)}) (Y^{\alpha}(\bar\rho_s^m)-\OX_s^m)ds\\ &\quad+\frac{\sigma_2}{\gamma}\int_0^t(1-e^{-\frac{\gamma}{m}(t-s)}) D(Y^{\alpha}(\bar\rho_s^m)-\OX_s^m)dB_s^2\nn
\end{align}
and
\begin{align}\label{OYm}
\OY_{t}^{m}=\OY_0+\nu\int_0^t\left(\OX_{s}^{m}-\OY_{s}^{m}\right) S^{\beta}\left(\OX_{s}^{m}, \OY_{s}^{m}\right) d s\,.
\end{align}
Similar to Theorem \ref{thmtight1} one can prove the following result of tightness.
\begin{theorem}\label{thmtight2}
	Let Assumption \ref{asum} hold and $(\OX_t^m,\OY_t^m,\OV_t^m)_{t\in[0,T]}$ satisfy the system \eqref{mXeq}--\eqref{mVeq}.  For each countable subsequence $\{m_k\}_{k\in\mathbb N} \subset [0,\frac{1}{2}]$ with $\lim_{k\rightarrow \infty}m_k=0$, the sequence of probability distributions   $\{\rho^{m_k}\}_{k\in\mathbb N}$ of $\{(\OX^{m_k},\OY^{m_k})\}_{k\in\mathbb N}$  is tight.
\end{theorem}
Then following the lines of the proof in Theorem \ref{thmzero-limit}, one can obtain
\begin{theo-frmd}
	Let Assumption \ref{asum} hold and $(\OX_t^m,\OY_t^m)_{t\in[0,T]}$ satisfy the system \eqref{OXm}--\eqref{OYm}. Then as $m\rightarrow 0^+$, the sequence of  stochastic processes $\{(\OX^m,\OY^m)\}_{0< m\leq \frac{1}{2}}$  converge weakly to $(\OX,\OY)$  which is the unique solution to the following coupled SDE:
	\begin{align*}
	\OX_t&=\OX_0
	+\lambda_1\int_0^t (\OY_s-\OX_s)ds +\sigma_1\int_0^tD(\OY_s-\OX_s)dB_s^1\nn\\
	&\quad+\lambda_2\int_0^t (Y^{\alpha}(\bar\rho_s)-\OX_s)ds +\sigma_2\int_0^t D(Y^{\alpha}(\bar\rho_s)-\OX_s)dB_s^2,
	\\
	\OY_{t}&=\OY_0+\nu\int_0^t\left(\OX_{s}-\OY_{s}\right) S^{\beta}\left(\OX_{s}, \OY_{s}\right) d s\,.
	\end{align*}
	Moreover it holds that
	\begin{equation}\label{est-convegence-m-xy}
	\sup_{t\in[0,T]}\EE\left[\big|\OX^m_t-\OX_t\big|^2
	+\big|\OY^m_t-\OY_t\big|^2
	\right]\leq C \, m\,,
	\end{equation}
	where the constant $C$ depends only on $\EE[|\OX_0|^4]$, $\EE[|\OY_0|^4]$, $\EE[|\OV_0|^4]$, $\lambda_1$, $\sigma_2$, $\lambda_2$, $\sigma_2$, $d,\beta,T,C_u,M,C_l$, and $\nu$.
\end{theo-frmd}

\section{Convergence to the global minimum}
In this section we present some results on the global convergence of the PSO model \eqref{PSO} without memory effects. 
The extension to the case with memory effects is not strightforward and is actually under study. Here we will follow the presentation in \cite{HJK}, we refer to \cite{carrillo2018analytical,carrillo2019consensus,benfenati2021binary,fhps20-2,fornasier2021anisotropic} for similar results for CBO and related models. A different approach to the global convergence of CBO has been presented recently in \cite{fornasier2021consensusbased}.

Let $(\OX_t,\OV_t)_{t\geq 0}$ be the solution to the nonlinear SDE \eqref{MVeq} (dropping the superscript $m$), and
consider the quantity \[\CH(t):=(\frac{\gamma}{2m})^2|\OX_t-\EE[\OX_t]|^2+|\OV_t|^2+\frac{\gamma}{2m}(\OX_t-\EE[\OX_t])\cdot \OV_t,\] then it holds that
\begin{equation}
\begin{split}
\label{ineqH}
\CH(t)&\geq \frac{1}{2}(\frac{\gamma}{2m})^2|\OX_t-\EE[\OX_t]|^2+\frac{1}{2}|\OV_t|^2\\
\CH(t) &\leq \frac{3}{2}(\frac{\gamma}{2m})^2|\OX_t-\EE[\OX_t]|^2+\frac{3}{2}|\OV_t|^2\\
&\leq \frac{3}{2}((\frac{\gamma}{2m})^2+1)(|\OX_t-\EE[\OX_t]|^2+|\OV_t|^2)\,.
\end{split}
\end{equation}
The goal is then to obtain the decay property of $\CH(t)$.

In the following we shall use the notation
\begin{equation}
\delta\OX_t:=\OX_t-\EE[\OX_t]\,, 
\end{equation}
then $\EE[|\delta\OX_t|^2]$ is the variance of $X_t$. Now we can derive an evolution inequality of the quantity $\EE[\CH(t)]$.
\begin{theorem} Under the Assumption \ref{asum},	let $(\OX_t,\OV_t)_{t\geq 0}$ be the solution to the nonlinear SDE \eqref{MVeq}. Then  $\EE[\CH(t)]$ satisfies
	\begin{align}
	\frac{d}{dt}\EE[\CH(t)]\leq  &-\frac{\gamma}{m}\EE[|\OV_t|^2]  \nn \\
	 &-\left(\frac{\lambda\gamma}{2m^2}-(\frac{2\lambda^2}{\gamma m}+\frac{\sigma^2}{m^2})\frac{2e^{-\alpha\underline \TE}}{\EE[e^{-\alpha\TE(\OX_t)}]}\right)\EE[|\delta\OX_t|^2]\,.
	\end{align}
\end{theorem}
\begin{proof}
	First the integration by parts formula gives
	\begin{equation}\label{term1}
	\frac{d}{dt} \EE[|\delta\OX_t|^2]=2\EE[\delta\OX_t\cdot \OV_t]\,,
	\end{equation}
	where we have used the fact that $\EE[\delta\OX_t\cdot \EE[V_t]]=0$.
	Applying It\^{o}-Doeblin formula and taking zero-value of the stochastic integrals, we have for any $\varepsilon>0$,
	\begin{align}\label{term2}
	\frac{d}{dt} \EE[|\OV_t|^2] =& -2\frac{\gamma}{m}\EE[|\OV_t|^2]+2\frac{\lambda}{m}\EE[\OV_t\cdot(X^\alpha(\rho_t)-\OX_t)]\nn\\
	&+\frac{\sigma^2}{m^2}\EE[|X^\alpha(\rho_t)-\OX_t|^2] \nn \\
	\leq&-(\frac{2\gamma}{m}-\frac{\lambda}{\varepsilon m})\EE[|\OV_t|^2]+(\frac{\varepsilon\lambda}{ m}+\frac{\sigma^2}{m^2})\EE[|X^\alpha(\rho_t)-\OX_t|^2]
	\,.
	\end{align}
	Further by It\^{o}-Doeblin formula, it holds that
	\begin{align}\label{ito}
	&\frac{d}{dt}\EE[\delta\OX_t\cdot \OV_t]\nn\\
	&\qquad =\EE[|\OV_t|^2]-(\EE[\OV_t])^2-\frac{\gamma}{m}\EE[\delta\OX_t\cdot\OV_t]+\frac{\lambda}{m}\EE[\delta\OX_t\cdot(X^\alpha(\rho_t)-\OX_t)]\nn\\
	&\qquad \leq \EE[|\OV_t|^2] -\frac{\gamma}{2m}\frac{d}{dt} \EE[|\delta\OX_t|^2]-\frac{\lambda}{m}\EE[|\delta\OX_t|^2]+\frac{\lambda}{m}\EE[\delta\OX_t\cdot(X^\alpha(\rho_t)-\EE[\OX_t])]
	\nn\\
	&\qquad = \EE[|\OV_t|^2]-\frac{\gamma}{2m}\frac{d}{dt} \EE[|\delta\OX_t|^2]-\frac{\lambda}{m}\EE[|\delta\OX_t|^2]\,.
	\end{align}
	where we have used \eqref{term1} and the fact that $\EE[\delta\OX_t\cdot(X^\alpha(\rho_t)-\EE[\OX_t])]=0$. Thus, we have
	\begin{align}\label{term3}
	(\frac{\gamma}{2m})^2\frac{d}{dt} \EE[|\delta\OX_t|^2]&+\frac{\gamma}{2m}\frac{d}{dt}\EE[\delta\OX_t\cdot \OV_t] \nn \\
	& \leq \frac{\gamma}{2m}\EE[|\OV_t|^2]-\frac{\lambda\gamma}{2m^2}\EE[|\delta\OX_t|^2]\,.
	\end{align}
	Collecting estimates \eqref{term2} and \eqref{term3}  yields that
	\begin{align}
	\frac{d}{dt}\EE[\CH(t)]&\leq -\left(\frac{2\gamma}{m}-\frac{\lambda}{\varepsilon m}-\frac{\gamma}{2m}\right)\EE[|\OV_t|^2]-\frac{\lambda\gamma}{2m^2}\EE[|\delta\OX_t|^2] \nn\\
	&\quad+(\frac{\varepsilon\lambda}{ m}+\frac{\sigma^2}{m^2})\EE[|X^\alpha(\rho_t)-\OX_t|^2] 
	\,.
	\end{align}
	To estimate the term $\EE[|\OX_t-X^\alpha(\rho_t)|^2]$, we apply Jensen's inequality to obtain 
	\begin{align}\label{fact}
	\EE[|\OX_t-X^\alpha(\rho_t)|^2]&\leq \frac{\iint |x-y|^2\omega_\alpha^\TE(y) \rho_t(dy)\rho_t(dx)}{\int \omega_\alpha^\TE (y)\rho_t(dy) } \nn \\ &\leq 
	2e^{-\alpha\underline \TE}\frac{\EE[|\delta\OX_t|^2]}{\EE[e^{-\alpha\TE(\OX_t)}]}\,.
	\end{align}
	Hence, by choosing $\varepsilon=\frac{2\lambda}{\gamma}$  we obtain
	\begin{align}\label{eqdH}
	\frac{d}{dt}\EE[\CH(t)]\leq&  -\frac{\gamma}{m}\EE[|\OV_t|^2] \nn \\ &-\left(\frac{\lambda\gamma}{2m^2}-(\frac{2\lambda^2}{\gamma m}+\frac{\sigma^2}{m^2})\frac{2e^{-\alpha\underline \TE}}{\EE[e^{-\alpha\TE(\OX_t)}]}\right)\EE[|\delta\OX_t|^2],
	\end{align}
	which completes the proof.
\end{proof}
Next we study the evolution of the quantity $\EE[e^{-\alpha\TE(\OX_t)}]$, and we need an additional assumption on the cost function $\TE$ that
\begin{center}
	$\textbf{\textit{A1}}$:  $\TE\in C^2(\R^d)$ with $\|\nabla^2\TE\|_\infty\leq c_\TE$ for some constant $c_\TE>0$.
\end{center}
\begin{lemma}\label{lemW}
	Under the Assumption \ref{asum} and $\textbf{\textit{A1}}$,	let $(\OX_t,\OV_t)_{t\geq 0}$ be the solution to the nonlinear SDE \eqref{MVeq}. Then it holds that
\begin{align}
	\frac{d^2}{dt^2}(\EE[e^{-\alpha\TE(\OX_t)}])^2 \geq &-\frac{\gamma}{m}\frac{d}{dt}(\EE[e^{-\alpha\TE(\OX_t)}])^2 \nn \\
	 &-4(\alpha +\frac{\alpha\lambda}{m} 2(\frac{2m}{\gamma})^2)c_\TE e^{-2\alpha\underline \TE}\EE[\CH(t)]\,.
	\end{align}
\end{lemma}
\begin{proof}
	First, applying It\^{o}-Doeblin formula and taking zero-value of the stochastic integrals, we have
	\begin{align*}
	&\frac{d}{dt}\EE[e^{-\alpha\TE(\OX_t)}]=-\alpha \EE [e^{-\alpha\TE(\OX_t)}\nabla \TE(\OX_t)\cdot \OV_t ]\nn\\
	&\quad = -\alpha \EE [ \int_0^t d \la e^{-\alpha\TE(\OX_s)}\nabla \TE(\OX_s), \OV_s \ra ]+\alpha \EE [ e^{-\alpha\TE(\OX_0)}\la  \nabla \TE(\OX_0), \OV_0 \ra ] \\
	&\quad = \alpha \EE [e^{-\alpha\TE(\OX_0)}  \la \nabla \TE(\OX_0), \OV_0 \ra ]-\alpha \EE [ \int_0^t  \la  e^{-\alpha\TE(\OX_s)}\OV_s\nabla^2 \TE(\OX_s), \OV_s \ra  ds]\\
	&\qquad +\alpha^2 \EE [ \int_0^t e^{-\alpha\TE(\OX_s)} |\la \nabla \TE(\OX_s), \OV_s \ra|^2ds ]\\
	&\qquad-\alpha \EE [ \int_0^t e^{-\alpha\TE(\OX_s)} \la \nabla \TE(\OX_s), -\frac{\gamma}{m}\OV_s \ra ds]\\
	&\qquad-\alpha \EE [ \int_0^t e^{-\alpha\TE(\OX_s)} \la \nabla \TE(\OX_s), \frac{\lambda}{m}(X^\alpha(\rho_s)-\OX_s) \ra ds]\,.
	\end{align*}
	Further, differentiating both sides with respect to $t$ gives
	\begin{align}
	&\frac{d^2}{dt^2}\EE[e^{-\alpha\TE(\OX_t)}]=-\alpha \EE [ \la  e^{-\alpha\TE(\OX_t)}\OV_t\nabla^2 \TE(\OX_t), \OV_t \ra ]\nn\\
	&\qquad +\alpha^2 \EE [ e^{-\alpha\TE(\OX_t)} |\la \nabla \TE(\OX_t), \OV_t \ra|^2]\nn\\
	&\qquad -\alpha \EE [  e^{-\alpha\TE(\OX_t)} \la \nabla \TE(\OX_t), -\frac{\gamma}{m}\OV_t \ra ]\nn\\
	&\qquad-\alpha \EE [ e^{-\alpha\TE(\OX_t)} \la \nabla \TE(\OX_t), \frac{\lambda}{m}(X^\alpha(\rho_t)-\OX_t) \ra ] \nn\\
	&\quad \geq -\frac{\gamma}{m} \frac{d}{dt}\EE[e^{-\alpha\TE(\OX_t)}]-\alpha \EE [ \la  e^{-\alpha\TE(\OX_t)}\OV_t\nabla^2 \TE(\OX_t), \OV_t \ra ]\nn\\
	&\qquad -\alpha \EE [ e^{-\alpha\TE(\OX_t)} \la \nabla \TE(\OX_t), \frac{\lambda}{m}(X^\alpha(\rho_t)-\OX_t) \ra ] \nn\\
	&\quad =:-\frac{\gamma}{m} \frac{d}{dt}\EE[e^{-\alpha\TE(\OX_t)}]+I_1+I_2\,,
	\end{align}
	where one has used the fact that
	\begin{equation}
	-\frac{\gamma}{m} \frac{d}{dt}\EE[e^{-\alpha\TE(\OX_t)}]=\frac{\alpha\gamma}{m}\EE [  e^{-\alpha\TE(\OX_t)} \la \nabla \TE(\OX_t), \OV_t \ra ]\,.
	\end{equation}	
	According to assumption $\textbf{\textit{A1}}$, it is easy to see that
	\begin{align}\label{I13}
	I_1&\geq -\alpha \EE [  e^{-\alpha\TE(\OX_t)}\|\nabla^2 \TE(\OX_t)\|_\infty|\OV_t|^2 ]\geq -\alpha c_\TE e^{-\alpha \underline{\TE}} \EE [|\OV_t|^2] .
	\end{align}
	We further notice that 
	\begin{align}
	&\left|\EE [  e^{-\alpha\TE(\OX_t)}\la  \nabla \TE(\OX_t), (X^\alpha(\rho_t)-\OX_t) \ra ]\right|\nn\\
	&\quad = \left|\EE [e^{-\alpha\TE(\OX_t)}\la \nabla \TE(\OX_t)-\nabla \TE(X^\alpha(\rho_t)), (\OX_t-X^\alpha(\rho_t)) \ra ]\right|\nn\\
	&\quad \leq e^{-\alpha\underline \TE}c_\TE\EE[|\OX_t-X^\alpha(\rho_t)|^2]\,,
	\end{align}
	where we have used the fact that $\EE [e^{-\alpha\TE(\OX_t)} \la \nabla \TE(X^\alpha(\rho_t)), (\OX_t-X^\alpha(\rho_t)) \ra ]=0$. Furthermore since $\EE[|\OX_t-X^\alpha(\rho_t)|^2]\leq 2e^{-\alpha\underline \TE}\frac{\EE[|\delta\OX_t|^2]}{\EE[e^{-\alpha\TE(\OX_t)}]}$, one has
	\begin{align}\label{I4}
	I_2\geq &-\frac{\alpha\lambda}{m} e^{-
		\alpha\underline \TE}c_\TE\EE[|\OX_t-X^\alpha(\rho_t)|^2] \nn\\
		\geq& -\frac{\alpha\lambda}{m} 2e^{-2\alpha\underline \TE}c_\TE\frac{\EE[|\delta\OX_t|^2]}{\EE[e^{-\alpha\TE(\OX_t)}]}\,.
	\end{align}
	This combining with \eqref{I13} leads to
	\begin{align}
	\frac{d^2}{dt^2}\EE[e^{-\alpha\TE(\OX_t)}]&\geq -\frac{\gamma}{m} \frac{d}{dt}\EE[e^{-\alpha\TE(\OX_t)}] -\alpha c_\TE e^{-\alpha \underline{\TE}} \EE [|\OV_t|^2] \nn\\
	&\quad-\frac{\alpha\lambda}{m}c_\TE 2e^{-2\alpha\underline \TE}\frac{\EE[|\delta\OX_t|^2]}{\EE[e^{-\alpha\TE(\OX_t)}]}\,.
	\end{align}
	Using this, one can obtain
	\begin{align}
	&\frac{d}{dt}\left(\frac{1}{2}\frac{d}{dt}(\EE[e^{-\alpha\TE(\OX_t)}])^2\right)=\frac{d}{dt}\left(\EE[e^{-\alpha\TE(\OX_t)}]\frac{d}{dt}(\EE[e^{-\alpha\TE(\OX_t)}])\right)\nn\\
	&\quad =\left(\frac{d}{dt}(\EE[e^{-\alpha\TE(\OX_t)}])\right)^2+\EE[e^{-\alpha\TE(\OX_t)}]\frac{d^2}{dt^2}\EE[e^{-\alpha\TE(\OX_t)}]\nn\\
	&\quad \geq -\frac{\gamma}{2m}\frac{d}{dt}(\EE[e^{-\alpha\TE(\OX_t)}])^2-\alpha c_\TE e^{-2\alpha \underline{\TE}} \EE [|\OV_t|^2]\nn\\
	&\qquad -\frac{\alpha\lambda}{m} 2e^{-2\alpha\underline \TE}c_\TE\EE[|\delta\OX_t|^2] \nn\\
	&\quad \geq  -\frac{\gamma}{2m}\frac{d}{dt}(\EE[e^{-\alpha\TE(\OX_t)}])^2 \nn\\
	&\qquad -\left(2\alpha +2\frac{\alpha\lambda}{m} 2(\frac{2m}{\gamma})^2\right)c_\TE e^{-2\alpha\underline \TE}\EE[\CH(t)]
	\,,
	\end{align}
	where we have used \eqref{ineqH} in the last inequality. This completes the proof.
\end{proof}
Our main theorem on global convergence can be described in the following way:
\begin{theo-frmd}
	Under the Assumption \ref{asum} and $\textbf{\textit{A1}}$,	let $(\OX_t,\OV_t)_{t\geq 0}$ be the solution to the nonlinear SDE \eqref{MVeq}. Further we assume that the initial data $\OX_0$ and $\OV_0$ satisfy
	\begin{equation}\label{eq1}
	\mu:=\frac{\lambda\gamma}{2m^2}-(\frac{2\lambda^2}{\gamma m}+\frac{\sigma^2}{m^2})\frac{4e^{-\alpha\underline \TE}}{\EE[e^{-\alpha\TE(\OX_0)}]}>0,
	\end{equation}
	and
	{\small \begin{align}\label{asuE}
		&2\alpha\frac{\gamma}{m}\EE[e^{-\alpha\TE(\OX_0)}]\left(\EE[e^{-\alpha\TE(\OX_0)}\nabla \TE(\OX_0)\cdot \OV_0]\right)_+\nn\\
		&\quad+4(\alpha +\frac{\alpha\lambda}{m} 2(\frac{2m}{\gamma})^2)c_\TE e^{-2\alpha\underline \TE}\frac{\EE[\CH(0)]}{\chi(\frac{\gamma}{m}-\chi)}
		< \frac{3}{4}(\EE[e^{-\alpha\TE(\OX_0)}])^2\,,
		\end{align}}
	where we denote $x_+=\max\{x,0\}$, $\forall\, x\in\mathbb R$, and
	$$\chi=\frac{\min\{\mu,\frac{\gamma}{m}\}}{\frac{3}{2}((\frac{\gamma}{2m})^2+1)}.$$ 
	Then $\EE[|\OX_t-\EE[\OX_t]|^2]\to 0,\EE[|\OV_t|^2]\to 0$ exponentially fast  as $t\to \infty$, and there exists  some $\tilde x$ depending on $\alpha$ such that $\EE[\OX_t]\to \tilde x$ and $X^\alpha(\rho_t)\to \tilde x$   exponentially fast  as $t\to \infty$. Moreover it holds that
	{\small \begin{equation}
	\TE(\tilde x)-\underline{\TE}\leq \frac{1}{\alpha}\log(2)-\frac{1}{\alpha}\log(\EE[e^{-\alpha\TE(\OX_0)}])-\underline{\TE} \to 0\mbox{ as }\alpha \to \infty\,.
	\end{equation}}
\end{theo-frmd}
\begin{remark}
	If we additionally assume the inverse continuity of $\TE$ holds, namely for any $x \in \RR^d$ 
	there exists  a minimizer $x^*$ of $\TE$ (which may depend on $x$) such that  it holds 
	\begin{equation*}
	|x-x^\ast| \leq  C_0|\TE(x)-\underline \TE|^\ell\,,
	\end{equation*}
	where  $\ell, C_0$ are some positive constants, then one can conclude that $\tilde x\to x^*$ as $\alpha\to \infty$.
\end{remark}
\begin{proof}
	Define 
	\begin{equation}
	T:=\inf\left\{t\geq 0:~\EE[e^{-\alpha\TE(\OX_t)}]<\frac{1}{2}\EE[e^{-\alpha\TE(\OX_0)}]\right\} \mbox{ with }\inf \emptyset=\infty\,.
	\end{equation}
	Obviously, $T>0$. Assume that $T<\infty$, then for $t\in[0,T]$, one can deduce that
	\begin{align}
	&\frac{\lambda\gamma}{2m^2}-(\frac{2\lambda^2}{\gamma m}+\frac{\sigma^2}{m^2})\frac{2e^{-\alpha\underline \TE}}{\EE[e^{-\alpha\TE(\OX_t)}]}\nn\\
	\geq \ &\frac{\lambda\gamma}{2m^2}-(\frac{2\lambda^2}{\gamma m}+\frac{\sigma^2}{m^2})\frac{4e^{-\alpha\underline \TE}}{\EE[e^{-\alpha\TE(\OX_0)}]}=\mu>0\,.
	\end{align}
	Consequently by \eqref{eqdH} we have
	\begin{align}
	\frac{d}{dt}\EE[\CH(t)]&\leq  -\frac{\gamma}{m}\EE[|\OV_t|^2]
	-\mu \EE[|\delta\OX_t|^2] \nn\\ &\leq -\min\{\mu,\frac{\gamma}{m}\}(\EE[|\delta\OX_t|^2]+\EE[|\OV_t|^2])\nn\\
	&\leq -\frac{\min\{\mu,\frac{\gamma}{m}\}}{\frac{3}{2}((\frac{\gamma}{2m})^2+1)}\EE[\CH(t)]
	\,,
	\end{align}
	where we have used the estimate \eqref{ineqH}. This implies that
	\begin{equation}
	\EE[\CH(t)]\leq \EE[\CH(0)]\exp\left(-\frac{\min\{\mu,\frac{\gamma}{m}\}}{\frac{3}{2}((\frac{\gamma}{2m})^2+1)}t\right)=\EE[\CH(0)]\exp(-\chi t)\,.
	\end{equation}
	One further notice that
	\begin{equation*}
	\chi\leq \frac{\frac{\gamma}{m}}{\frac{3}{2}((\frac{\gamma}{2m})^2+1)}<\frac{\gamma}{m}\,.
	\end{equation*}
	Set $\mc{Y}(t):=(\EE[e^{-\alpha\TE(\OX_t)}])^2$. Then we have
	\begin{equation}
	\mc{Y}'(0)= -2\alpha\EE[e^{-\alpha\TE(\OX_0)}]\EE[e^{-\alpha\TE(\OX_0)}\nabla \TE(\OX_0)\cdot \OV_0].
	\end{equation}
	By Gronwall's inequality, it follows from Lemma \ref{lemW} that
	\begin{align*}
	&\frac{d}{dt}\mc{Y}(t) \geq \mc{Y}'(0)\exp(-\frac{\gamma}{m}t)\nn\\
	&\qquad -4\left(\alpha +\frac{\alpha\lambda}{m} 2(\frac{2m}{\gamma})^2\right)c_\TE e^{-2\alpha\underline \TE}\exp(-\frac{\gamma}{m}t)\int_0^t\exp(\frac{\gamma}{m}s)\EE[\CH(s)]ds \nn\\
	&\quad \geq \mc{Y}'(0)\exp(-\frac{\gamma}{m}t) \nn\\
	&\qquad -4\left(\alpha +\frac{\alpha\lambda}{m} 2(\frac{2m}{\gamma})^2\right)c_\TE e^{-2\alpha\underline \TE}\EE[\CH(0)]\exp(-\frac{\gamma}{m}t)\int_0^t\exp((\frac{\gamma}{m}-\chi)s)ds
	\nn\\ 
	&\quad \geq \mc{Y}'(0)\exp(-\frac{\gamma}{m}t) \nn\\
	&\qquad -4\left(\alpha +\frac{\alpha\lambda}{m} 2(\frac{2m}{\gamma})^2\right)c_\TE e^{-2\alpha\underline \TE}\frac{\EE[\CH(0)]}{\frac{\gamma}{m}-\chi} \exp(-\chi t)\,,
	\end{align*}
	which implies that
	\begin{align*}
	\mc{Y}(t)\geq & \ \mc{Y}(0)- \frac{m}{\gamma} (-\mc{Y}'(0))_+ \nn\\ &-4\left(\alpha +\frac{\alpha\lambda}{m} 2(\frac{2m}{\gamma})^2\right)c_\TE e^{-2\alpha\underline \TE}\frac{\EE[\CH(0)]}{\chi(\frac{\gamma}{m}-\chi)}\,.
	\end{align*}
	By assumption \eqref{asuE}, this means that
	\begin{align*}
	&(\EE[e^{-\alpha\TE(\OX_t)}])^2\nn\\
	&\quad \geq  (\EE[e^{-\alpha\TE(\OX_0)}])^2-2\alpha\frac{\gamma}{m}\EE[e^{-\alpha\TE(\OX_0)}]\left(\EE[e^{-\alpha\TE(\OX_0)}\nabla \TE(\OX_0)\cdot \OV_0]\right)_+\nn\\
	&\qquad -4\left(\alpha +\frac{\alpha\lambda}{m} 2(\frac{2m}{\gamma})^2\right)c_\TE e^{-2\alpha\underline \TE}\frac{\EE[\CH(0)]}{\chi(\frac{\gamma}{m}-\chi)}\nn\\
	&\quad \geq  \frac{1}{4} (\EE[e^{-\alpha\TE(\OX_0)}])^2\,.
	\end{align*}
	This means that there exists $\delta>0$ such that $\EE[e^{-\alpha\TE(\OX_t)}]\geq \frac{1}{2}\EE[e^{-\alpha\TE(\OX_0)}]$ in $[T,T+\delta)$ as well. This then contradicts with the definition of $T$. Hence $T=\infty$. Consequently it holds that
	\begin{equation}\label{decay}
	\EE[\CH(t)]\leq \EE[\CH(0)]\exp(-\chi t) \mbox{ and }\EE[e^{-\alpha\TE(\OX_t)}]\geq\frac{1}{2}\EE[e^{-\alpha\TE(\OX_0)}],
	\end{equation}
	for all $t\geq 0$. Recalling the fact \eqref{fact} this infers that
	\begin{align}\label{esXa}
	&\EE[|\OX_t-X^\alpha(\rho_t)|^2]\leq 2e^{-\alpha\underline \TE}\frac{\EE[|\delta\OX_t|^2]}{\EE[e^{-\alpha\TE(\OX_t)}]}\nn\\
	&\qquad \leq 4e^{-\alpha\underline \TE}(\frac{2m}{\gamma})^2\frac{2\EE[\CH(0)]}{\EE[e^{-\alpha\TE(\OX_0)}]}\exp(-\chi t).
	\end{align}
	Additionally, one has
	\begin{align}\label{esE}
	&\EE[|\OX_t-\EE[\OX_t]|^2]\leq  2(\frac{2m}{\gamma})^2\EE[\CH(t)] \nn\\
	&\quad \leq  C\exp(-\chi t)\EE[|\OV_t|^2] \nn\\
	&\quad \leq  2\EE[\CH(t)] \leq  C\exp(-\chi t)\,.
	\end{align}
	Moreover we have
	\begin{equation}
	|\frac{d}{dt}\EE[\OX_t]|\leq \EE[|\OV_t|]\leq C\exp(-\frac{1}{2}\chi t)\to 0\mbox{ as }t\to\infty\,.
	\end{equation}
	This means that $\EE[\OX_t]\to \tilde x$ for some $\tilde x$ depending on $\alpha$, then it  follows from \eqref{esE} that $\OX_t\to \tilde x$ in mean square. Thus we have
	$X^\alpha(\rho_t)\to \tilde x$ according to \eqref{esXa}. Furthermore, by \eqref{decay} one as $\frac{1}{2}\EE[e^{-\alpha\TE(\OX_0)}]\leq \EE[e^{-\alpha\TE(\OX_t)}]\to e^{-\alpha\TE(\tilde x)}$. Therefore we conclude that
	\begin{equation}
	\TE(\tilde x)\leq \frac{1}{\alpha}\log(2)-\frac{1}{\alpha}\log\left(\EE[e^{-\alpha\TE(\OX_0)}]\right).
	\end{equation}
	By the Laplace principle \eqref{Laplace}, one has
	\begin{equation}
	0\leq \TE(\tilde x)-\underline{\TE}\leq  \frac{1}{\alpha}\log(2)-\frac{1}{\alpha}\log(\EE[e^{-\alpha\TE(\OX_0)}])-\underline{\TE} \to 0\mbox{ as }\alpha \to \infty\,.
	\end{equation}
	This completes the proof.
\end{proof}

\section{Numerical examples}
In this section, we illustrate through various numerical examples the previous theoretical analysis, i.e., the mean-field limit and the small inertial limit, and analyze the performance of SD-PSO-based methods against various prototype global optimization functions. We refer to \cite{carrillo2019consensus,fhps20-2,chen2020consensusbased,benfenati2021binary} for applications of CBO and related methods to high dimensional problems in machine learning.

\paragraph{The SD-PSO algorithm.}
First we introduce the time discrete versions of the SD-PSO systems \cite{Platen}.
The particle system \eqref{eq:psoc} is solved by the \emph{discrete PSO method without local best}
\coloredeq{eq:psoiDiscr}{
\begin{split}
X^{n+1}_i &= X^{n}_i + \Delta t \ V^{n+1}_i,\\
\iw V^{n+1}_i &= \iw V^{n}_i - \gamma \Delta t \ V_i^{n+1} + \lambda \Delta t  \left(\Q_\alpha^n-X^n_i\right)\\
&\quad  + \sigma \sqrt{\Delta t} \ D(\Q_\alpha^n-X^n_i) \ \theta^n_{i},
\end{split}}
where $\theta_{i} \sim \mathcal{N}(0,1)$ and the last equation can be rewritten as
\coloredeq{eq:veld}{
\begin{split}
V^{n+1}_i &= \left( \frac{m}{m + \gamma \ \Delta t}\right) V^{n}_i + \frac{\lambda \ \Delta t }{m + \gamma \ \Delta t}\left(\Q_{\alpha}^n-X^n_i\right)\\
&\quad + \frac{\sigma \ \sqrt{\Delta t} }{m + \gamma \ \Delta t} D(\Q_{\alpha}^n-X^n_i) \ \theta^n_{i}.
\end{split}}
In the general case, the SD-PSO system \eqref{eq:psocir} is solved by the \emph{discrete PSO method with local best}
\coloredeq{eq:psoDiscr}{
\begin{split}
X^{n+1}_i &= X^{n}_i + \Delta t \ V^{n+1}_i ,\\
V^{n+1}_i &= \left( \frac{m}{m + \gamma \ \Delta t}\right) V^{n}_i + \frac{\lambda_1 \ \Delta t }{m + \gamma \ \Delta t}\left(\P^n_i-X^n_i\right)\\
&+ \frac{\lambda_2 \ \Delta t }{m + \gamma \ \Delta t}\left(\G_\alpha^n-X^n_i\right)
+ \frac{\sigma_1 \ \sqrt{\Delta t} }{m + \gamma \ \Delta t} D(\P^n_i-X^n_i) \ \theta^n_{1,i}\\
& + \frac{\sigma_2 \ \sqrt{\Delta t} }{m + \gamma \ \Delta t} D(\G_\alpha^n-X^n_i) \ \theta^n_{2,i},\\
\P^{n+1}_i &= \P^{n}_i + \nu \ \Delta t \left(X^{n+1}_i-\P^{n}_i\right)S^\beta(X^{n+1}_i,\P^{n}_i), 
\end{split}}
where $\theta_{1,i}$, $\theta_{2,i} \sim \mathcal{N}(0,1)$. 
\begin{remark} Note that, the numerical scheme \eqref{eq:psoDiscr} using uniform noise becomes equivalent to the PSO algorithm \eqref{eq:psoi} under assumptions \eqref{eq:param} for $ \Delta t = 1 $, $\nu=0.5$, and taking the limit $\alpha$, $\beta \rightarrow \infty$ so that $\P_i^n$, $\G_{\alpha}^n$ match the local and global best definitions in \eqref{eq:psoEvolution}. In addition, in the limit $m\to 0^+$ scheme \eqref{eq:psoDiscr} is consistent with the zero-inertia limit \eqref{eq:cbozero} and reduces to the \emph{discrete CBO method with local best}
\coloredeq{eq:cboDiscr}{
\begin{split}
X^{n+1}_i &= X^{n}_i + {\lambda_1 \ \Delta t }\left(\P^n_i-X^n_i\right)
+ {\lambda_2 \ \Delta t }\left(\G_\alpha^n-X^n_i\right)\\
&+ {\sigma_1 \ \sqrt{\Delta t} } D(\P^n_i-X^n_i) \ \theta^n_{1,i}
 + {\sigma_2 \ \sqrt{\Delta t} } D(\G_\alpha^n-X^n_i) \ \theta^n_{2,i},\\
 \P^{n+1}_i &= \P^{n}_i + \nu \ \Delta t \left(X^{n+1}_i-\P^{n}_i\right)S^\beta(X^{n+1}_i,\P^{n}_i).
\end{split}}
\end{remark}

\subsection{Validation of the mean field limit}
In the following we validate numerically the mean field limit by considering as prototype functions for global optimization the Ackley function and the Rastrigin function in one dimension. The functions have multiple local minima that can easily trap the particle dynamics (see Figure \ref{Fig1}). We refer to \cite{Grassi2021PSO} for additional examples.

\begin{figure}[tb!]
\begin{minipage}{\linewidth}
\centering
\subcaptionbox{Ackley}{\includegraphics[scale=0.4]{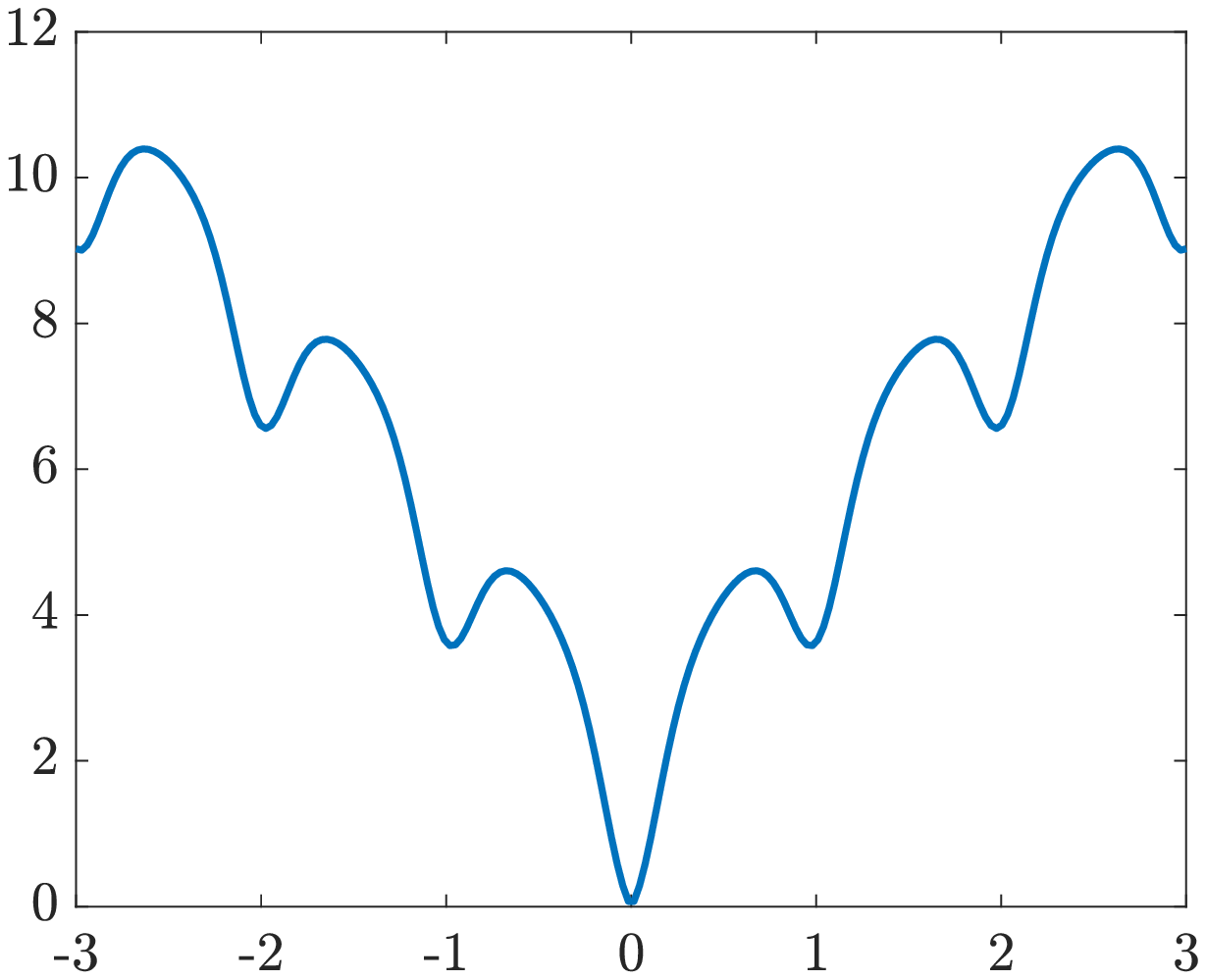}}\hskip -.5cm
\subcaptionbox{Rastrigin}{\includegraphics[scale= 0.4]{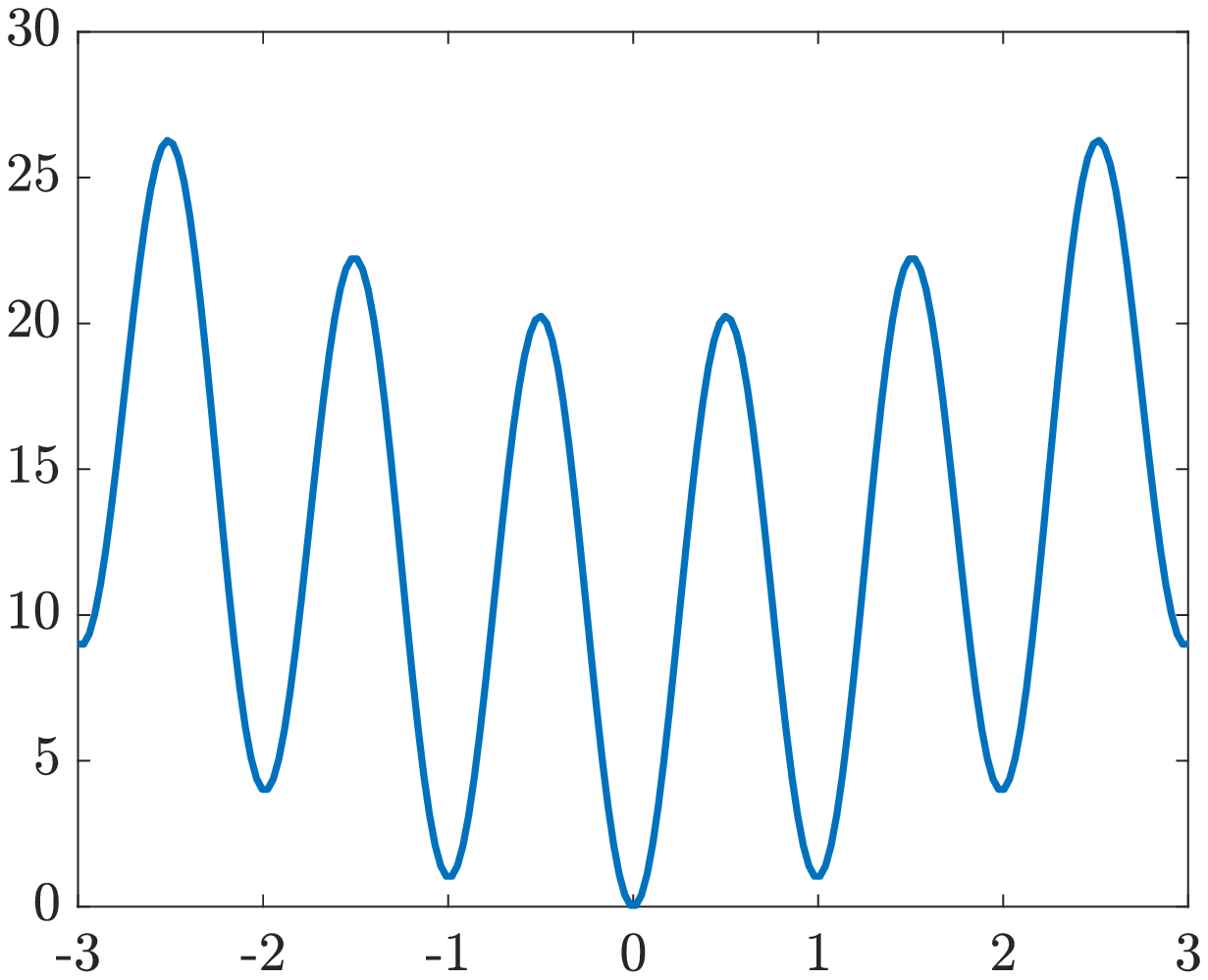}} 
\caption{One-dimensional Ackley and Rastrigin functions in the interval $\left[ -3, 3 \right]$ with global minimum in the origin.}
\label{Fig1}
\end{minipage}
\end{figure}

\paragraph{The MF-PSO solver.}
The corresponding MF-PSO equation without local best \eqref{PDEi} has been discretized using a dimensional splitting where the transport part is solved through a backward semi-Lagrangian method and the remaining Fokker-Planck term is discretized using an implicit central scheme. The MF-PSO equation with memory \eqref{PDEii} is solved by a further dimensional splitting where the additional memory term is discretized using a Lax-Wendroff method. Zero boundary conditions have been implemented outside the computational domain. We refer \cite{Grassi2021PSO,DP} for further details and additional discretizations of Vlasov-Fokker-Planck systems.


In the sequel we used $N=5 \times 10^5$ particles, a mesh size for the mean field solver of $90 \times 120$ points for $(x,v)\in [-3,3]\times [-4,4]$, and whenever present, the mesh and domain size in $y$ have been taken identical to those in $x$. To represent the particle solution, we used the probability density estimate based on a normal kernel reconstruction evaluated at equally-spaced points. In all simulations, the initial distribution is assumed to be uniform and the minimum is assumed in $x=0$.

\begin{figure}[tb]
\begin{minipage}{\linewidth}
\centering
\subcaptionbox{SD-PSO, $t = 0.5$}{\includegraphics[scale=0.25]{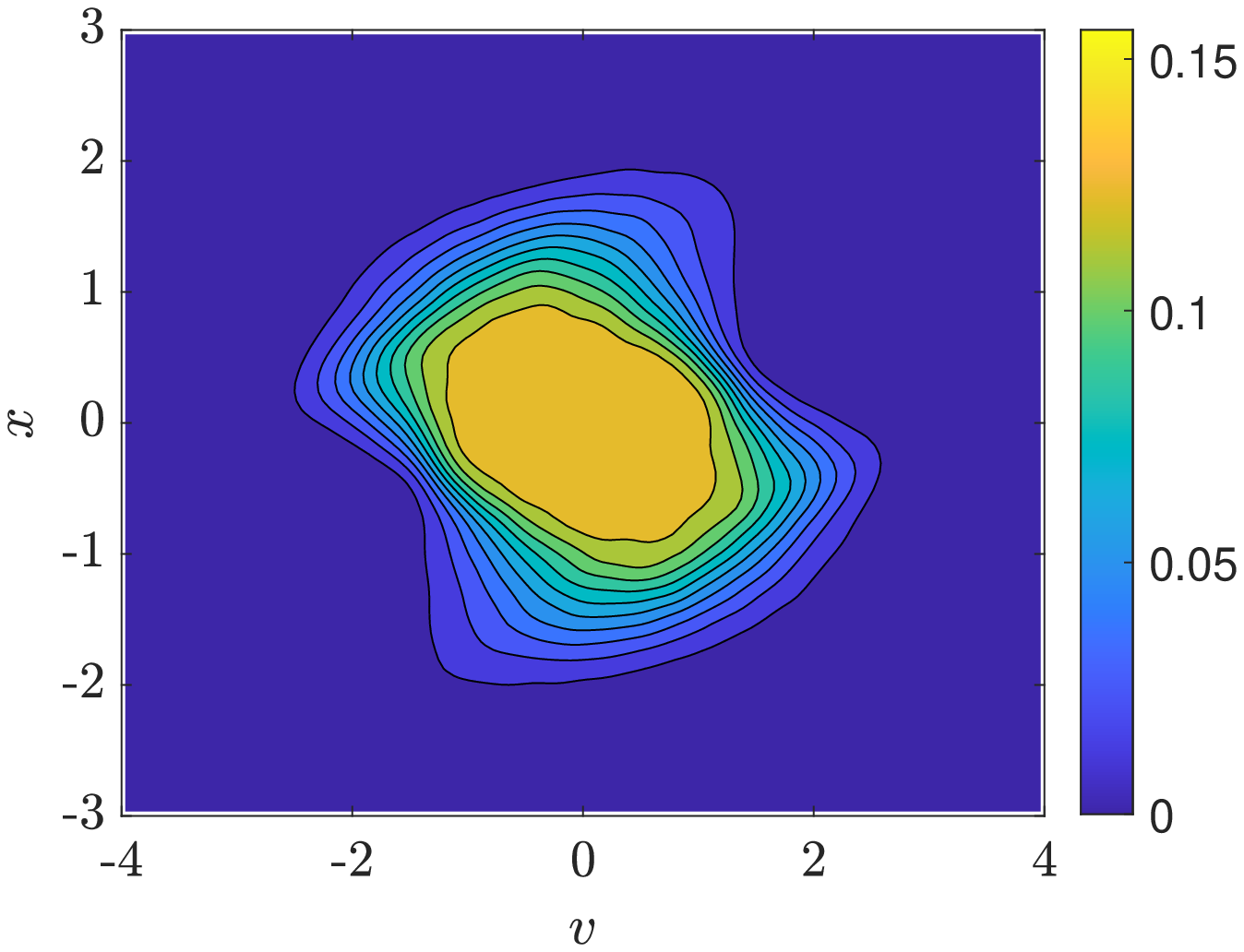}} 
\subcaptionbox{SD-PSO, $t = 1$}{\includegraphics[scale=0.25]{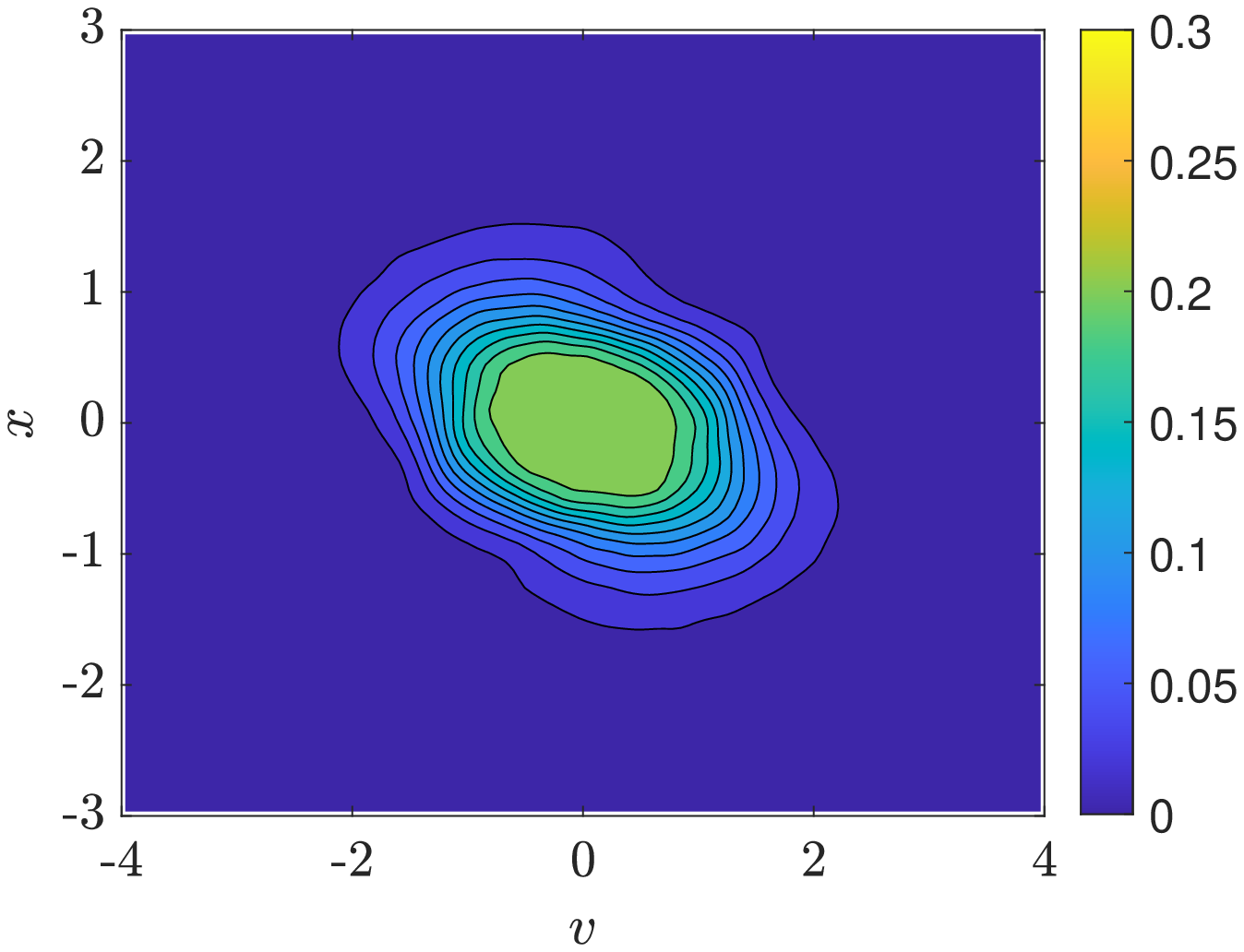}} 
\subcaptionbox{SD-PSO, $t = 3$}{\includegraphics[scale=0.25]{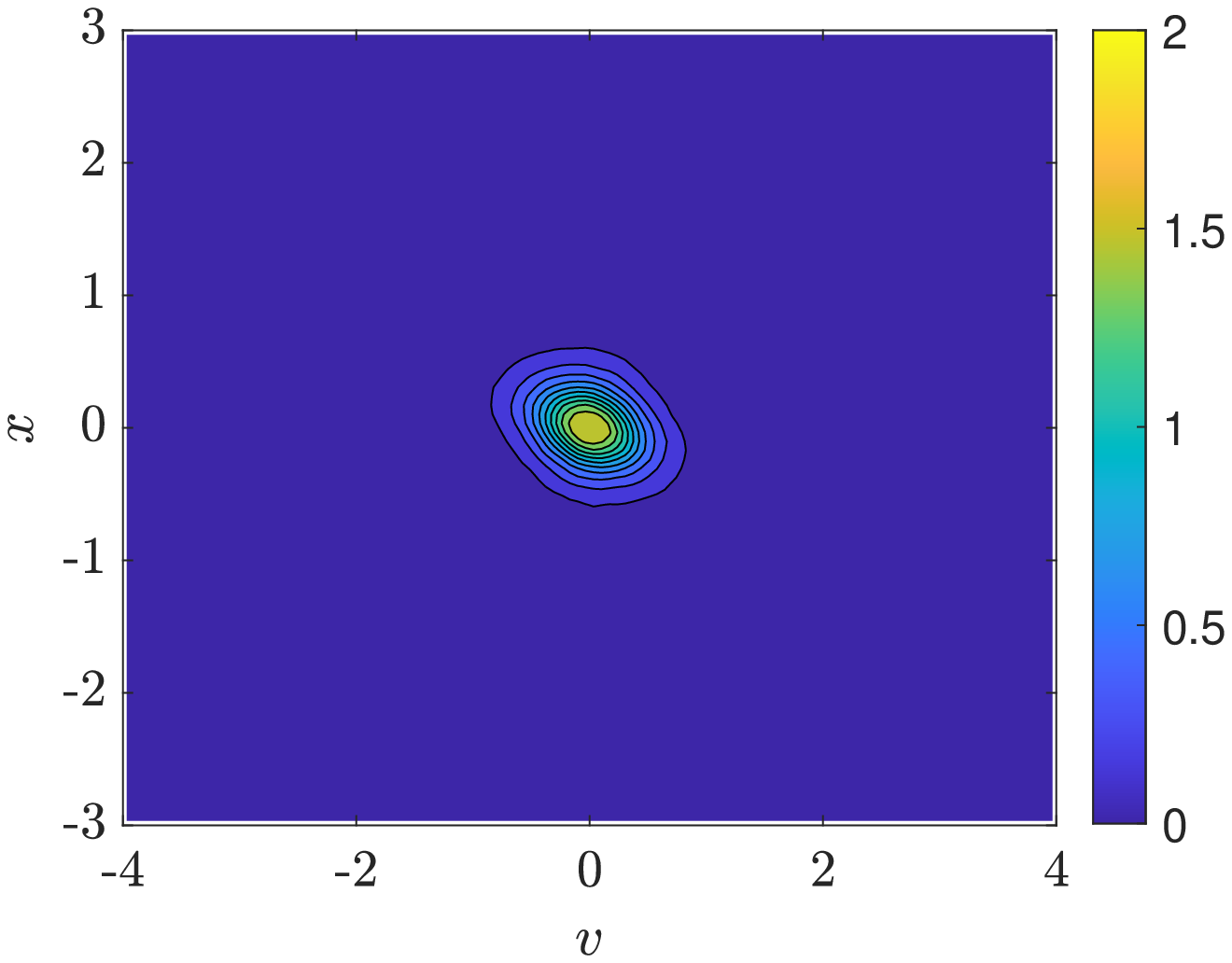}}\\
\subcaptionbox{MF-PSO, $t = 0.5$}{\includegraphics[scale=0.25]{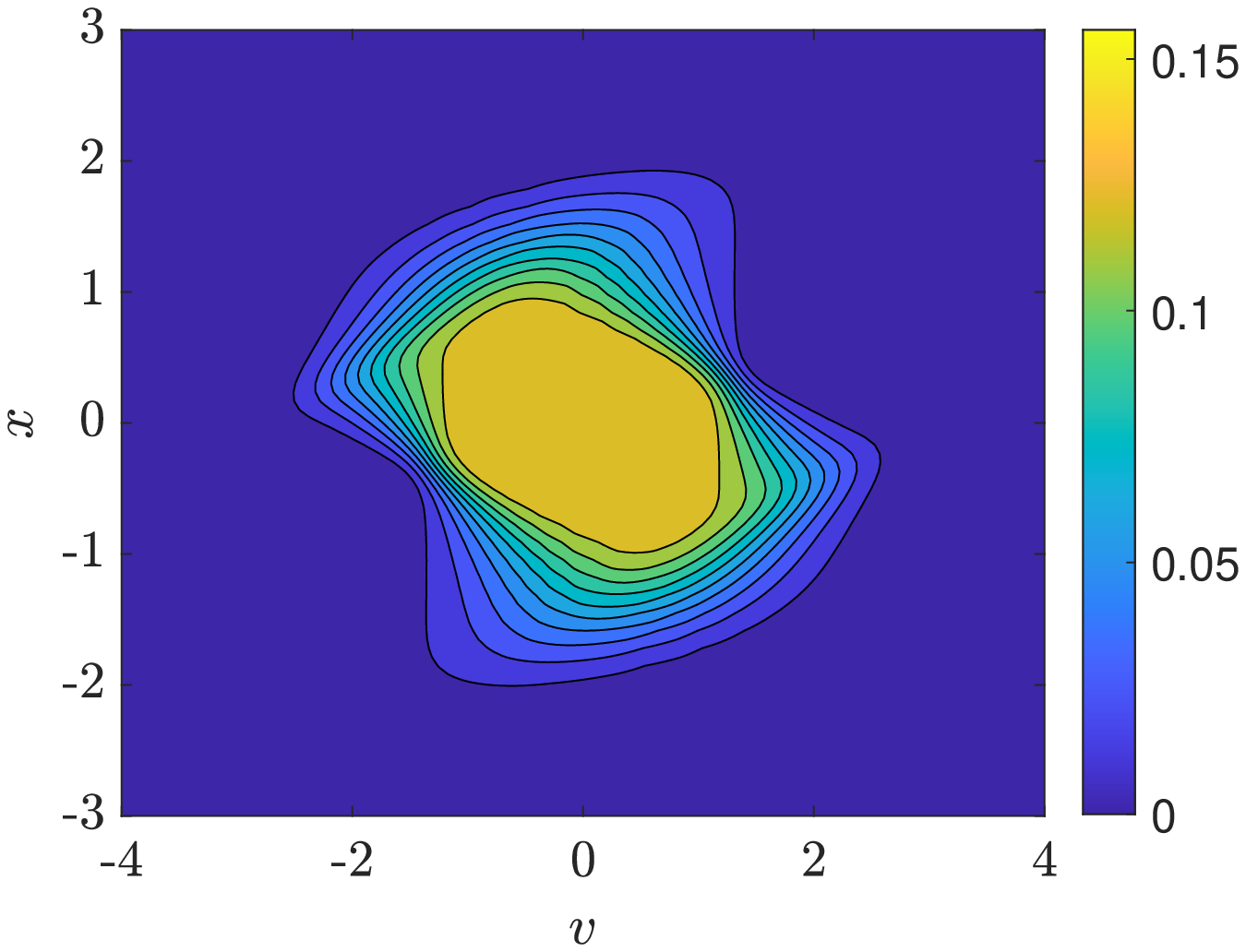}}  
\subcaptionbox{MF-PSO, $t = 1$}{\includegraphics[scale=0.25]{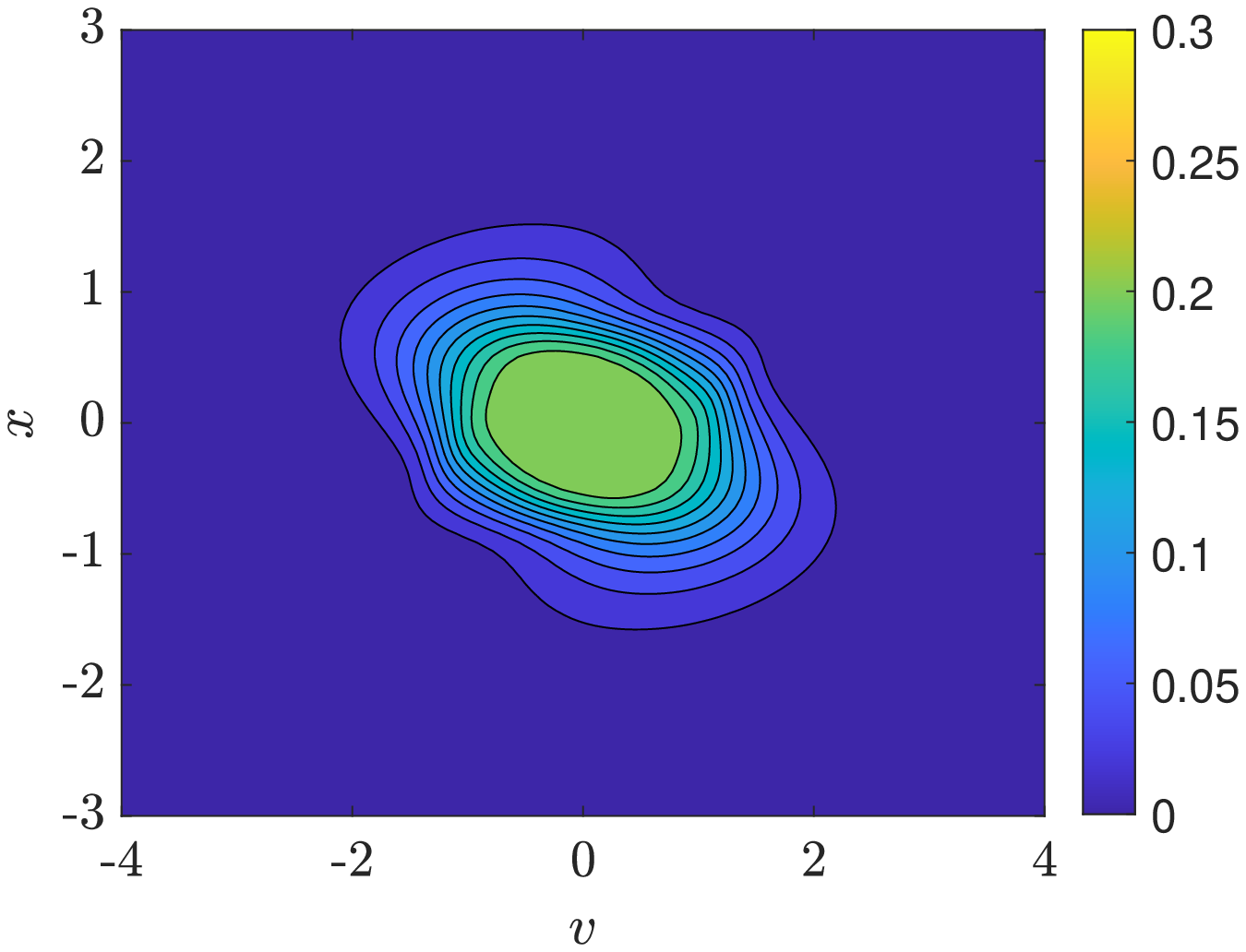}} 
\subcaptionbox{MF-PSO, $t = 3$}{\includegraphics[scale=0.25]{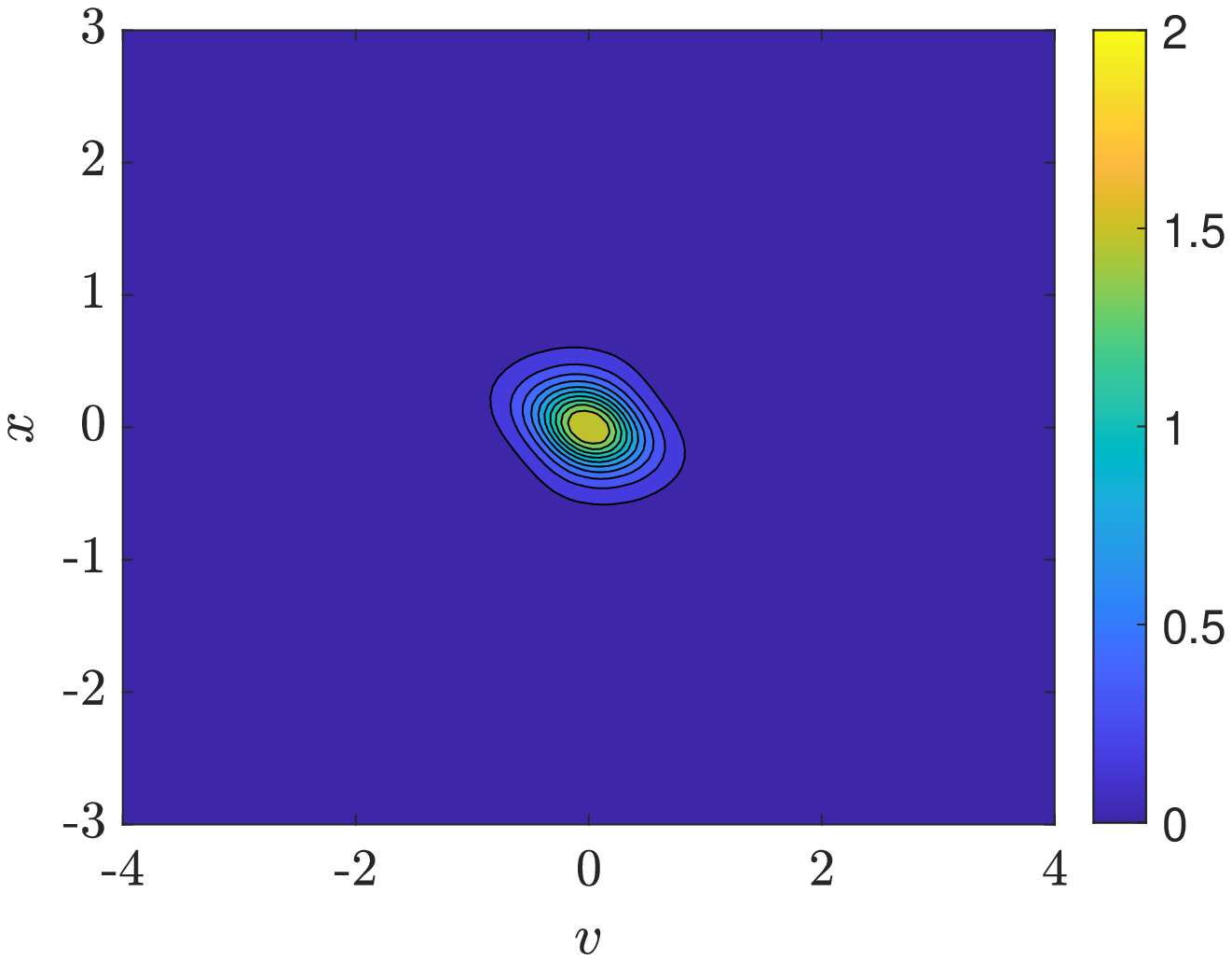}}  \\
\subcaptionbox{$\rho(x,t)$, $t = 0.5$}{\includegraphics[scale=0.25]{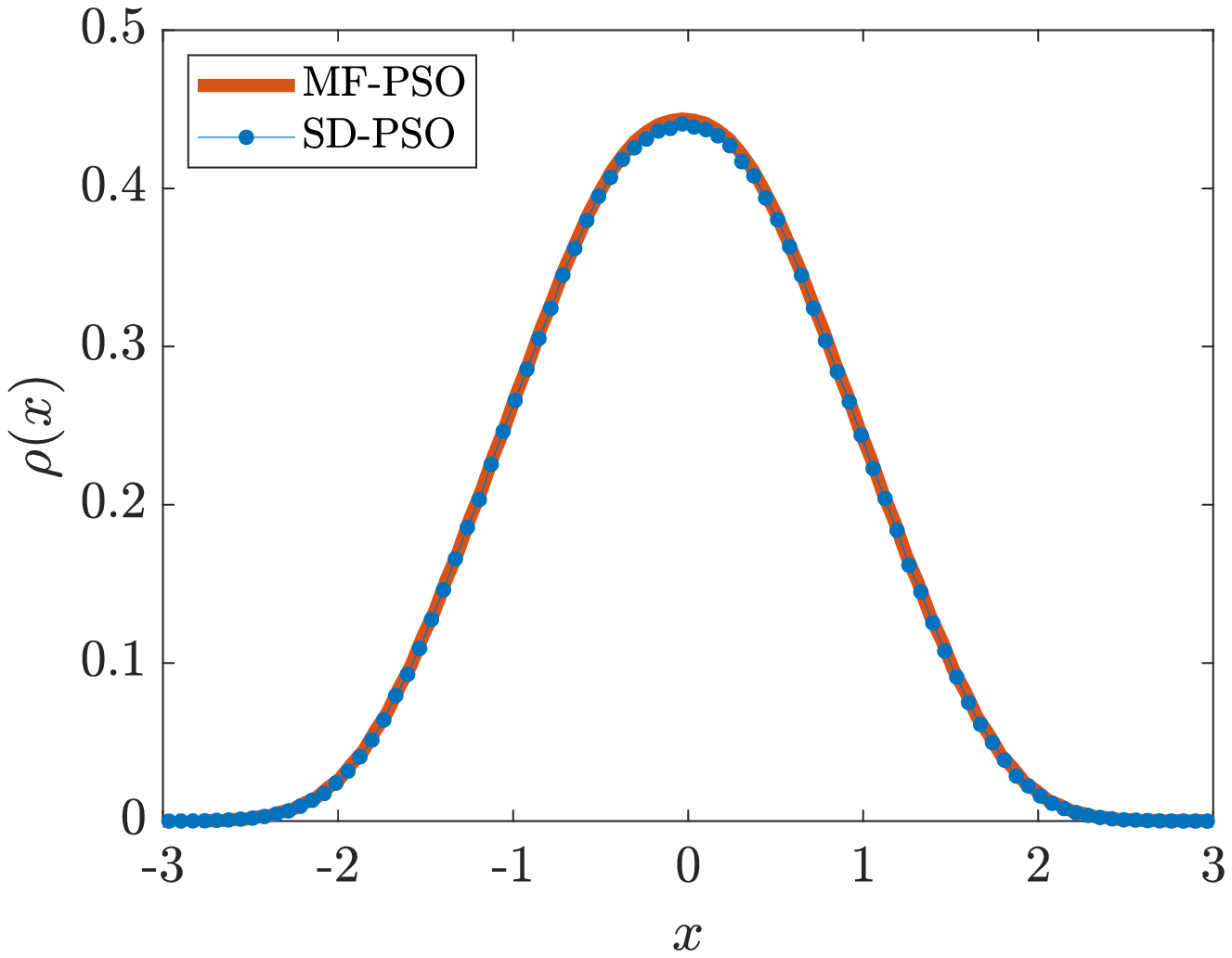}}
\subcaptionbox{$\rho(x,t)$, $t = 1$}{\includegraphics[scale=0.25]{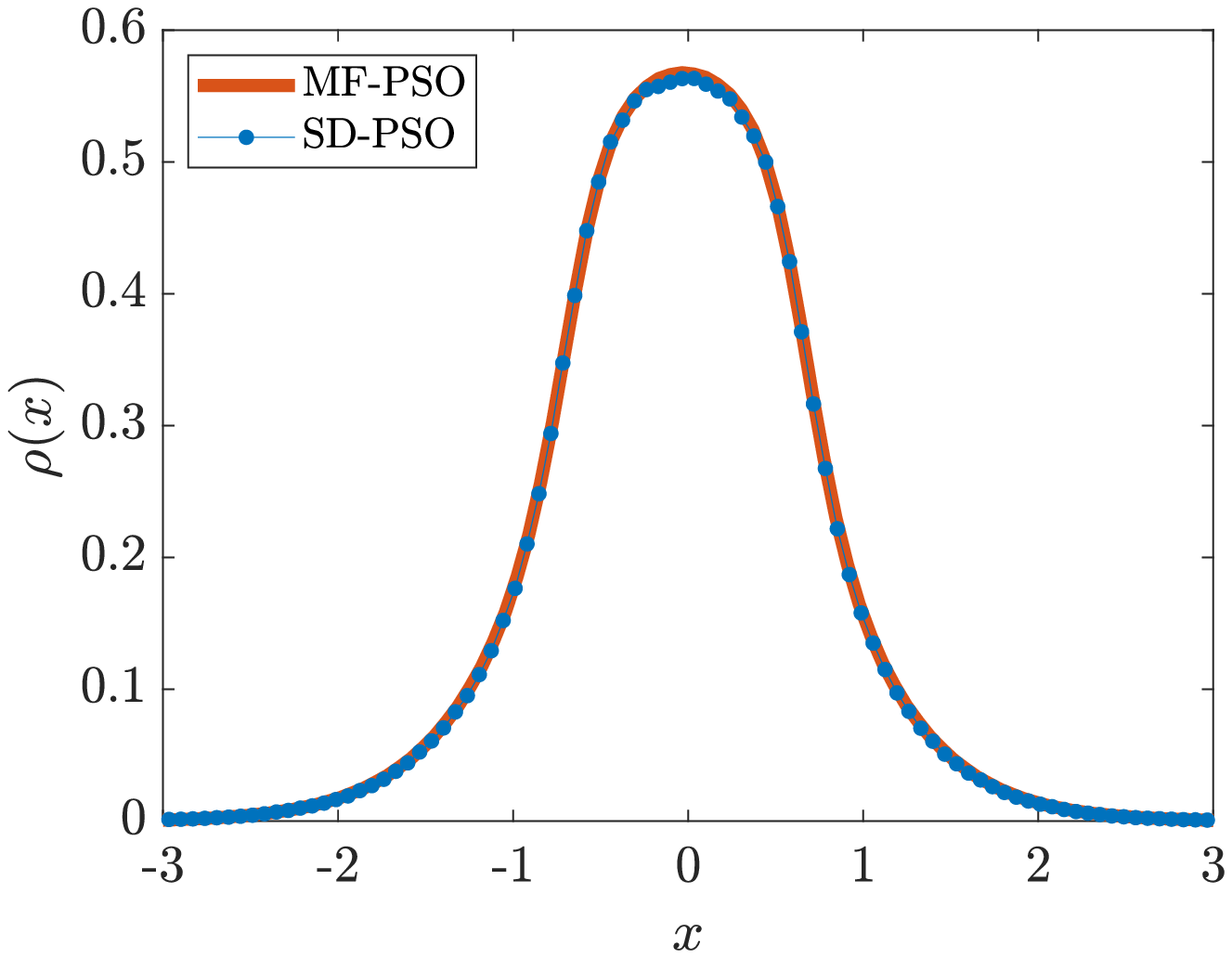}} 
\subcaptionbox{$\rho(x,t)$, $t = 3$}{\includegraphics[scale=0.25]{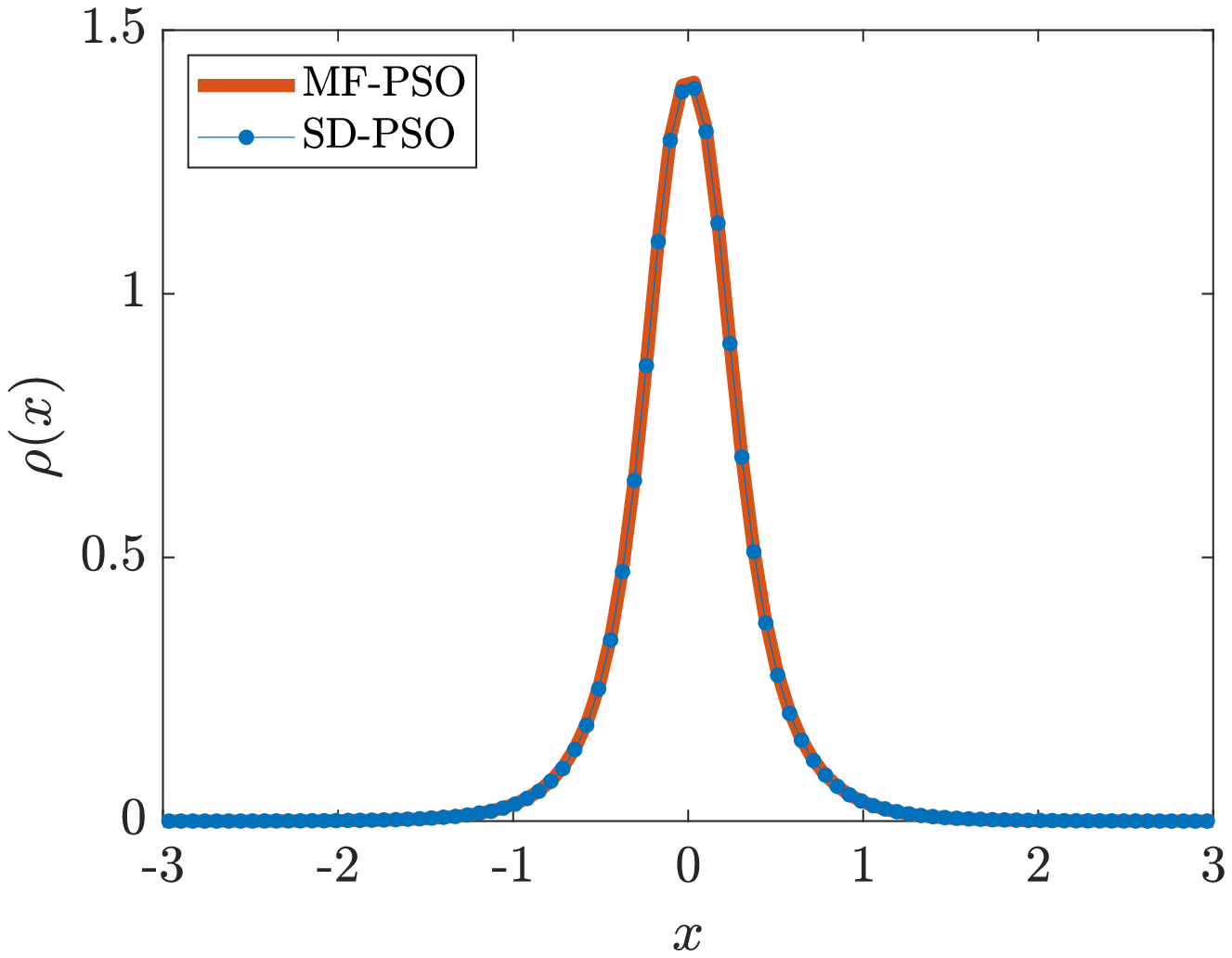}}\\
\caption{Mean field validation (no memory). Optimization of the Ackley function. First row: solution of the SD-PSO system \eqref{PSO} using $N=5 \times 10^5$ particles. Second row: solution of the MF-PSO limit \eqref{PDEi}. Third row: marginal densities.} 
\label{Fig2}
\end{minipage}
\end{figure}
\subsubsection{Absence of memory effects}
We consider the optimization process of the Ackley function. Here we report the results obtained with
\begin{equation}
\gamma = 0.5,\quad  \lambda = 1, \quad \sigma= {1}/{\sqrt{3}},\quad \alpha = 30.
\label{eq:paramt} 
\end{equation}
The values of $\lambda$ and $\sigma$ correspond to the standard PSO choice $c_k = 2 $ in \eqref{eq:param}. In Figure \ref{Fig2} we report the contour plots of the evolution, at times $ t = 0.5 $, $ t = 1 $ and $ t = 3 $, of the particle distribution computed through \eqref{eq:psoiDiscr} and by the direct discretization of the mean-field equation \eqref{PDEi} together with the evolution in time of the marginal density $\rho(x,t) = \int_{\RR^d} f(x,v,t)\,dv$.  

\subsubsection{Only local best dynamics}
In the second test case we introduce the dependence from the memory variable and compare the solutions of the discretized stochastic particle model \eqref{eq:psoDiscr} with the solver of the mean field limit \eqref{PDEii} in the case of the Rastrigin function. We assume $\lambda_2=0$ and $\sigma_2 = 0$, i.e. only the local best is present. The same parameters \eqref{eq:paramt} have been used together with $\beta = 30$ and $\nu = 0.5$ for the local best.
\begin{figure}[tb]
\begin{minipage}{\linewidth}
\centering
\subcaptionbox{SD-PSO, $t = 0.5$}{\includegraphics[scale=0.25]{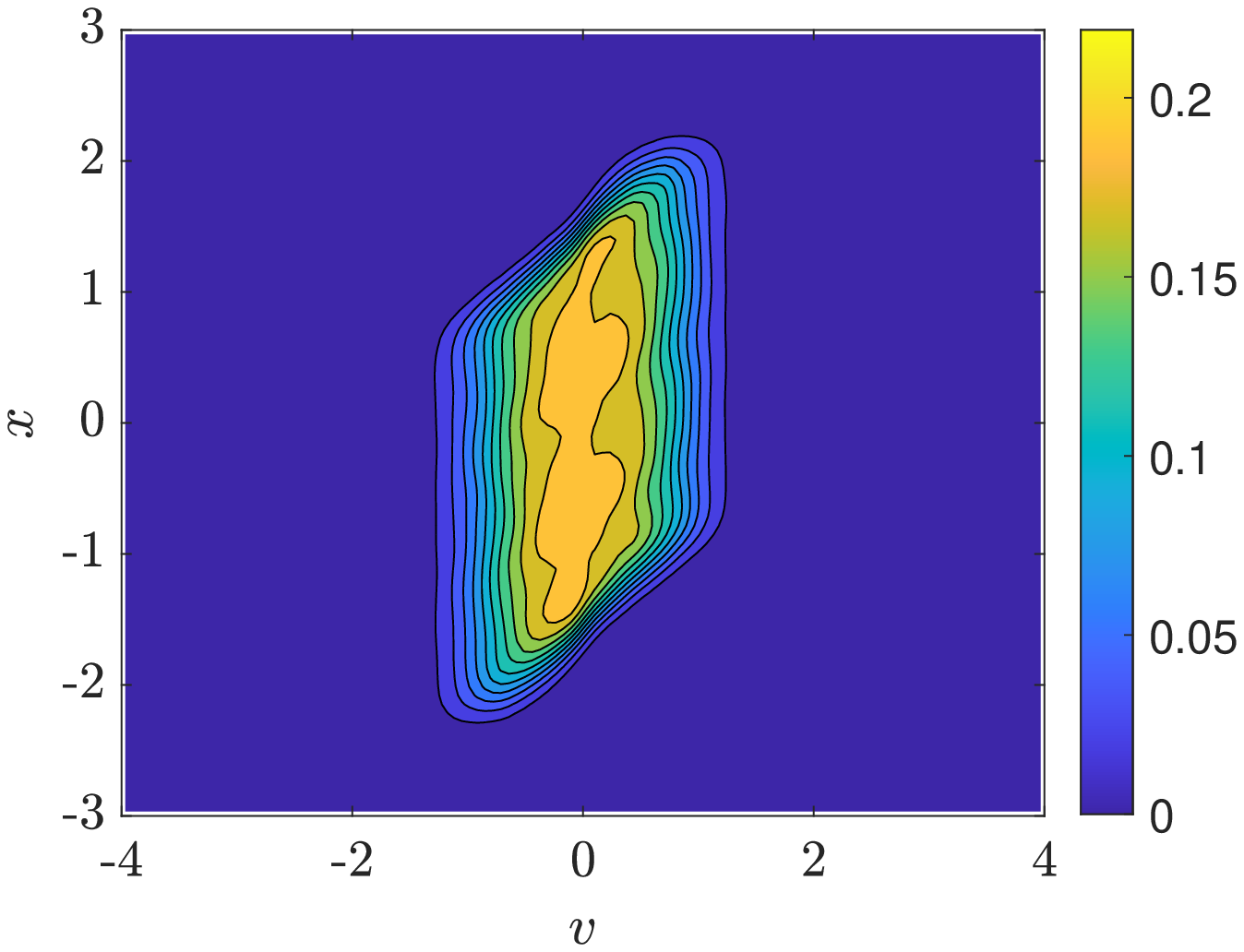}} 
\subcaptionbox{SD-PSO, $t = 3$}{\includegraphics[scale=0.25]{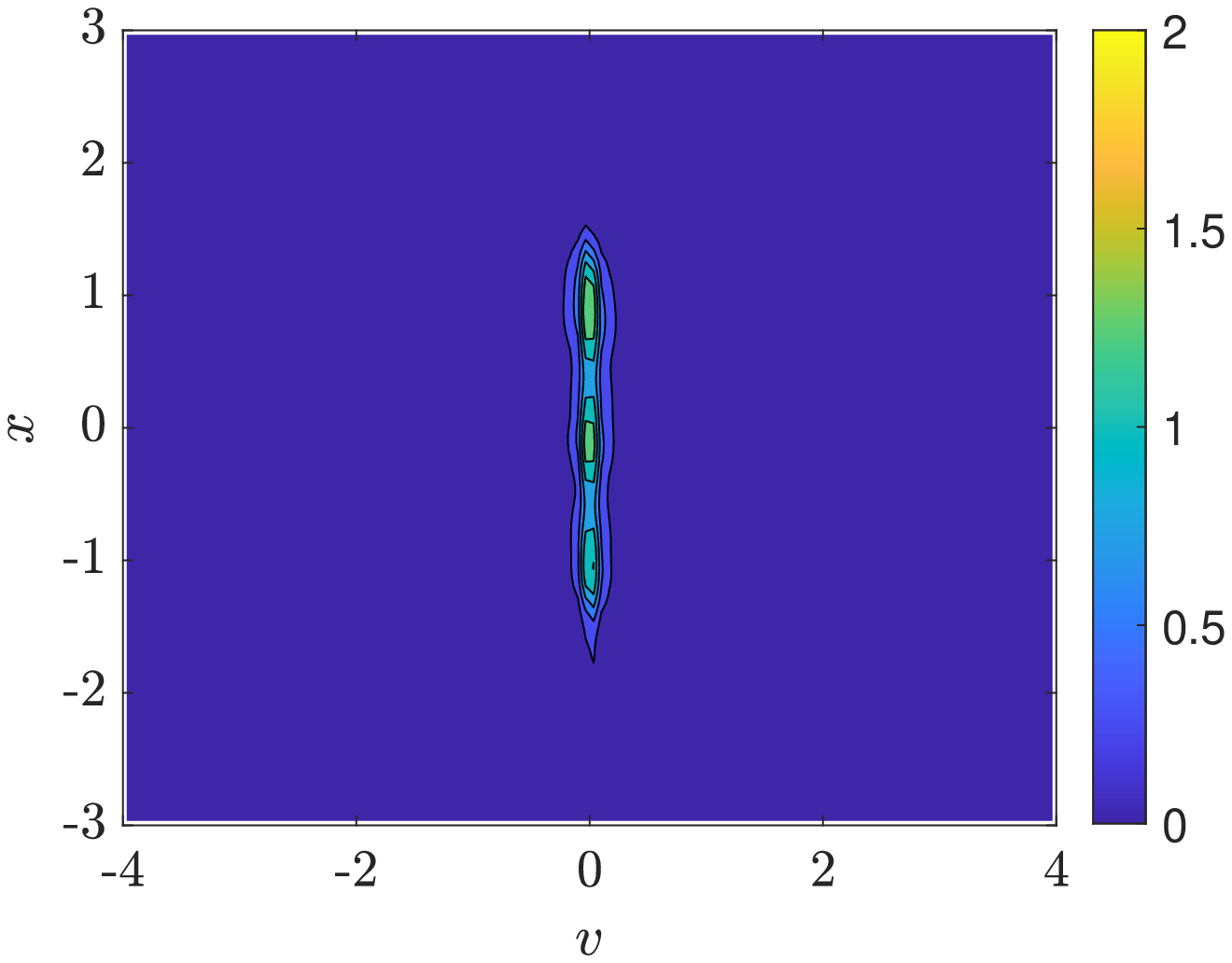}} 
\subcaptionbox{SD-PSO, $t = 6$}{\includegraphics[scale=0.25]{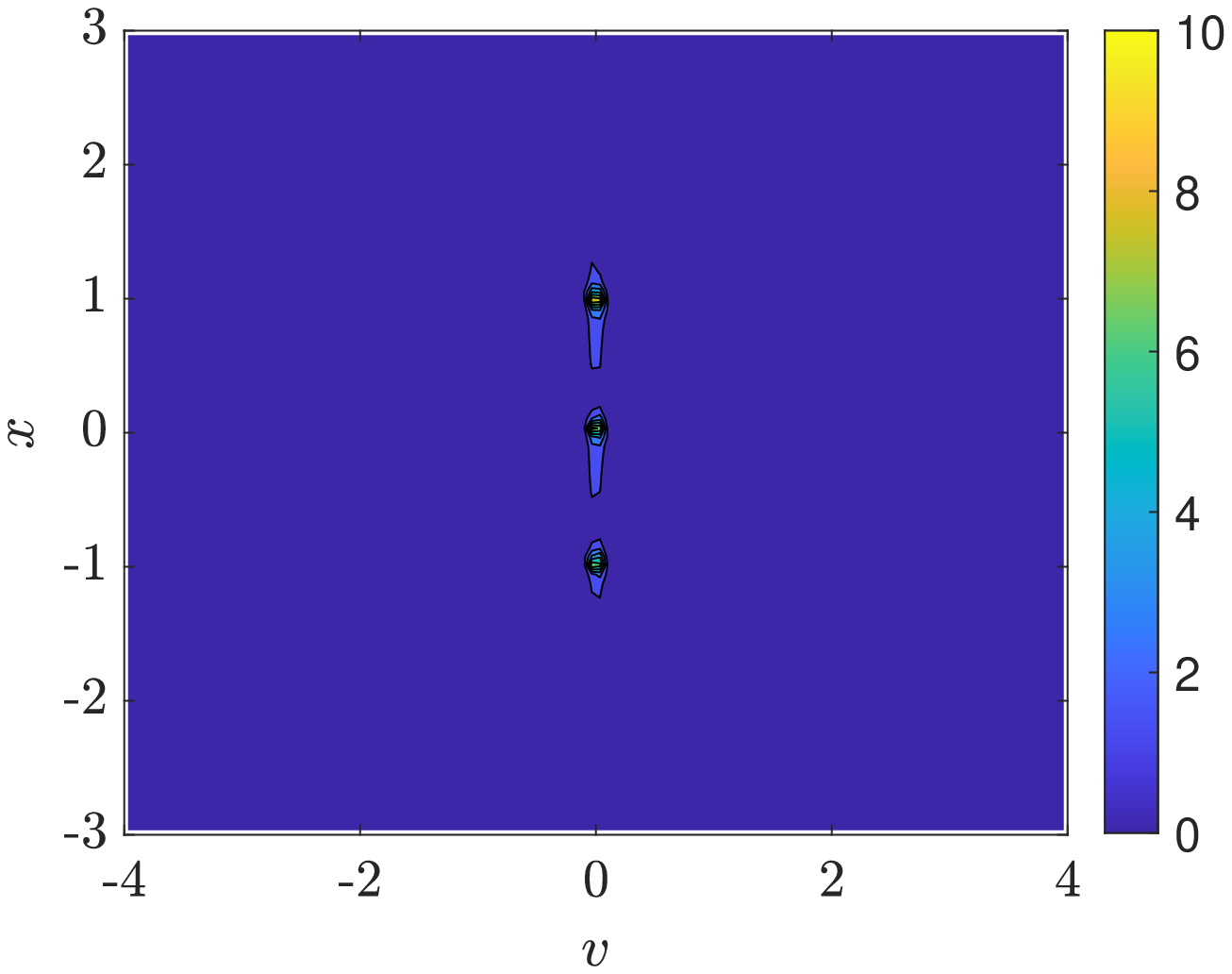}}\\
\subcaptionbox{MF-PSO, $t = 0.5$}{\includegraphics[scale=0.25]{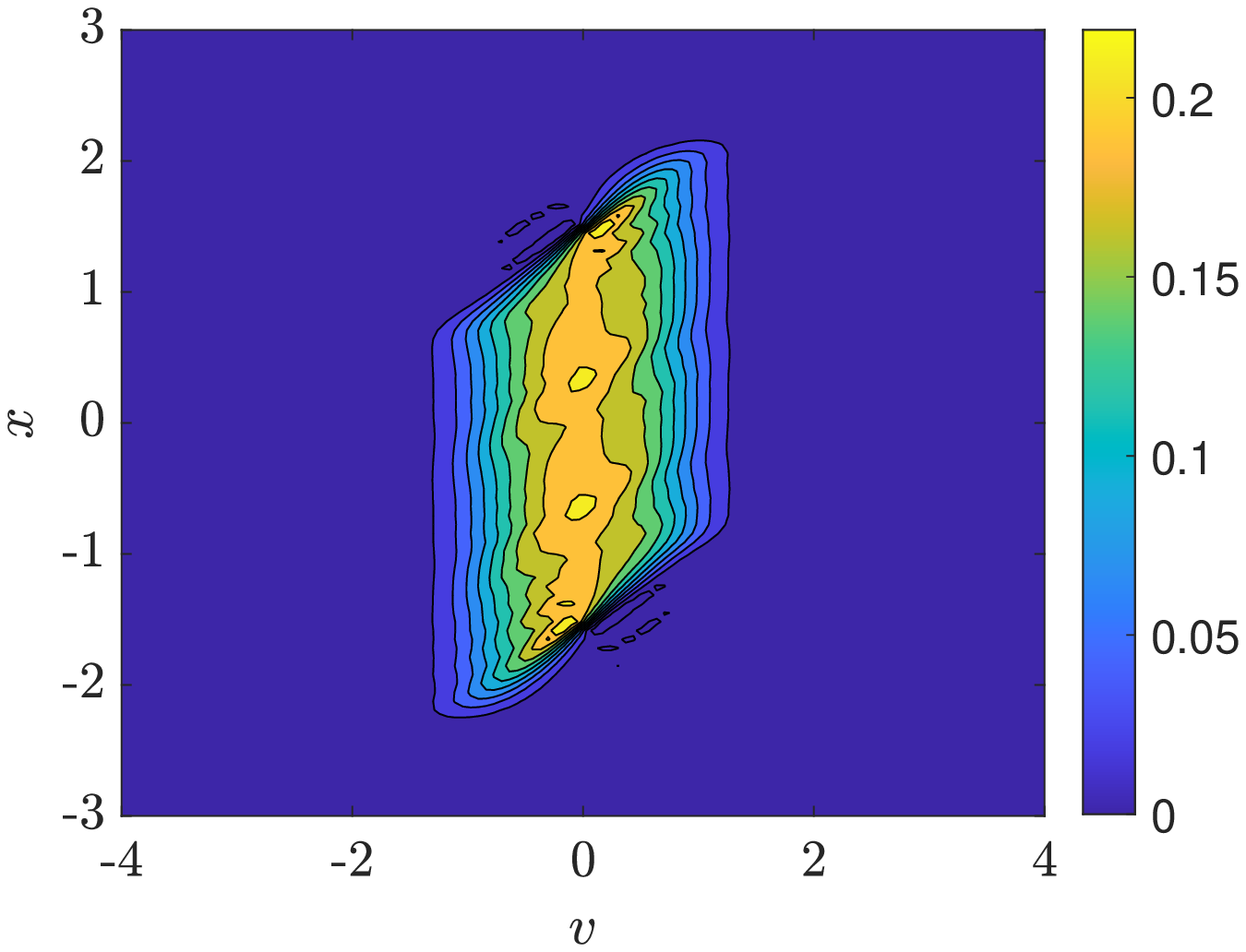}}  
\subcaptionbox{MF-PSO, $t = 3$}{\includegraphics[scale=0.25]{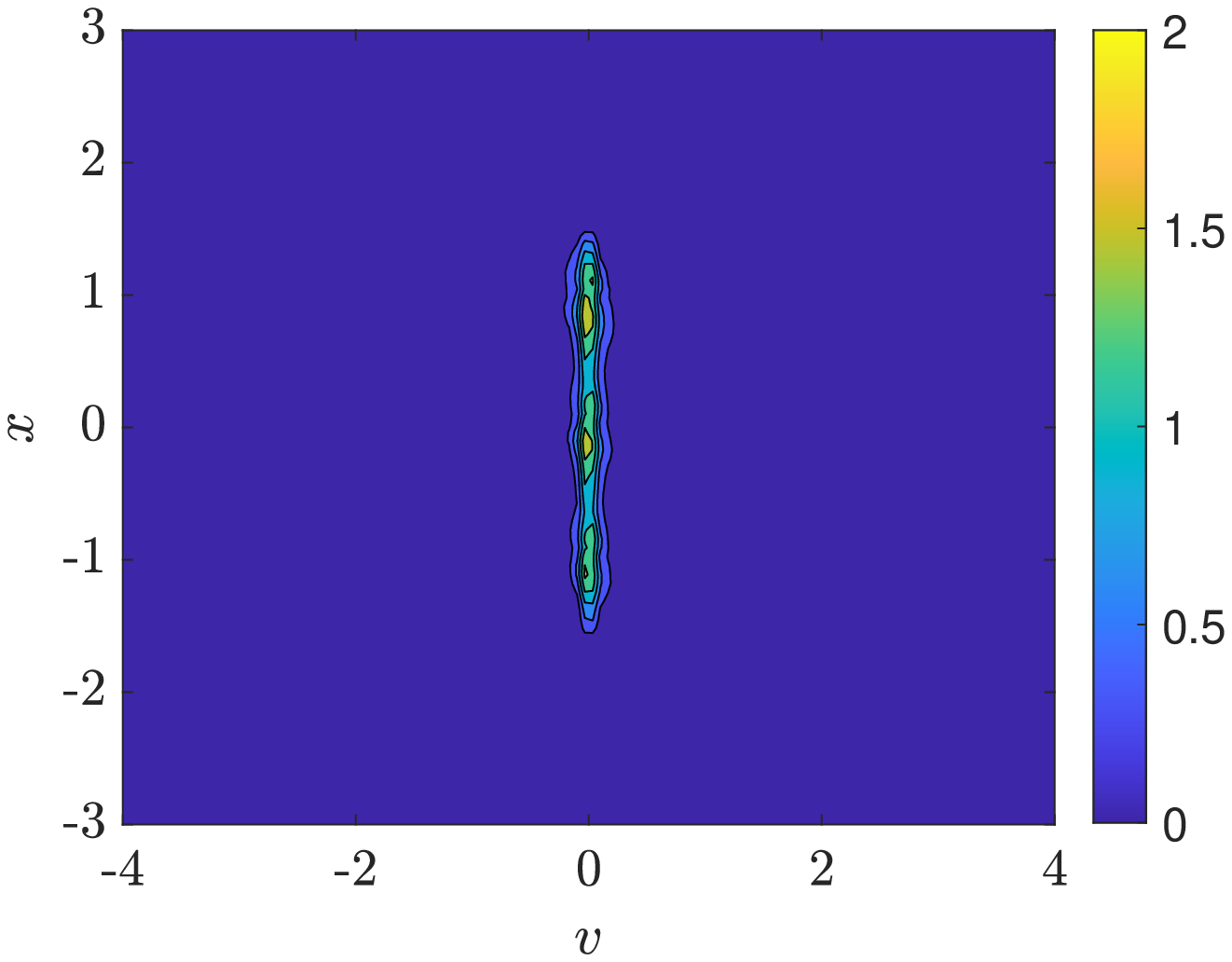}} 
\subcaptionbox{MF-PSO, $t = 6$}{\includegraphics[scale=0.25]{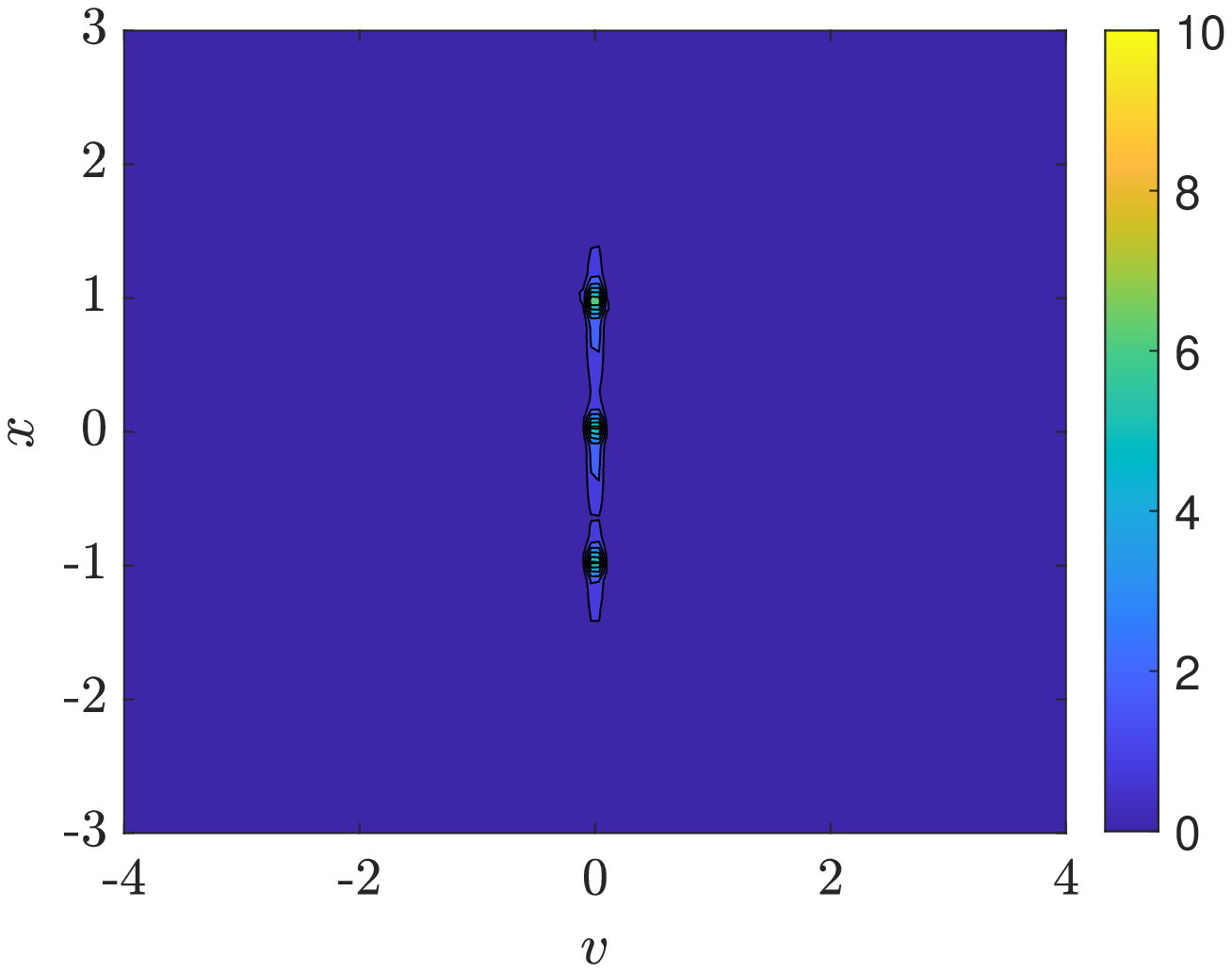}}  \\
\subcaptionbox{$\rho(x,t)$, $t = 0.5$}{\includegraphics[scale=0.25]{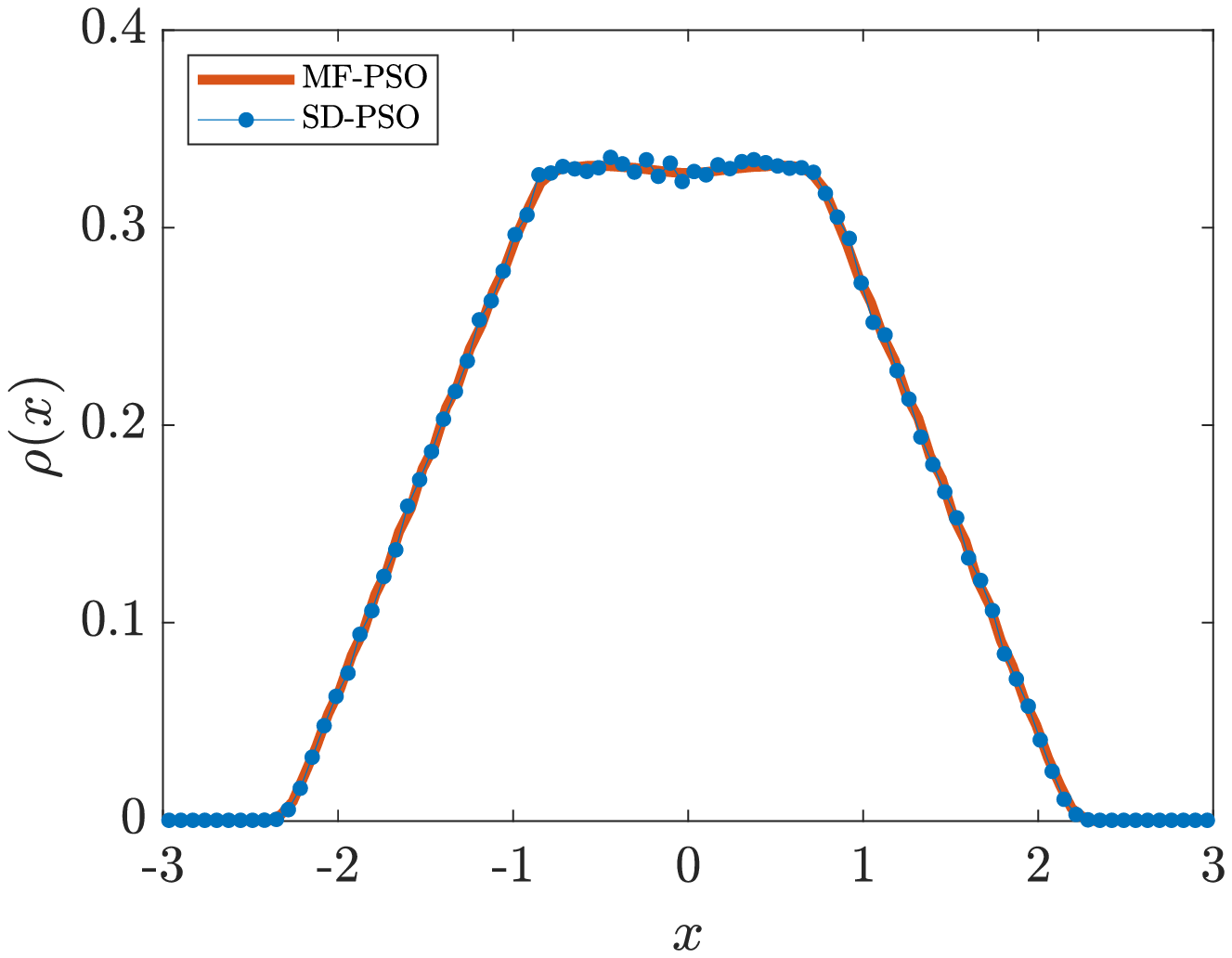}}\ 
\subcaptionbox{$\rho(x,t)$, $t = 3$}{\includegraphics[scale= 0.25]{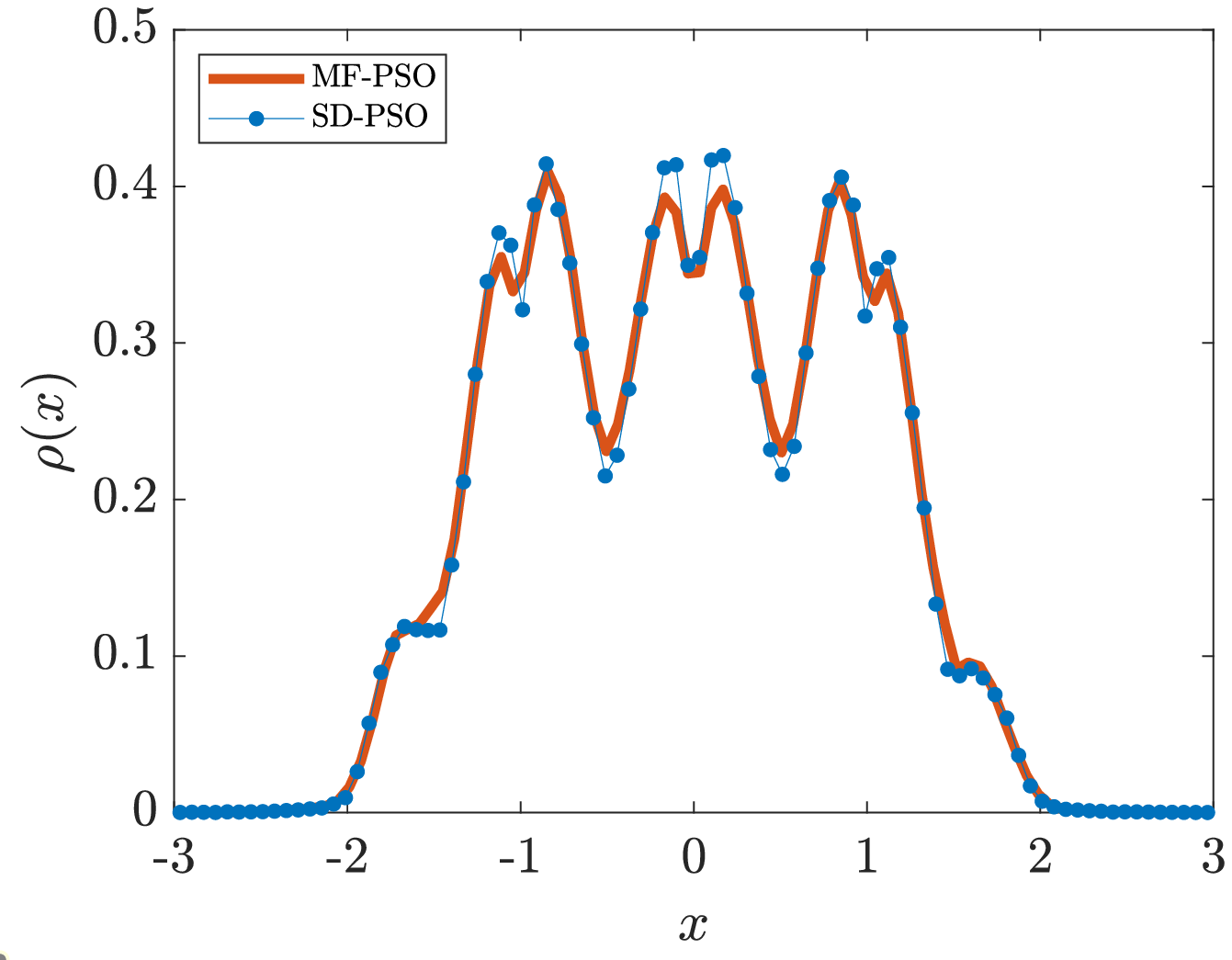}} \ 
\subcaptionbox{$\rho(x,t)$, $t = 6$}{\includegraphics[scale=0.25]{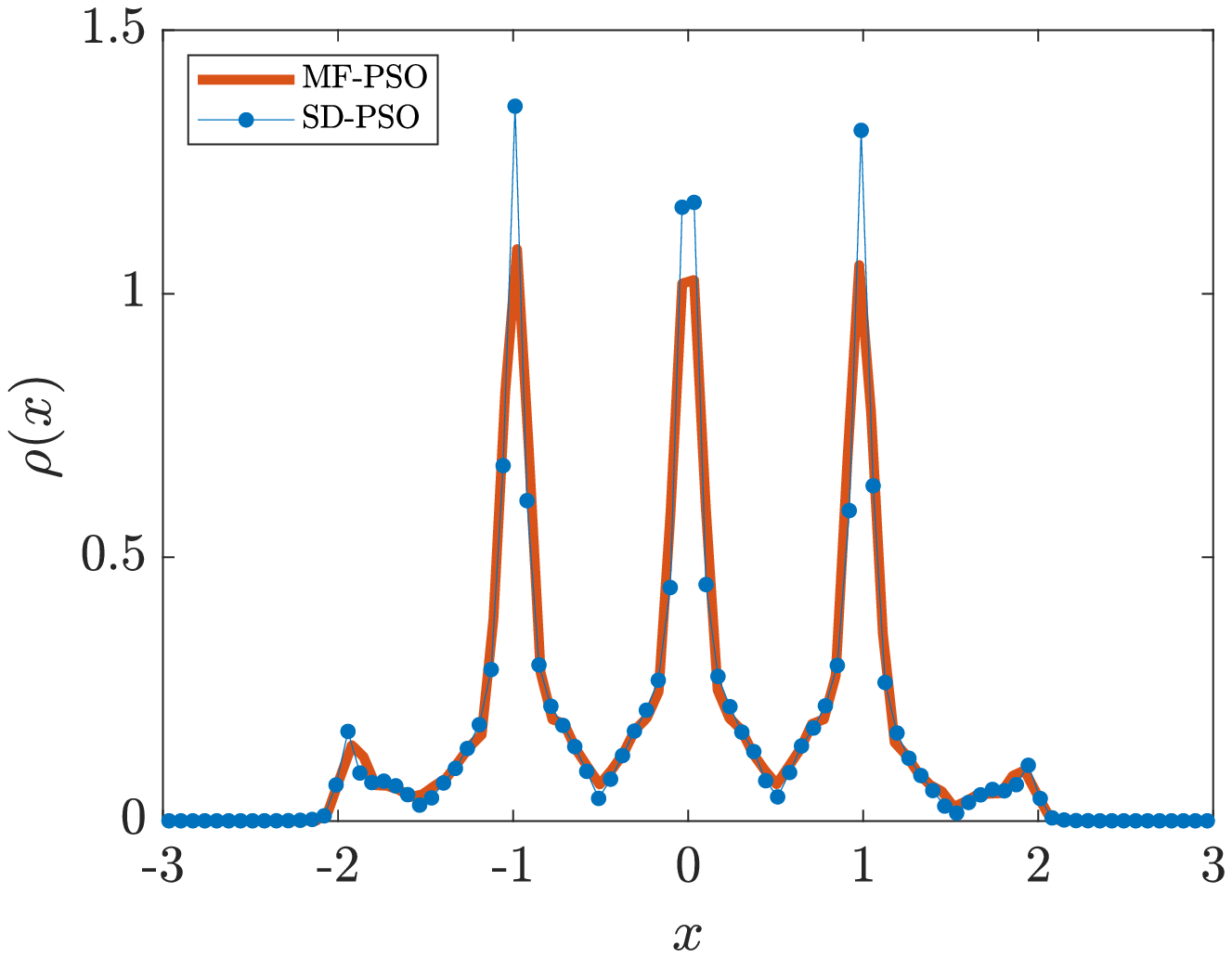}} \\
\caption{Mean field validation (local best only). Optimization of the Rastrigin function with minimum in $x=0$. First row: solution of the SD-PSO system \eqref{eq:psocir}. Second row: solution of the MF-PSO limit \eqref{PDEii}. Third row: marginal densities.} 
\label{Fig4}
\end{minipage} \\
\end{figure}
\noindent In Figure \ref{Fig4} we report the contour plot of the particle and mean-field solutions for the Rastrigin function, where now the final simulation time is $t=6$. The corresponding marginal densities are also reported.
Also in this second case, one can appreciate the good agreement between the particle and mean-field solutions.
We can note that in the presence of local best only, the particles tend to return to their local best position creating a "memory effect" that leads them to concentrate not only in the global minimum but also in the local minima. For large times we obtain a sequence of particle peaks with zero speed exactly in the positions of the local minima. Thus the dynamic allows us to identify each type of minimum present in the functions. 

\begin{figure}[htb]
\begin{minipage}{\linewidth}
\centering
\subcaptionbox{$\rho(x,t)$, $t = 0.5$}{\includegraphics[scale=0.25]{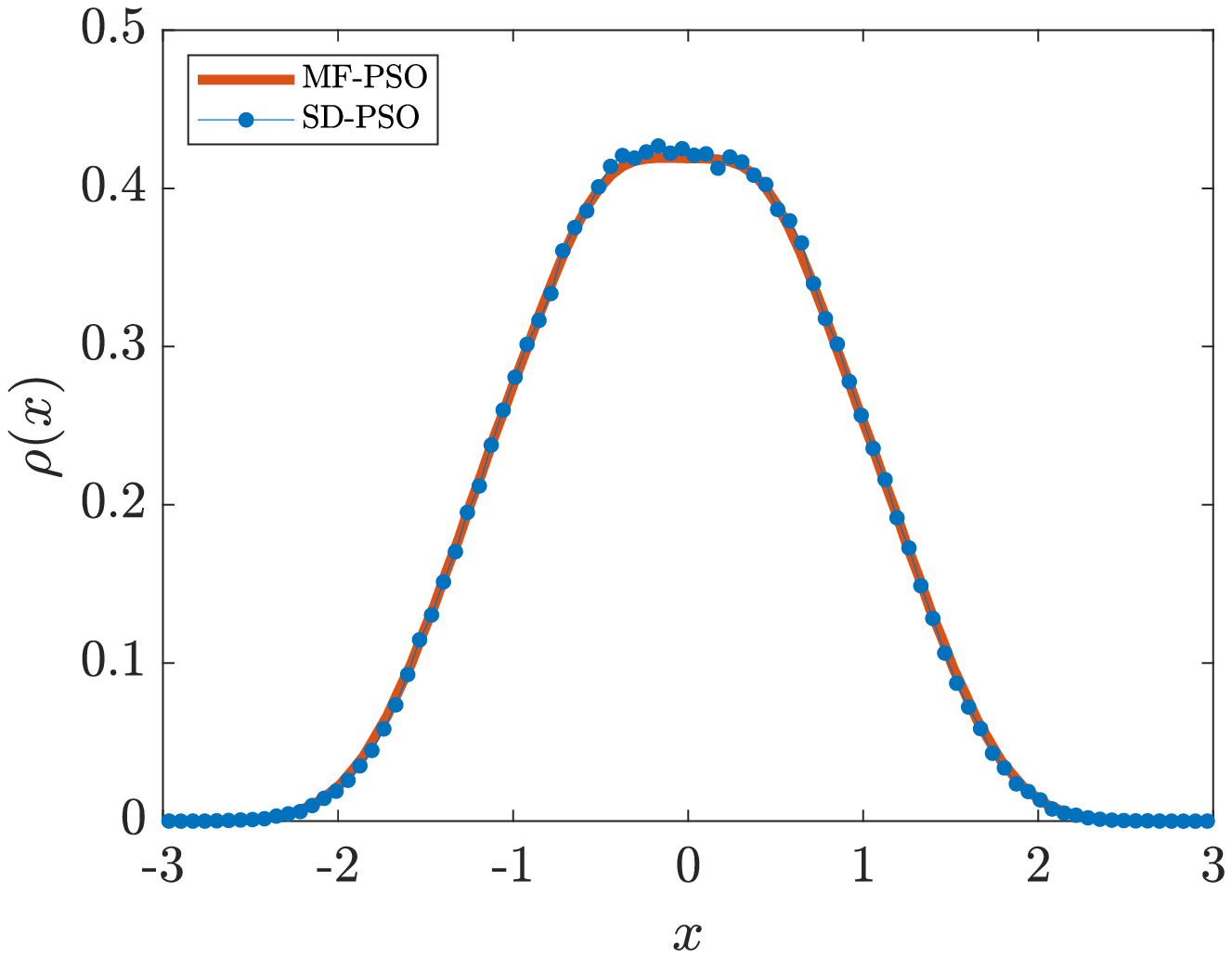}}\
\subcaptionbox{$\rho(x,t)$, $t = 1$}{\includegraphics[scale= 0.25]{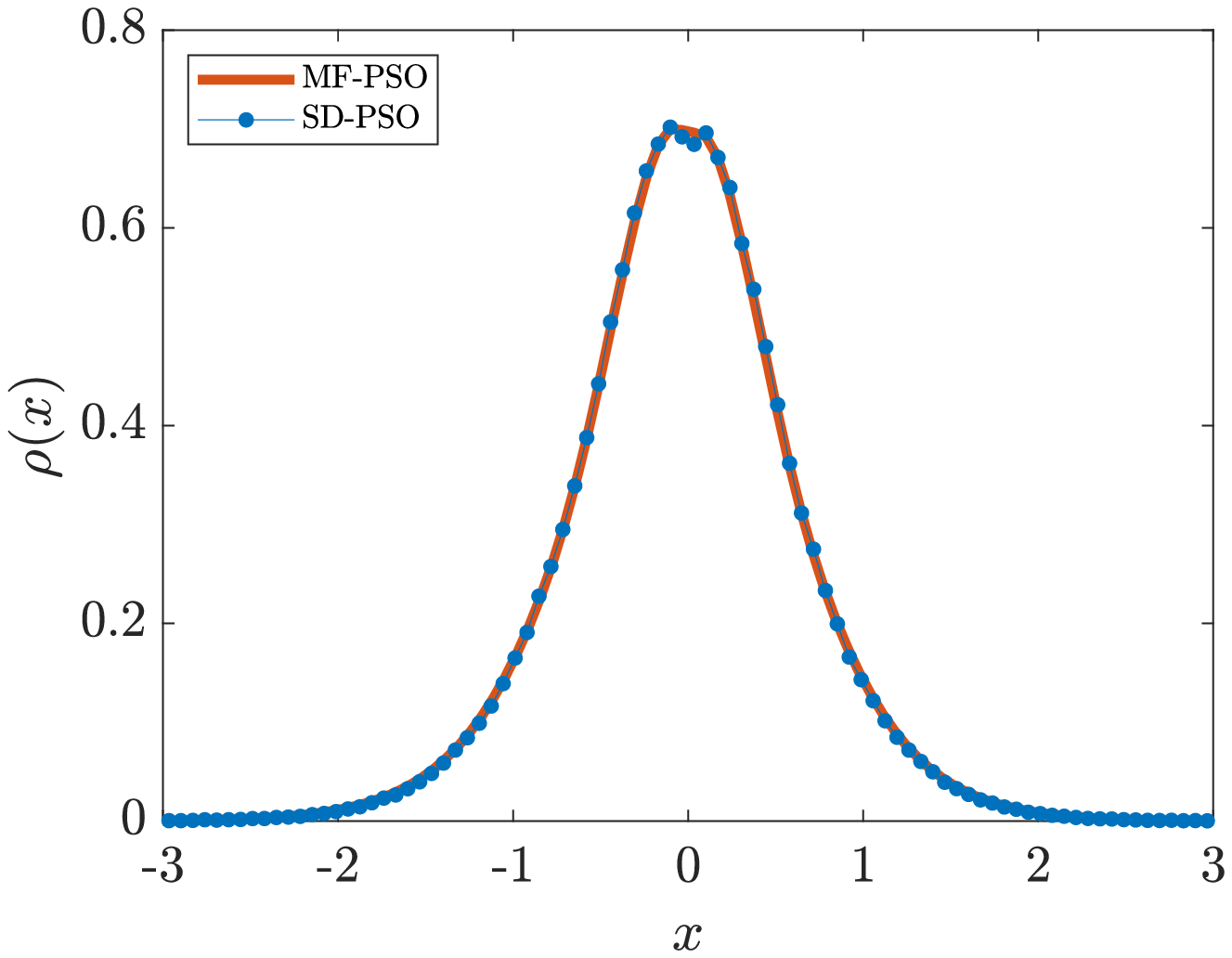}} \
\subcaptionbox{$\rho(x,t)$, $t = 3$}{\includegraphics[scale=0.25]{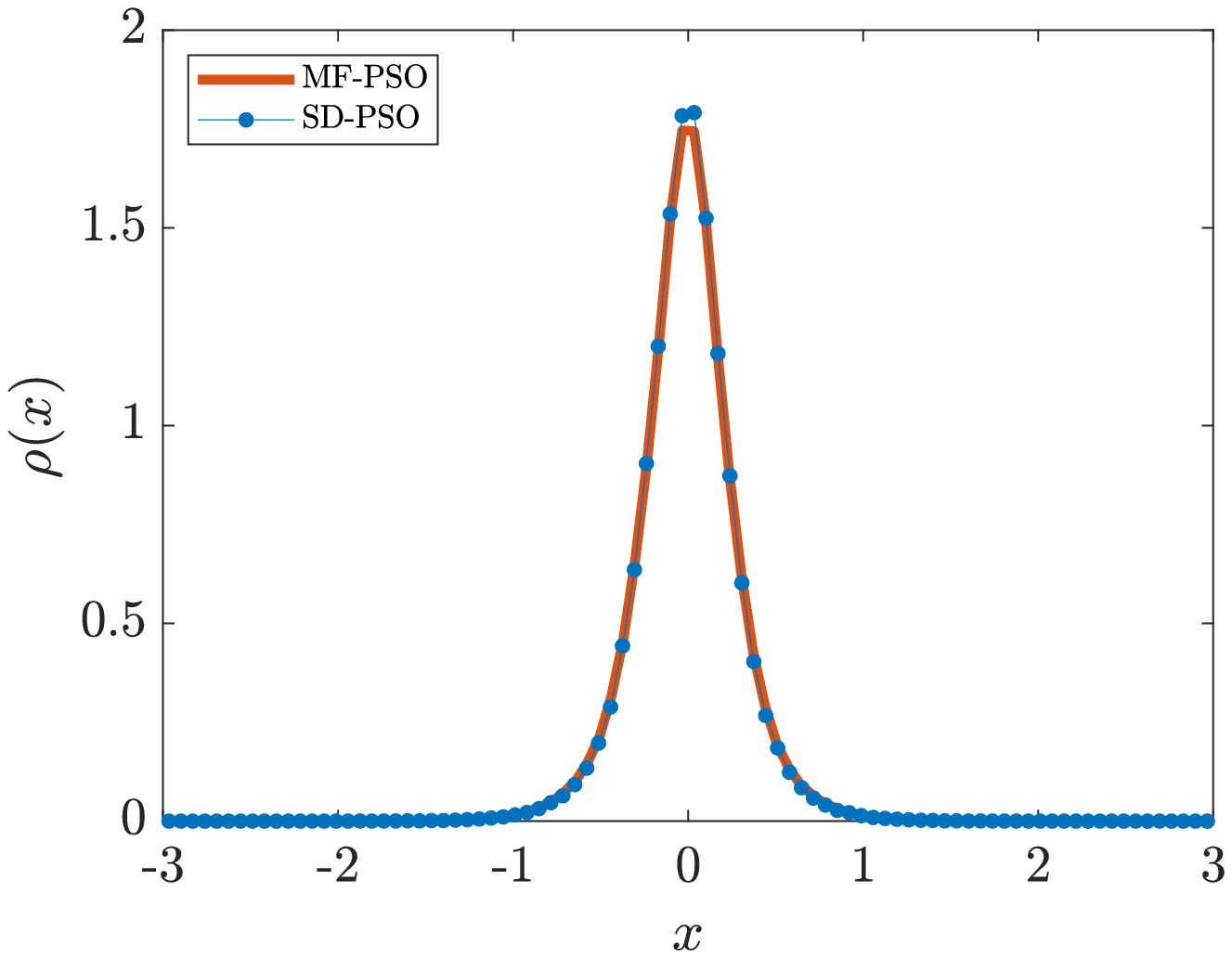}} \\
\subcaptionbox{$\rho(x,t)$, $t = 0.5$}{\includegraphics[scale=0.25]{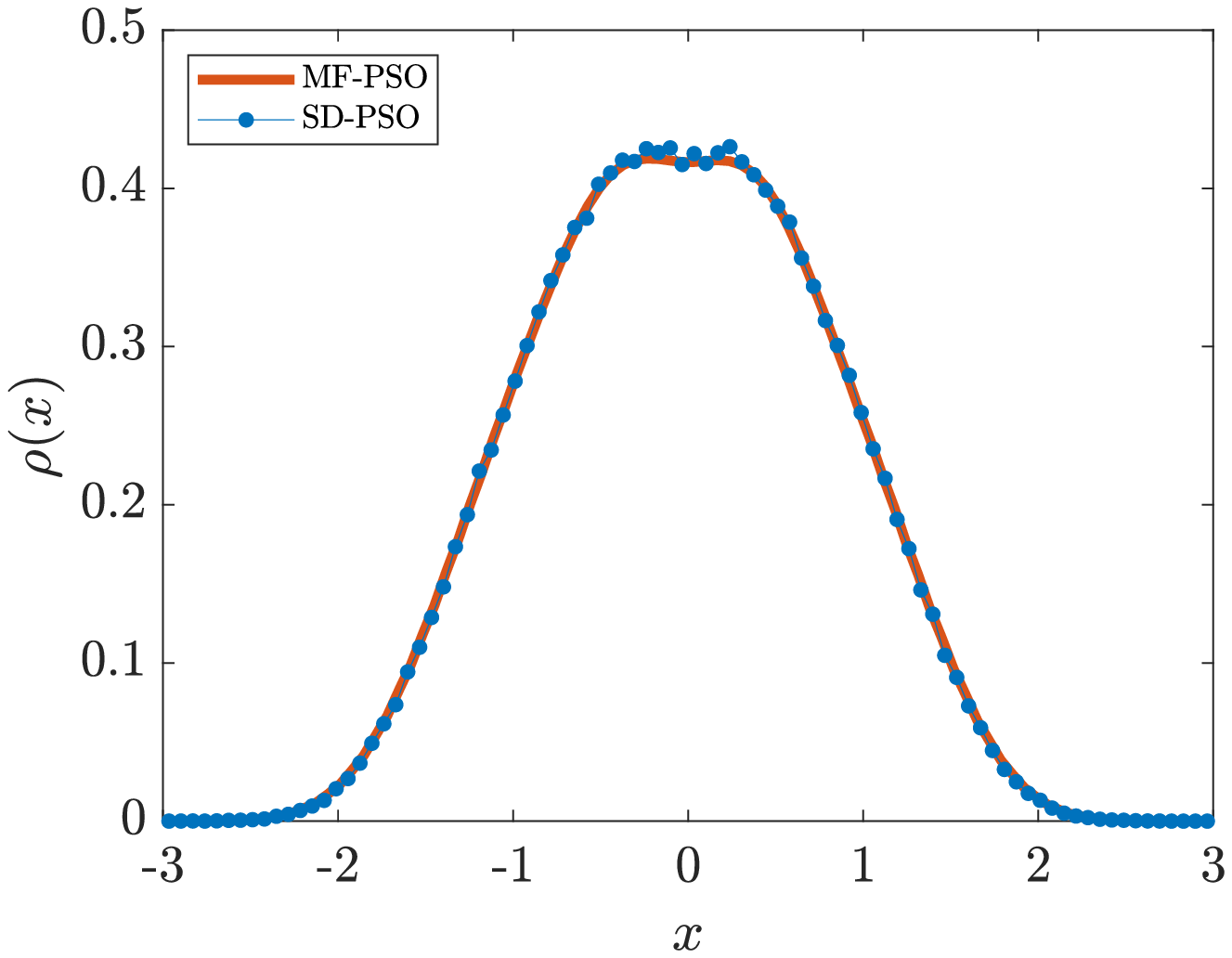}}\
\subcaptionbox{$\rho(x,t)$, $t = 1$}{\includegraphics[scale=0.25]{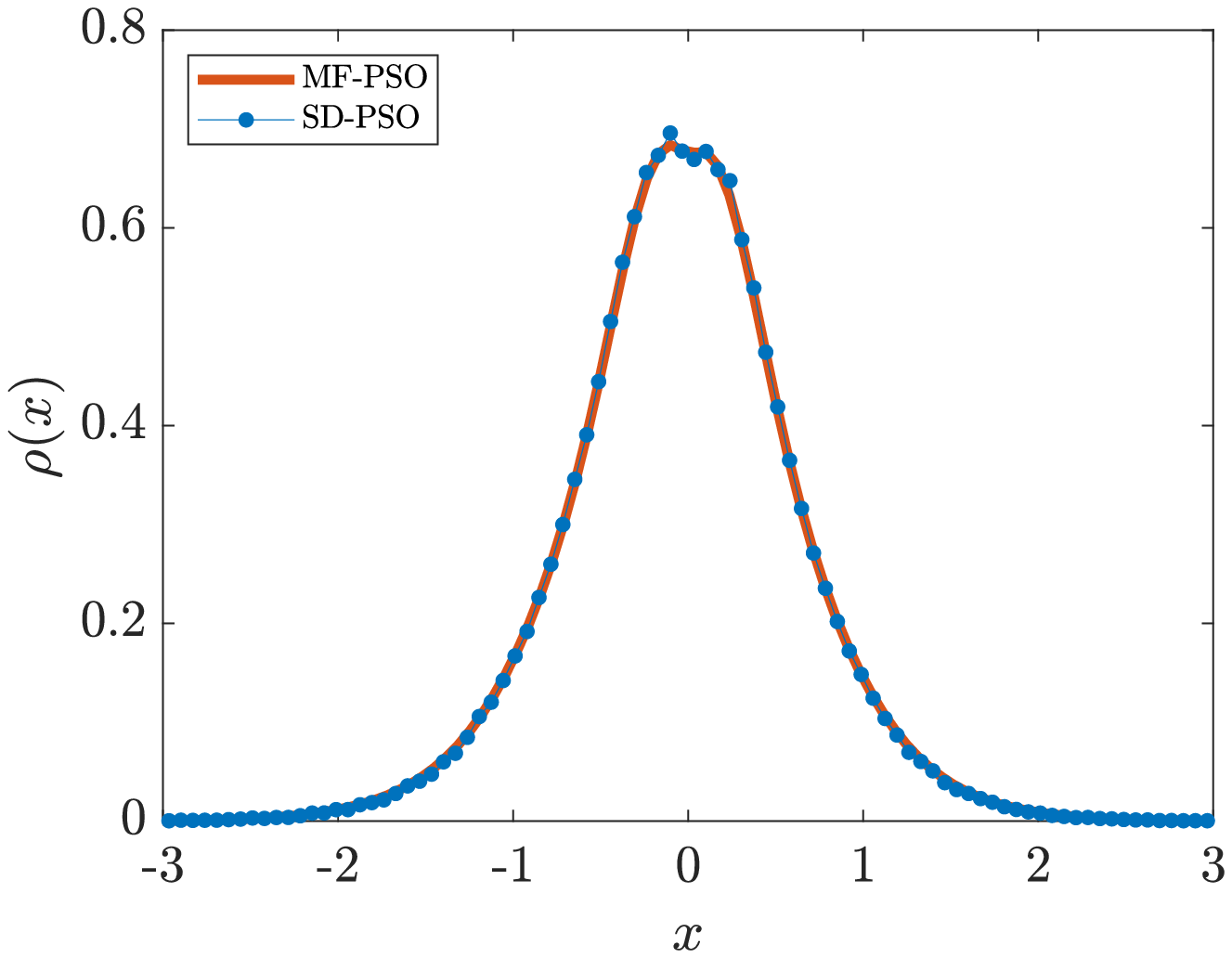}} \
\subcaptionbox{$\rho(x,t)$, $t = 3$}{\includegraphics[scale=0.25]{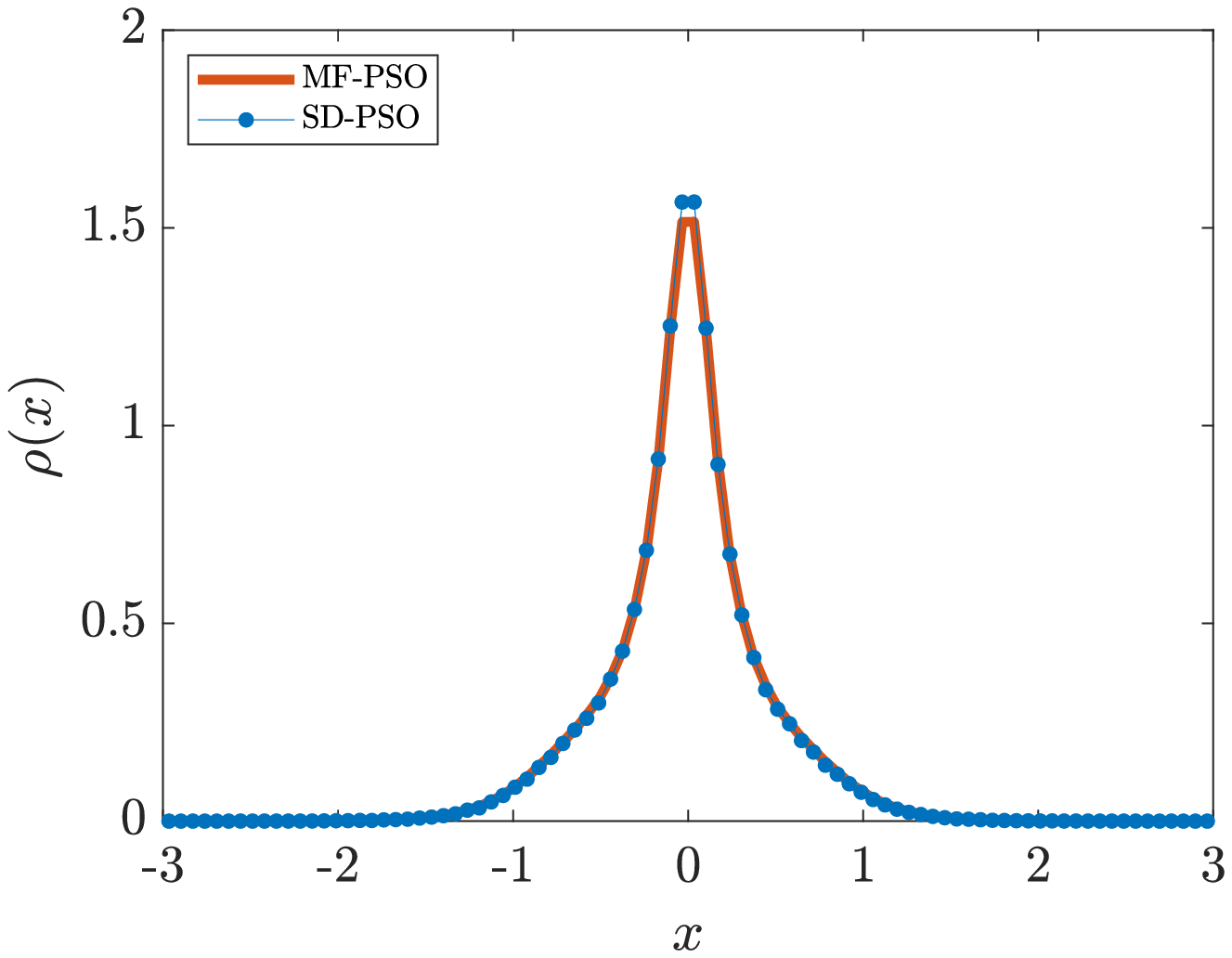}}
\caption{Mean field validation (general case). Evolution of the density $\rho(x,t)$ of the SD-PSO system \eqref{eq:psocir} and the MF-PSO limit \eqref{PDEii} for two different one-dimensional function with minimum in $x=0$. First row: optimization on the Ackley function. Second row: optimization oh the Rastrigin function.}
\label{Fig6}
\end{minipage}
\end{figure}
\subsubsection{The general case}
In the final test case, we keep the previous scenario, adding the contribution of the global best with the same weight as the local best. Therefore, we take $\lambda_1=\lambda_2=1$, $\sigma_1=\sigma_2={1}/{\sqrt{3}}$ and the same parameters \eqref{eq:paramt} in our numerical experiments. The solutions have been obtained by solving the discretized stochastic particle system \eqref{eq:psoDiscr} and the deterministic solver of the mean field equation \eqref{PDEii}. In Figure \ref{Fig6} we report the associated marginal density plots.
One can observe that the local minima effect disappears and the systems converge consistently towards the global minimum. Note that, by comparing the results for the Ackley function in Figure \ref{Fig6} and those in the last row of Figure \ref{Fig2} obtained by solving the same problem in absence of memory terms, at the same time instants, a faster convergence towards the global minimum is observed.

\subsection{Numerical small inertia limit}
From the analysis in Section 4, the classical CBO model \eqref{eq:CBOp} is produced as a hydrodynamic approximation of the mean-field PSO system \eqref{PDEi} in the limit of small inertia.  Therefore, we compare the particle solution to a discretization of the mean-field limit CBO system \eqref{eq:CBOp}, starting from the discretization of the stochastic particle model without memory effect \eqref{eq:psoiDiscr} and decreasing the inertial weight $m \to 0$ ($\gamma \to 1$).

\begin{figure}[htb]
\begin{minipage}{\linewidth}
\centering
\subcaptionbox{$\rho(x,t)$, $t = 0.2$}{\includegraphics[scale=0.25]{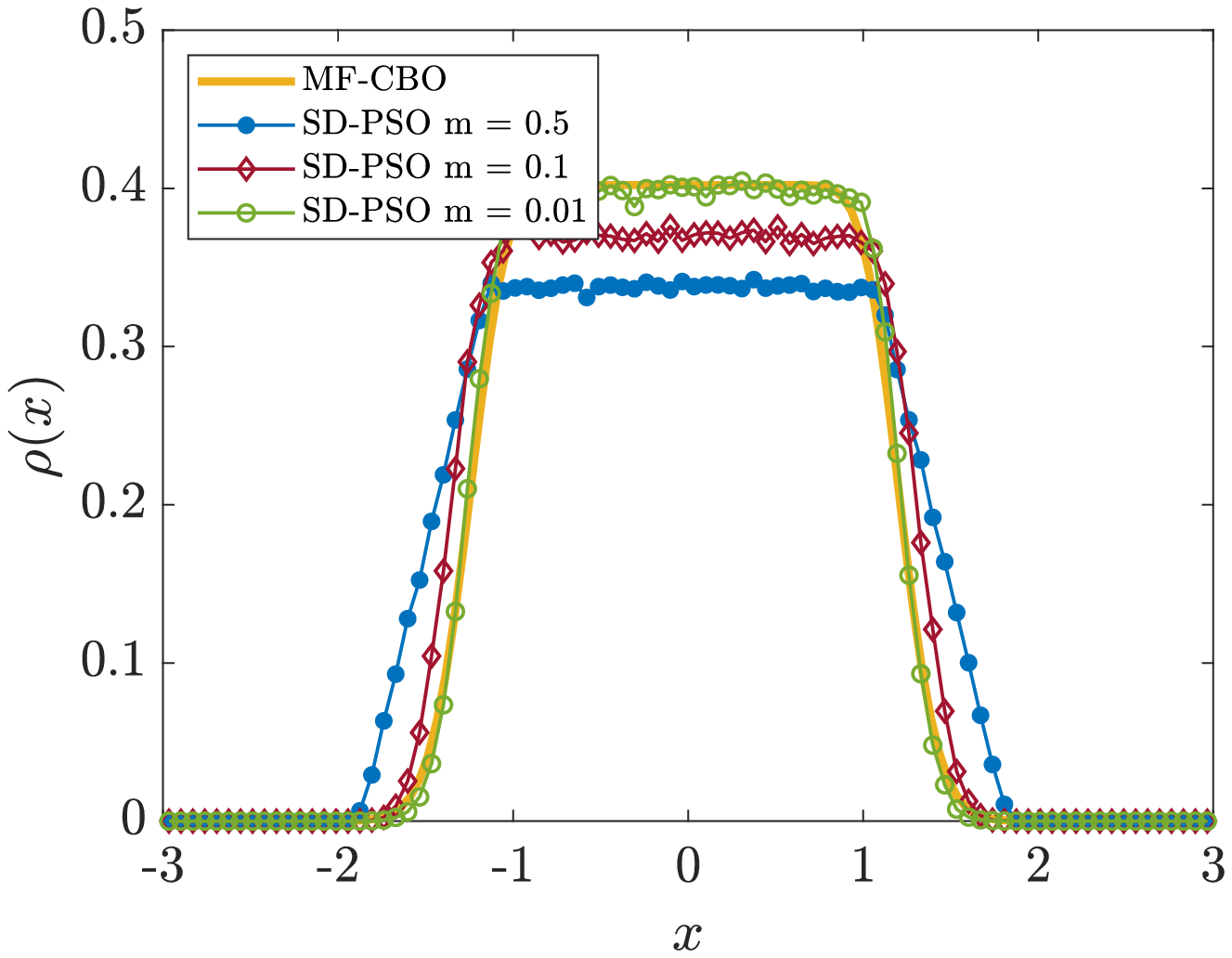}}\
\subcaptionbox{$\rho(x,t)$, $t = 1$}{\includegraphics[scale=0.25]{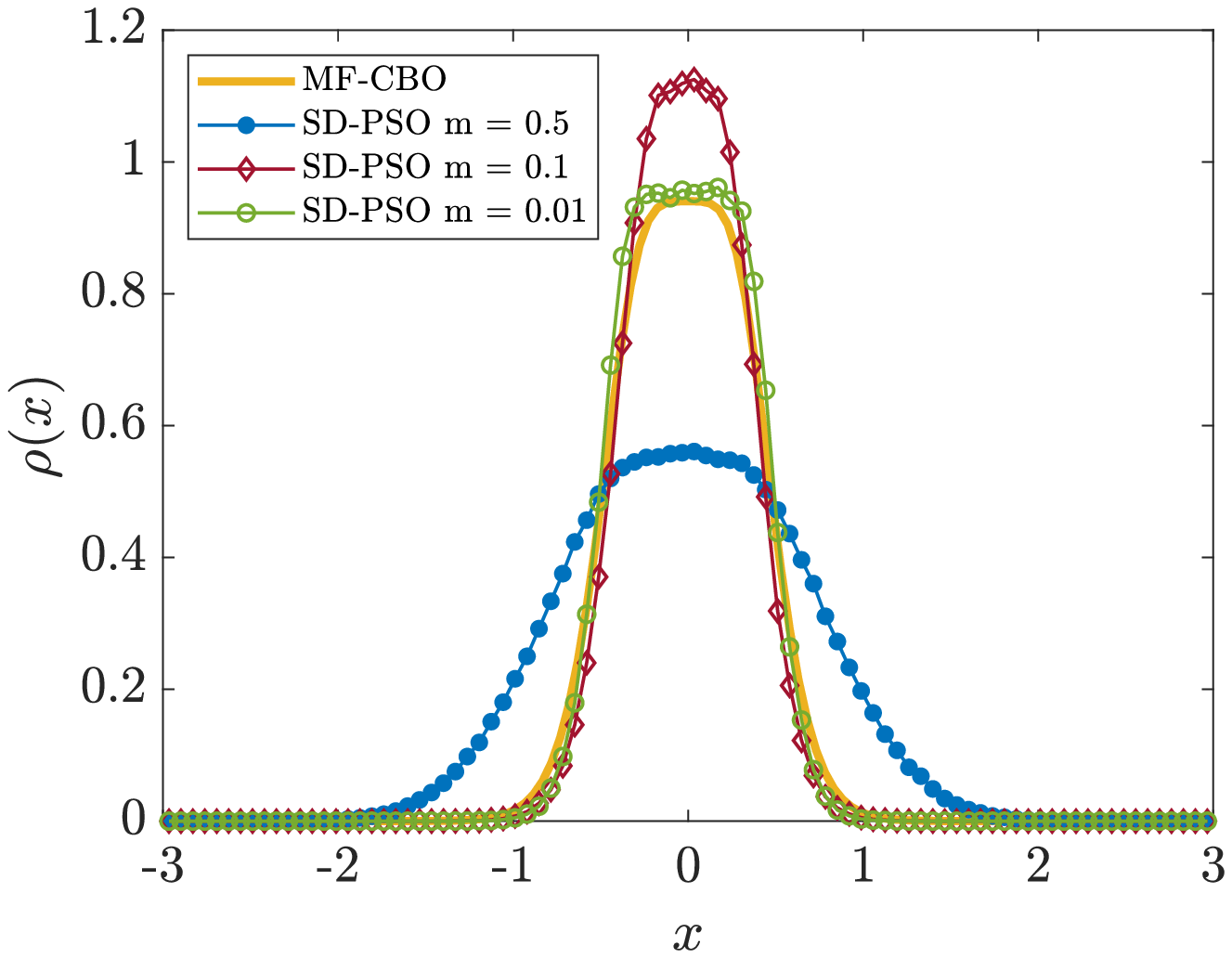}} \
\subcaptionbox{$\rho(x,t)$, $t = 2$}{\includegraphics[scale=0.25]{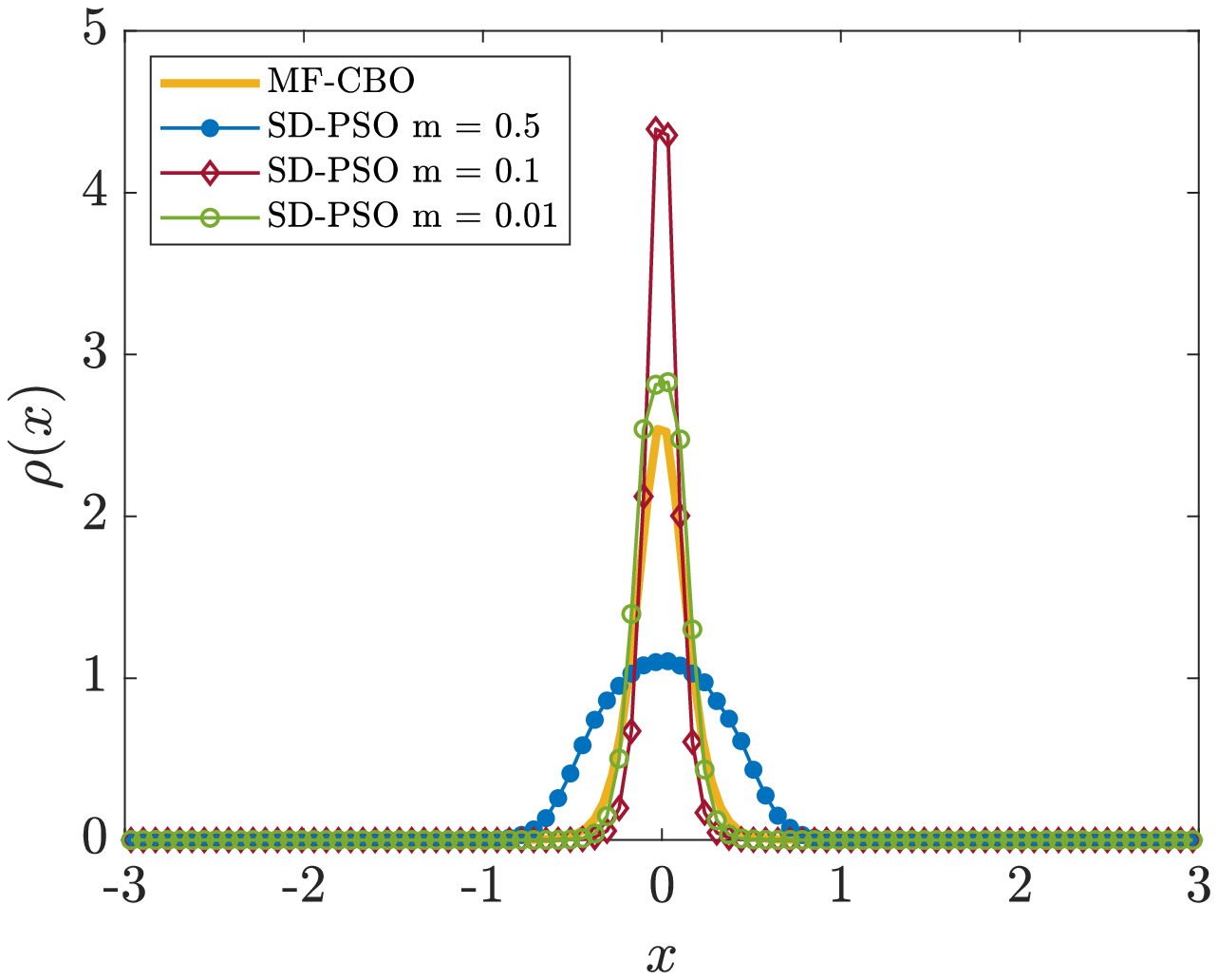}} \\
\subcaptionbox{$\rho(x,t)$, $t = 0.2$}{\includegraphics[scale=0.25]{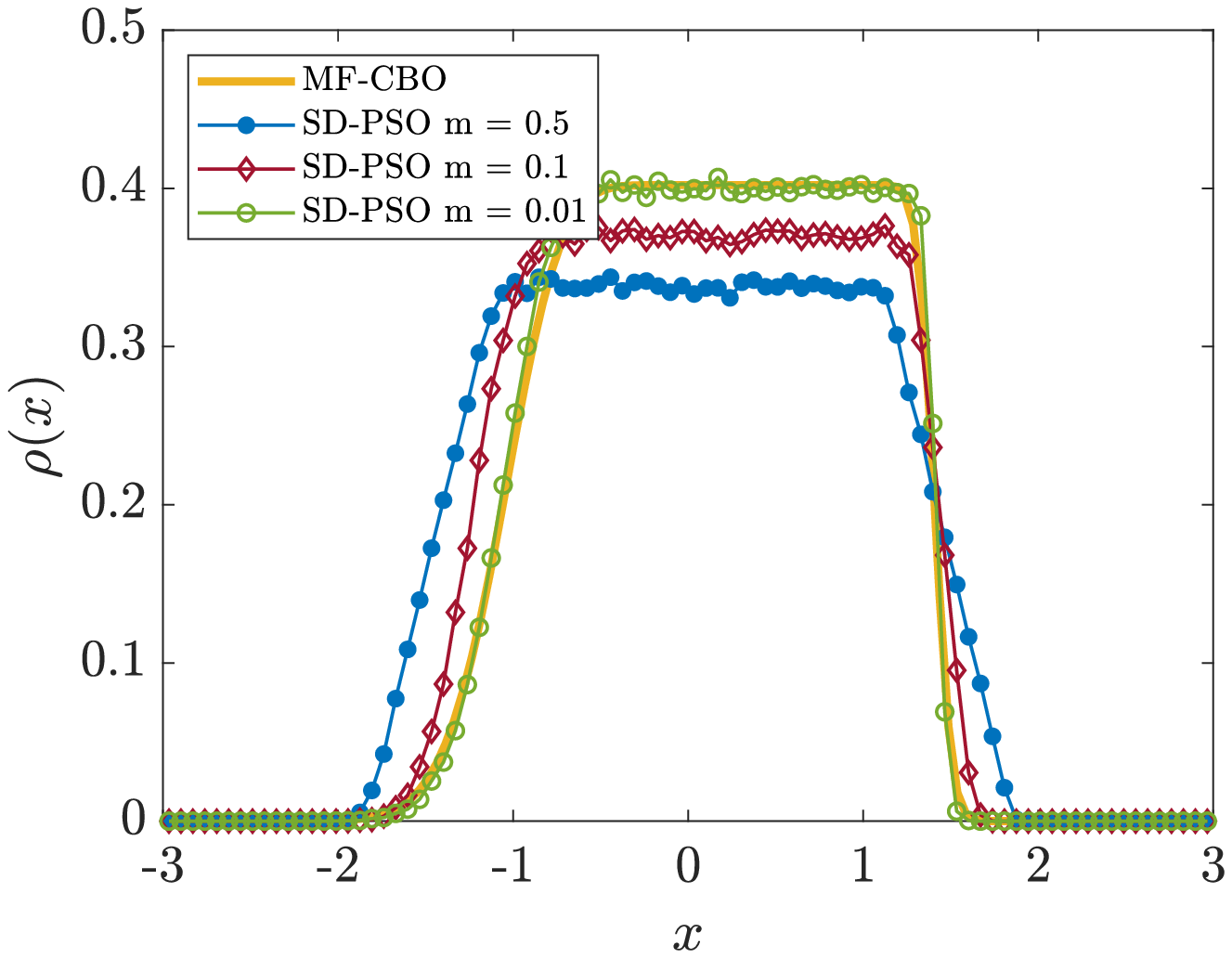}}\
\subcaptionbox{$\rho(x,t)$, $t = 1$}{\includegraphics[scale=0.25]{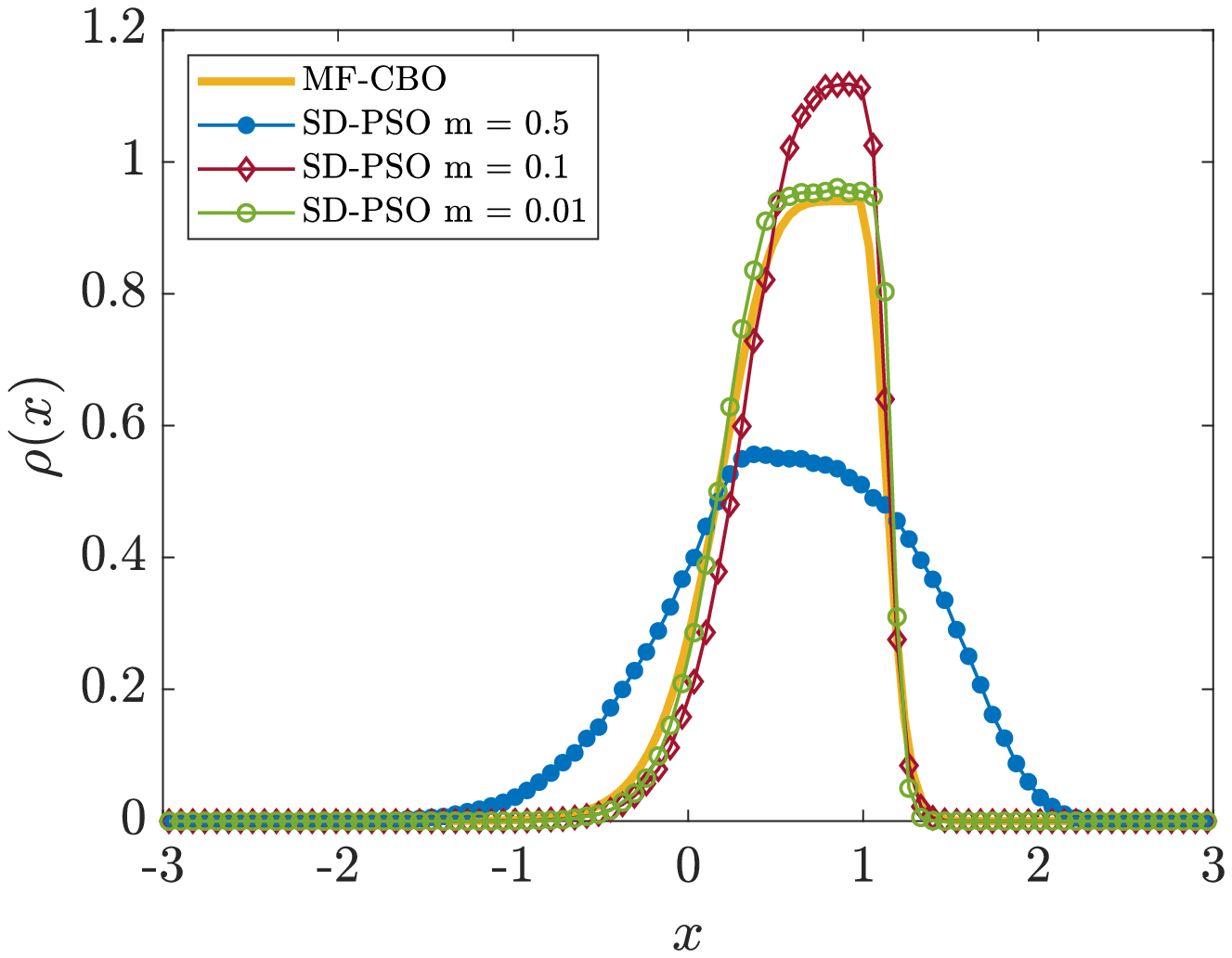}} \
\subcaptionbox{$\rho(x,t)$, $t = 2$}{\includegraphics[scale=0.25]{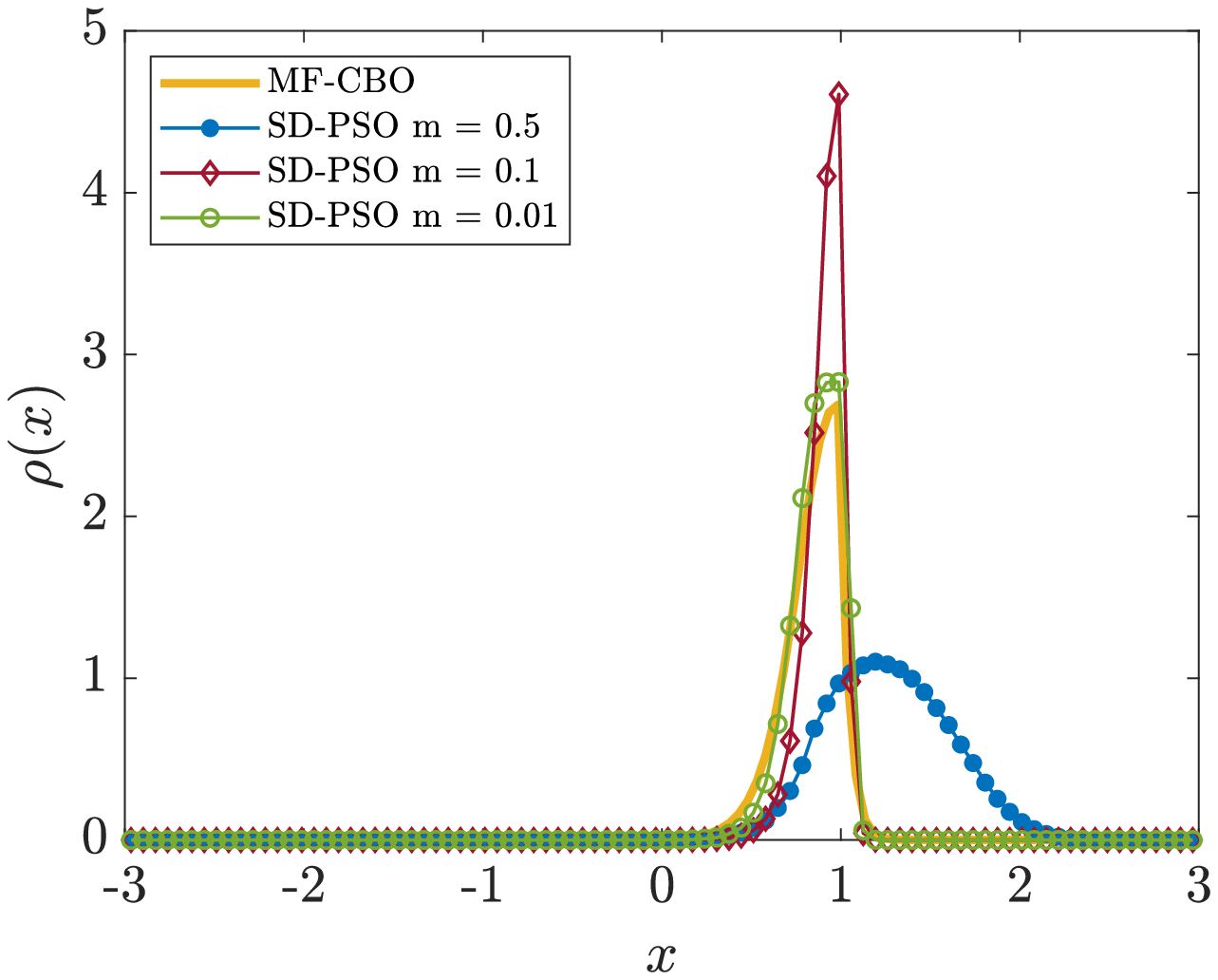}}
\caption{Low inertia limit. Evolution of the density $\rho(x,t)$ of the SD-PSO discretization \eqref{eq:psoiDiscr}, for decreasing inertial weight $m = 0.5, 0.1, 0.01$, and the mean-field CBO model \eqref{eq:CBOp} for the Ackley function with a uniform initial data. First row: minimum in $x = 0$. Second row: minimum in $x = 1$.}
\label{Fig7}
\end{minipage}
\end{figure}
In Figure \ref{Fig7} we report the plots of the density that describes the solution of the mean-field CBO model and the stochastic PSO model for different inertial weights ($m = 0.5$, $m = 0.1$ and $m = 0.01$). We considered the minimization problem for the Ackley function with minimum in $x = 0$ and in $x = 1$ with $N=5 \times 10^{5}$ particles for the SD-PSO discretization and a grid of $120$ points in space for the mean field CBO solver. It is clear that in the case of $ m = 0.5 $ the two densities at the final time $ t = 2 $ are considerably different and a slower convergence is observed in the SD-PSO system, for $ m = 0.1 $ the agreement is higher and the particle solution seems to converge faster to the minimum, finally in the case $ m = 0.01 $ both densities simultaneously grow towards a Dirac delta centered in the minimum. For smaller values of $m$ the two solutions becomes indistinguishable and we omitted the results.

\subsection{Performance on high-dimensional test cases}
In this section we report the results of several experiments concerning the behavior of the stochastic PSO models, discretized using \eqref{eq:psoiDiscr} in absence of memory or \eqref{eq:psoDiscr} in the general case, in high dimension ($d=20$) for various prototype test functions (see Table \ref{Tabfun}). 
Defining the success rate is critical as it completely alters the performance of the algorithm.  In particular, depending on the shape of the objective function, the distance between the estimated minimum and the real minimum can be used as an indicator as in \cite{pinnau2017consensus,carrillo2018analytical,carrillo2019consensus,fhps20-1,fhps20-2}. For some functions, however, this choice may be a poor indicator of the algorithm's performance, since the corresponding value function may be far from its minimum optimal value. In the first round of test cases, since we are limited to the Ackley and Rastrigin functions, for comparison purposes we rely on the choice reported below. Later, when we test the performance of the algorithm for a broader spectrum of test functions, we will generalize the definition of success rate by including the value of the function \cite{TW}.  

Thus we define:
\begin{itemize}
\item  the \textit{success rate}, computed averaging over $n_r$ runs and using as convergence criterion 
\[
\begin{split}
\Vert\Q_\alpha^{n_*}-x^{\ast}\Vert_{\infty}< \delta_{err},\qquad {\rm or}\qquad \Vert\bar\P_\alpha^{n_*}-x^{\ast}\Vert_{\infty}< \delta_{err} 
  \end{split}
  \label{ConvCrit}
\]
where $x^{\ast}$ is the minimum and $n_*$ the final time.
 \item the \textit{error}, evaluated as expected value in the $L_2$ norm over the successful runs
\[
\begin{split}
\mathbb{E}(\Vert \Q_\alpha^{n_*}-x^{\ast} \Vert_2 ),\qquad {\rm or}\qquad \mathbb{E}(\Vert \bar\P_\alpha^{n_*}-x^{\ast} \Vert_2 );
  \end{split}
\] 
\item the \textit{number of iterations}, where we stop the iteration if
\[
\Vert \Q^{n}_{\alpha}-\Q^{n-1}_{\alpha}\Vert < \delta_{stall},\qquad {\rm or}\qquad \Vert \bar\P^{n}_{\alpha}-\Q^{n-1}_{\alpha}\Vert < \delta_{stall}
\]
for $n_{stall}$ consecutive iterations or a maximum $n_{max}$ iterations has been reached. 
\end{itemize}

In the sequel, we consider $n_r=500$, $\delta_{err}=0.25$, $\delta_{stall}=10^{-4}$ and $n_{max}=10^4$. 
We remark that, increasing the problem dimension, a larger value of $ \alpha \gg 1 $ provides better performance \cite{pinnau2017consensus,fhps20-2}. On the other hand, a large value of $ \alpha $ may generate numerical instabilities given by the definition of the regularized global best. To avoid this, we used the algorithm presented in \cite{fhps20-2} which allow the use of arbitrary large values of $\alpha$.

In the following test cases, we address the role of the various parameters, of the presence of memory and of the local best when solving high dimensional global optimization problems. We refer also to \cite{Totzeck2018ANC} for additional comparisons. In our experiments, the PSO constraints \eqref{eq:param} have shown strong limitations in terms of success rates and have not been considered. We refer to \cite{Grassi2021PSO} for further details and comparisons.


\begin{table}[tb]
\begin{center}
\resizebox{4.5in}{1in}{
\begin{tabular}{l||l|l|ccc|l|ccc}
\multicolumn{3}{l}{\hspace{-5pt}\textbf{Rastrigin}} &  \multicolumn{4}{c}{\hspace{-45pt}Case without memory} &\multicolumn{3}{c}{Case with memory} \\
\hline
\hline
\hspace{5pt}$m$& & \hspace{5pt}$\sigma$ & $N=50$ & $N=100$ & $N=200$ & \hspace{5pt}$\sigma_2$ & $N=50$ & $N=100$ & $N=200$\\                  
\hline
{\rm $0.00$}& Rate& {\rm $9.0$}& 100.0\% & 100.0\% & 100.0\% &{\rm $11.0$}& 100.0\% & 100.0\% & 100.0\%\\
& Error &&   1.19e-04 & 1.11e-04 & 9.68e-05& & 6.83e-04 & 4.70e-04 & 4.69e-04\\
& $n_{iter}$& &10000.0 & 10000.0 & 9912.4 && 10000.0 & 9878.2 & 3290.2\\
\hline
{\rm $0.01$}& Rate &{\rm $7.0$} &  100.0\% & 100.0\% & 100.0\% &{\rm \ $9.0$}& 100.0\% & 100.0\% & 100.0\%\\
& Error &&  9.74e-05 & 2.01e-05 & 1.62e-05& & 8.60e-04 & 8.56e-04 & 8.81e-04\\
& $n_{iter}$& &10000.0 & 6899.2 & 2060.1 &&9939.5 & 7012.2 & 5422.1\\
\hline
{\rm $0.05$}& Rate &{\rm $3.5$}& 37.0\% & 74.0\% & 94.0\% &{\rm \ $4.5$}& 100.0\% & 100.0\% & 100.0\%\\
& Error &&4.27e-04 & 1.26e-04 & 1.14e-04& &  1.15e-03 & 6.67e-04 & 6.54e-04\\
& $n_{iter}$& &8233.2 & 7814.0 & 7326.6 &&9978.0 & 7657.6 & 5639.7\\
\hline
{\rm $0.10$}& Rate &{\rm $ 2.0$}& 1.0\% & 5.5\% & 29.5\% &{\rm \ $ 3.0$}& 80.8\% & 96.8\% & 100.0\%\\
& Error & & 2.00e-04 & 1.28e-04 & 1.11e-04& & 2.94e-03 & 8.96e-04 & 8.24e-04\\
& $n_{iter}$& &6155.4 & 6221.9 & 6214.3 &&9661.5 & 8676.5 & 7331.8\\
\hline
\hline
\end{tabular}}
\end{center}
\caption{SD-PSO with and without memory for $\lambda_1=\sigma_1=0$, $\lambda_2=1$, $\Delta t = 0.01$, $\nu=50$, $\beta = 3 \times 10^3$ and $\alpha = 5 \times 10^4$.}
\label{tab:1R}
\end{table}

\subsubsection{Effect of the inertial parameter $m$ }
First we test the algorithm performance for the Ackley and the Rastrigin functions in $[-3, 3]^d$, $d=20$.  In the left column of Table \ref{tab:1R} and Table \ref{tab:1A} we report the results obtained without memory effects \eqref{eq:psoiDiscr} and in the right column the results with memory effects \eqref{eq:psoDiscr}. Since, typically, optimizing the Rastrigin function is far more difficult than the Ackley function, we explore the space of parameters searching for optimal values of $\sigma$ and $\Delta t$ for the Rastrigin function,  then we used the same values for the Ackley function. This optimization was done empirically through several simulations with simple variations of a given step size for the parameters.

\begin{table}[htb]
\begin{center}
\resizebox{4.5in}{1in}{
\begin{tabular}{l||l|l|ccc|l|ccc}
\multicolumn{3}{l}{\hspace{-5pt}\textbf{Ackley}} &  \multicolumn{4}{c}{\hspace{-45pt}Case without memory} &\multicolumn{3}{c}{Case with memory} \\
\hline
\hline
\hspace{5pt}$m$& & \hspace{5pt}$\sigma$ & $N=50$ & $N=100$ & $N=200$ & \hspace{5pt}$\sigma_2$ & $N=50$ & $N=100$ & $N=200$\\                  
\hline
{\rm $0.00$}& Rate& {\rm $9.0$}& 100.0\% & 100.0\% & 100.0\% &{\rm $11.0$}& 100.0\% & 100.0\% & 100.0\%\\
& Error &&  8.46e-05 & 4.20e-05 & 1.27e-05& & 1.02e-04 & 7.66e-05 & 5.44e-05\\
& $n_{iter}$& &1364.9 & 1032.4 & 869.2 &&  2457.0 & 1778.0 & 1513.1\\
\hline
{\rm $0.01$}& Rate &{\rm $7.0$} &  100.0\% & 100.0\% & 100.0\% &{\rm \ $9.0$}& 100.0\% & 100.0\% & 100.0\%\\
& Error && 9.49e-05 & 5.89e-05 & 2.81e-05& & 2.34e-03 & 1.91e-04 & 1.61e-04\\
& $n_{iter}$& &2192.9 & 1886.7 & 1723.6 && 6430.4 & 5447.8 & 4598.3\\
\hline
{\rm $0.05$}& Rate &{\rm $3.5$}&100.0\% & 100.0\% & 100.0\% &{\rm \ $4.5$}& 100.0\% & 100.0\% & 100.0\%\\
& Error &&2.27e-04 & 1.48e-04 & 1.03e-04& &  2.41e-04 & 1.84e-04 & 1.48e-04\\
& $n_{iter}$&&5367.3 & 4459.4 & 3928.4 &&7186.1 & 5996.0 & 5074.6\\
\hline
{\rm $0.10$}& Rate &{\rm $ 2.0$}& 99.5\% & 100.0\% & 100.0\% &{\rm \ $ 3.0$}& 100.0\% & 100.0\% & 100.0\%\\
& Error & & 8.31e-04 & 2.76e-04 & 1.91e-04& & 3.90e-03 & 2.64e-03 & 2.06e-03\\
& $n_{iter}$& &5480.8 & 4514.1 & 3909.4&&8590.6 & 7326.4 & 6350.2\\
\hline
\hline
\end{tabular}}
\end{center}
\caption{SD-PSO with and without memory for $\lambda_1=\sigma_1=0$, $\lambda = \lambda_2=1$, $\Delta t = 0.01$, $\nu=50$, $\beta = 3 \times 10^3$ and $\alpha = 5 \times 10^4$.}
\label{tab:1A}
\end{table}

The results are given for different numbers of particles $N$. We consider $\alpha = 5 \times 10^4$, whereas
the memory parameters $\beta$ and $\nu$ were chosen respectively $\beta=3\times 10^3$ and $\nu=50$.
Note that, even if we rely only on the global best since we fix $\lambda_1=\sigma_1=0$, due to the regularization of the memory process the two approaches, with and without memory, differs and a higher noise in required in presence of memory. Low inertia values yields better performances overall, however, it should be noticed that for the Rastrigin function the best results in term of convergence speed are obtained with a small but non zero inertia value of $m=0.01$.  


\subsubsection{Effect of the local best dynamics}
Subsequently, we have introduced the local best dynamics in the same optimization process. To reduce the number of free parameters we assume $\lambda_1 = \xi \cdot \lambda_2$, $\sigma = \xi \cdot \sigma_2$
with $\xi \in [0,1]$ so that the local best is always weighted less than the global best. In this test we keep the inertial value $m = 0$ and $\lambda_1=1$, so that we are solving the generalized stochastic differential CBO model with memory using algorithm \eqref{eq:cboDiscr}. For each value of $\xi$ reported, we have computed an optimal $\sigma_2$ achieving the maximum rate of success. We chose $\beta = 3\times 10^3$, $\Delta t=0.01$, $\nu=50$  and  $\alpha = 5 \times 10^4$ as in the previous case. 

In Tables \ref{tab:3R} and \ref{tab:3A} we report the behavior of the particle optimizer on the Ackley and Rastrigin functions for different positions of the minimum $x^{\ast} = 0$, $x^{\ast} = 1$ and $x^{\ast} = 2$. Since for large values of $\xi$ we must decrease $\sigma_2$ to achieve maximum convergence rate we observe that the total number of iterations may decrease and that a speed-up is obtained thanks to the local best.

\begin{table}[htb]
\begin{center}
\resizebox{4.5in}{0.9in}{
{\renewcommand{\arraystretch}{1}
\begin{tabular}{l|lccc|ccc}
{\vspace{-2pt}\textbf{Rastrigin}} &  \multicolumn{4}{c}{Case $\xi = 0$, $\sigma_2 = 11.0$} &\multicolumn{3}{c}{Case $\xi = 0.25$, $\sigma_2 = 8.5$} \\
\hline
\hline
& & $N=50$ & $N=100$ & $N=200$ & $N=50$ & $N=100$ & $N=200$\\  
\hline
{\rm $x^{\ast} = 0$} & Rate  &  100.0\% & 100.0\% & 100.0\% & 100.0\% & 100.0\% & 100.0\%\\
 & Error & 7.04e-04 & 4.58e-04 & 3.29e-04&9.28e-04 & 6.11e-04 & 4.31e-04\\
& $n_{iter}$ & 10000.0 & 9963.9 & 4635.1&9978.0 & 8311.5 & 5754.1\\
\hline
{\rm $x^{\ast} = 1$}& Rate  &  98.8\% & 100.0\% & 100.0\% & 99.2\% & 100.0\% & 100.0\%\\
 & Error & 7.08e-04 & 4.60e-04 & 3.27e-04&9.31e-04 & 6.74e-04 & 4.59e-04\\
& $n_{iter}$ &10000.0 & 10000.0 & 4670.0 & 9987.0 & 9746.7 & 7460.1\\
\hline
{\rm $x^{\ast} = 2$}& Rate  &96.0\% & 99.1\% & 100.0\% & 93.5\% & 100.0\% & 100.0\%\\
 & Error &  6.91e-04 & 4.52e-04 & 3.28e-04&8.78e-04 & 6.74e-04 & 5.66e-04\\
& $n_{iter}$ &10000.0 & 10000.0 & 5035.5 & 9980.3 & 9854.1 & 8971.9\\
\hline  
\hline       
\end{tabular}}}
\end{center}
\caption{SD-PSO with memory ($m=0$) for $\lambda_1 = \xi \cdot \lambda_2$, $\sigma = \xi \cdot \sigma_2$, $\lambda_1$, $\lambda_2=1$, $\Delta t = 0.01$, $\nu=50$, $\beta = 3 \times 10^3$, $\alpha = 5 \times 10^4$.}
\label{tab:3R}
\end{table}

\begin{table}[htb]
\begin{center}
\resizebox{4.5in}{0.9in}{
{\renewcommand{\arraystretch}{1}
\begin{tabular}{l|lccc|ccc}
{\textbf{Ackley}} &  \multicolumn{4}{c}{Case $\xi = 0$, $\sigma_2 = 11.0$} &\multicolumn{3}{c}{Case $\xi = 0.25$, $\sigma_2 = 8.5$} \\
\hline
\hline
& & $N=50$ & $N=100$ & $N=200$ & $N=50$ & $N=100$ & $N=200$\\  
\hline
{\rm $x^{\ast} = 0$ \ } & Rate  &  100.0\% & 100.0\% & 100.0\% & 100.0\% & 100.0\% & 100.0\%\\
 & Error & 7.36e-05 & 5.13e-05 & 3.26e-05&2.54e-05 & 1.13e-05 & 1.07e-05\\
& $n_{iter}$ &2778.6 & 2030.0 & 1623.0 & 1942.9 & 1663.8 & 1442.5\\
\hline
{\rm $x^{\ast} = 1$ \ }& Rate  &  100.0\% & 100.0\% & 100.0\% & 100.0\% & 100.0\% & 100.0\%\\
 & Error & 7.31e-05 & 5.14e-05 & 3.26e-05&2.58e-05 & 1.12e-05 & 1.02e-05\\
& $n_{iter}$ & 5298.5 & 3640.6 & 2575.9 &  2465.3 & 1948.5 & 1632.5\\
\hline
{\rm $x^{\ast} = 2$ \ }& Rate  & 100.0\% & 100.0\% & 100.0\% &  100.0\% & 100.0\% & 100.0\%\\
 & Error & 7.30e-05 & 5.07e-05 & 3.22e-05&2.64e-05 & 1.09e-05 & 1.01e-05\\
& $n_{iter}$ &7819.8 & 5771.3 & 4235.9 &3126.8 & 2286.0 & 1803.8\\
\hline  
\hline       
\end{tabular}}}
\end{center}
\caption{SD-PSO with memory ($m=0$) for $\lambda_1 = \xi \cdot \lambda_2$, $\sigma = \xi \cdot \sigma_2$, $\lambda_2=1$, $\Delta t = 0.01$, $\nu=50$, $\beta = 3 \times 10^3$, $\alpha = 5 \times 10^4$.}
\label{tab:3A}
\end{table}

\begin{table}[htb]
{\large
\begin{center}\resizebox{4.5in}{2.2in}{
  \centering
  \begin{tabular}{  c | c | c | c | c | c  }
    \textbf{Name} & \textbf{Function $\TE(x)$}& \textbf{Range} & \textbf{$x^\ast$} & \textbf{$\TE(x^\ast)$} & \textbf{Sketch in 2D} \\
    \hline \hline
    Ackley & $-20\ \mbox{exp}\left( -0.2\sqrt{\frac{1}{d}\sum_{i=1}^{d}{(x_i)^2}}\right)-\mbox{exp}\left(\frac{1}{d}\sum_{i=1}^d{ \cos \left( 2\pi (x_i)\right) }\right) +20 +e$  & $[-32,32]^d$ & $ (0,\dots,0)$ & 0 &
     \begin{minipage}{.25\textwidth}  \centering \vspace{5pt}
      \includegraphics[scale=0.20]{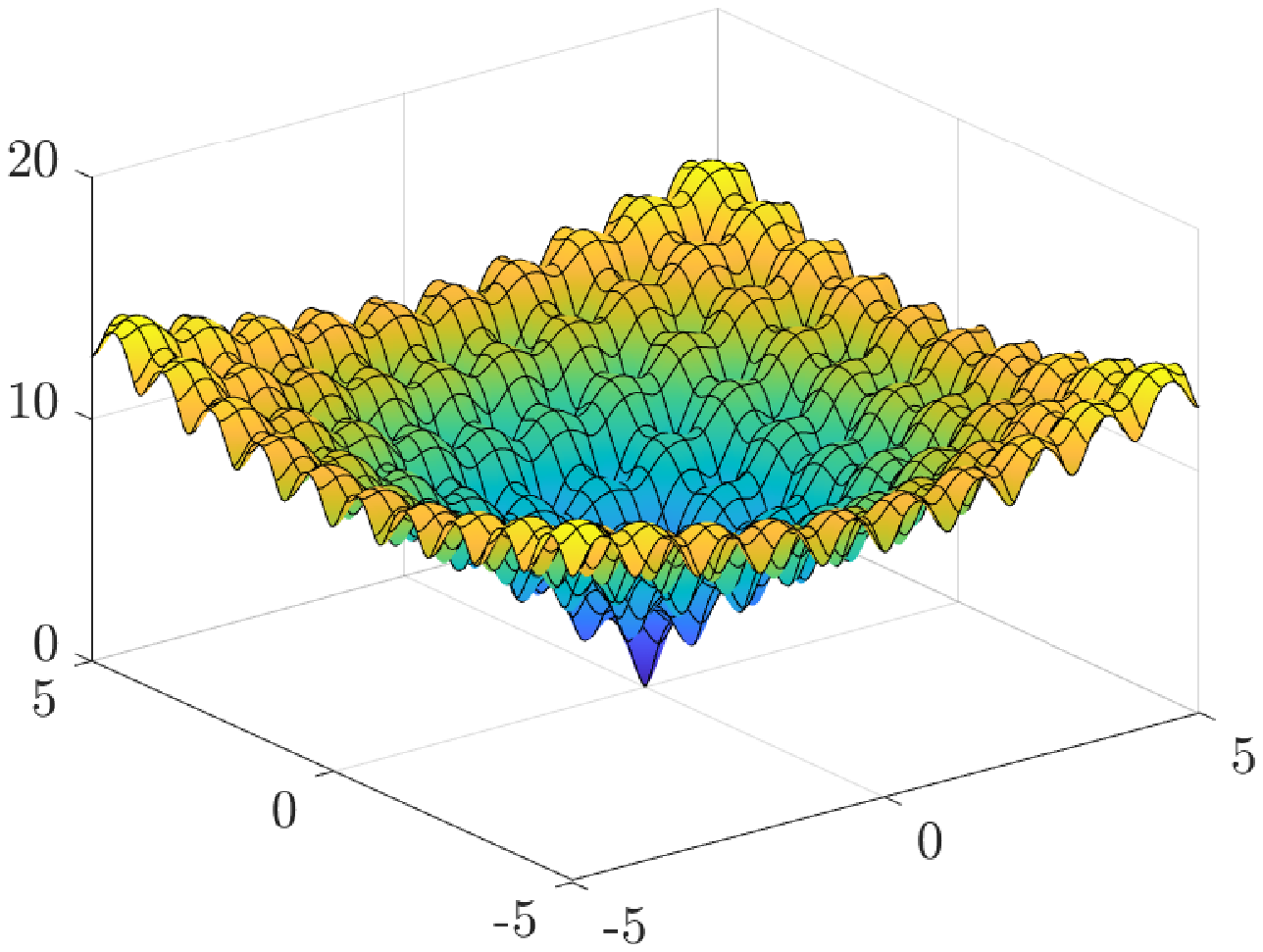} 
    \end{minipage}
    \\ \hline
     Griewank & $1+\sum_{i=1}^{d} \frac{(x_i)^2}{4000}-\prod_{i=1}^{d} \mbox{cos}\left(\frac{x_i}{i}\right)$   & $[-600,600]^d$ & $ (0,\dots,0)$ & 0 &
     \begin{minipage}{.25\textwidth}  \centering \vspace{5pt}
      \includegraphics[scale=0.20]{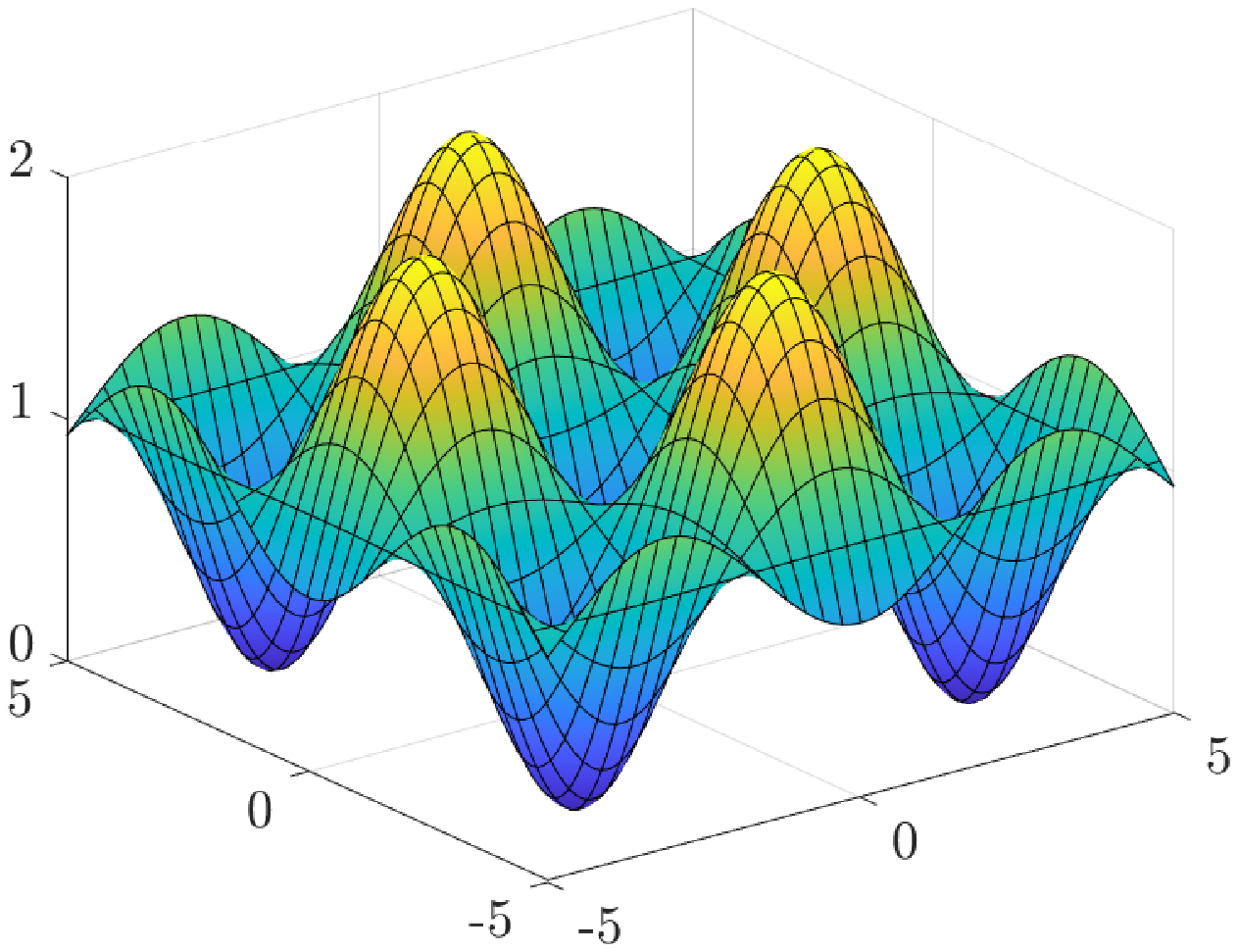} 
    \end{minipage}
    \\ \hline
        Rastrigin & $10d + \sum_{i=1}^d \left[(x_i)^2 - 10 \cos\left(2\pi (x_i) \right)\right]$  & $[-5.12,5.12]^d$ & $ (0,\dots,0)$ & 0 &   \begin{minipage}{.25\textwidth}  \centering \vspace{5pt}
      \includegraphics[scale=0.20]{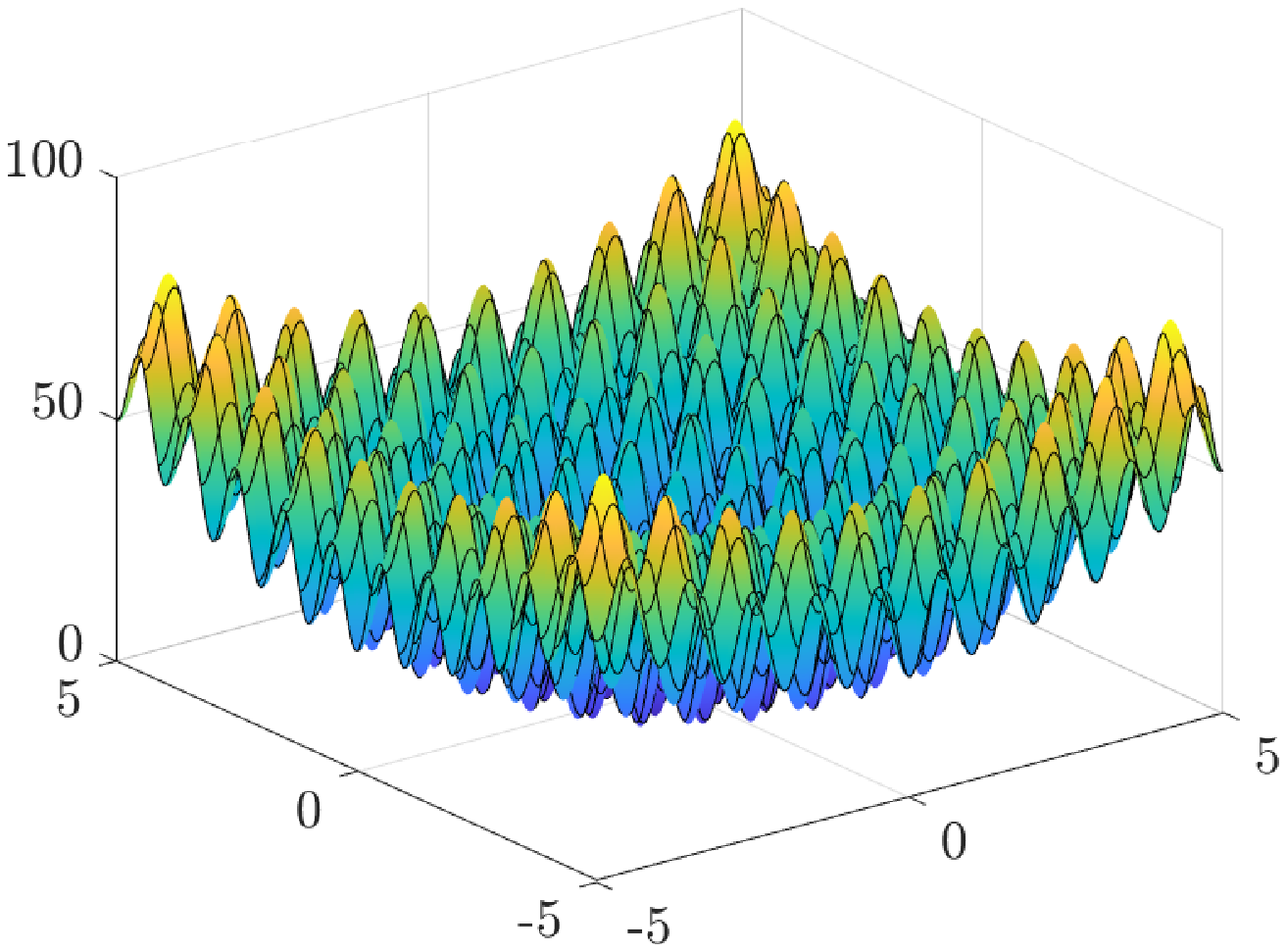} 
    \end{minipage}
    \\ \hline
       Rosenbrock & $1-\cos{\left( 2\pi \sqrt{\sum_{i=1}^{d}{(x_i)^2}}\right)} + 0.1 \sqrt{\sum_{i=1}^{d}{(x_i)^2}} $  & $[-5,10]^d$ & $ (1,\dots,1)$ & 0 &  \begin{minipage}{.25\textwidth}  \centering \vspace{5pt}
      \includegraphics[scale=0.20]{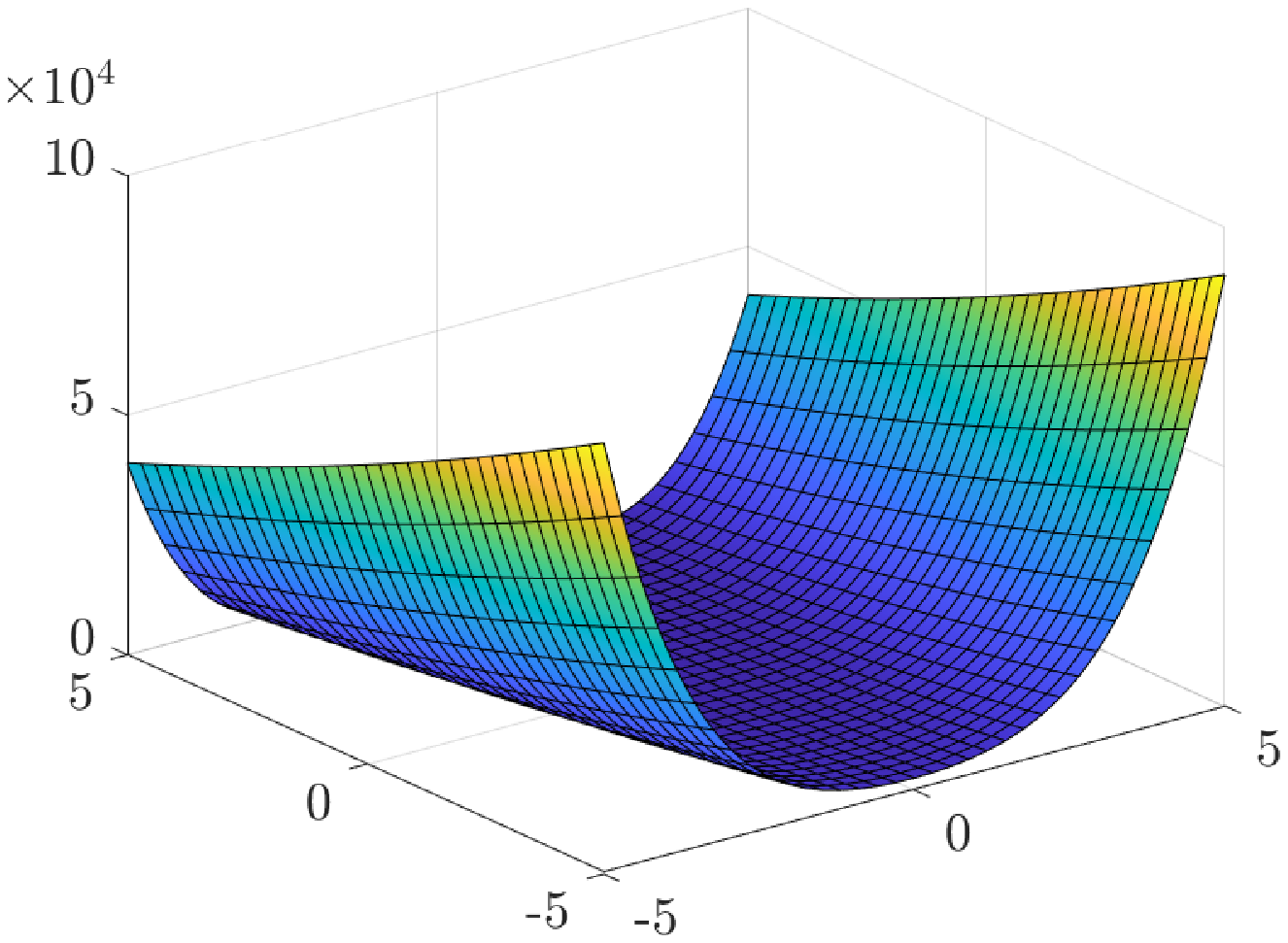} 
    \end{minipage}
    \\ \hline
       Salomon & $1-\cos{\left( 2\pi \sqrt{\sum_{i=1}^{d}{(x_i)^2}}\right)} + 0.1 \sqrt{\sum_{i=1}^{d}{(x_i)^2}}  $  & $[-100,100]^d$ & $ (0,\dots,0)$ & 0 &  \begin{minipage}{.25\textwidth}  \centering \vspace{5pt}
      \includegraphics[scale=0.20]{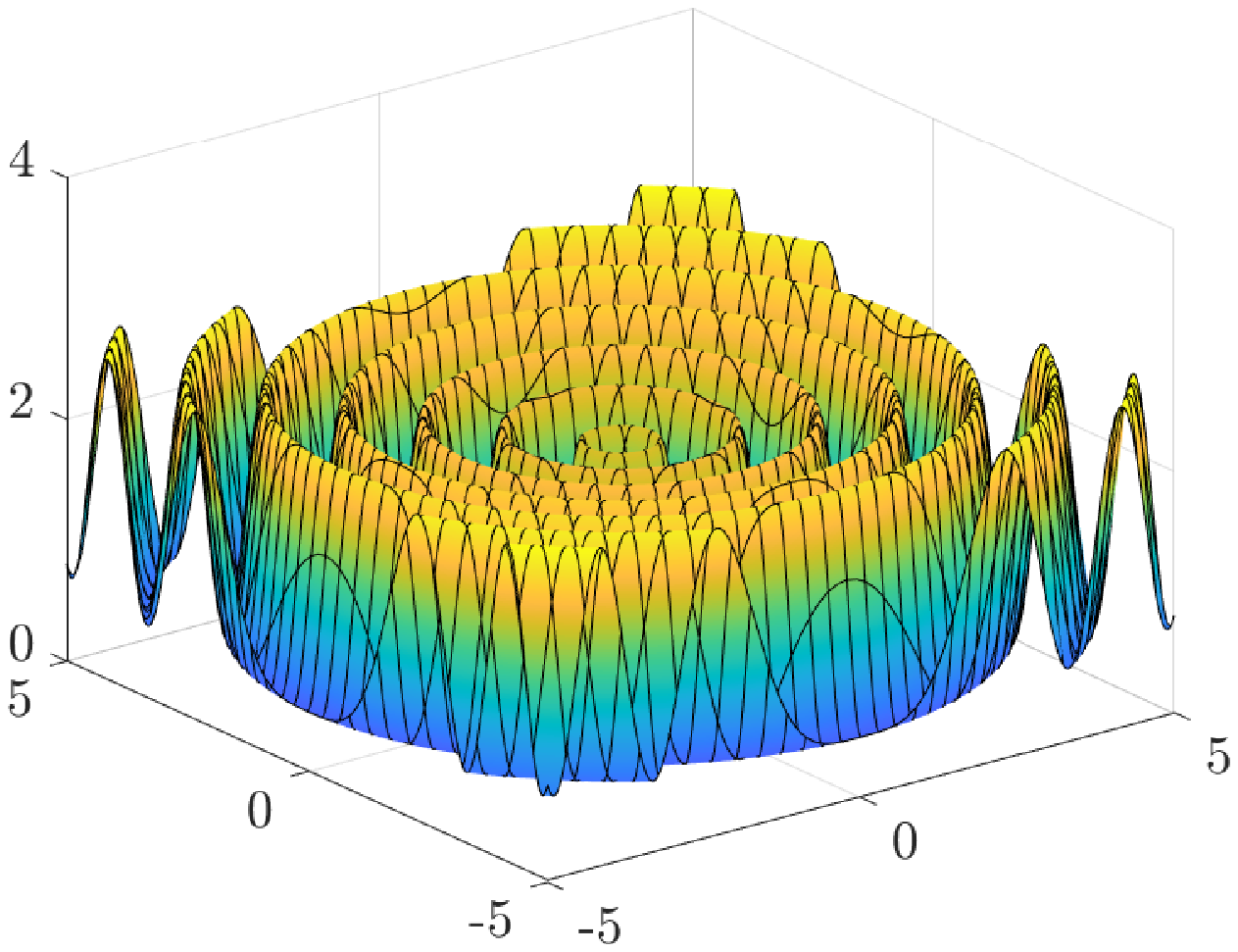} 
    \end{minipage}
    \\ \hline
     Schwefel 2.20 & $\sum_{i = 1}^d \vert x_i \vert  $  & $[-100,100]^d$ & $ (0,\dots,0)$ & 0 &  \begin{minipage}{.25\textwidth}  \centering \vspace{5pt}
      \includegraphics[scale=0.20]{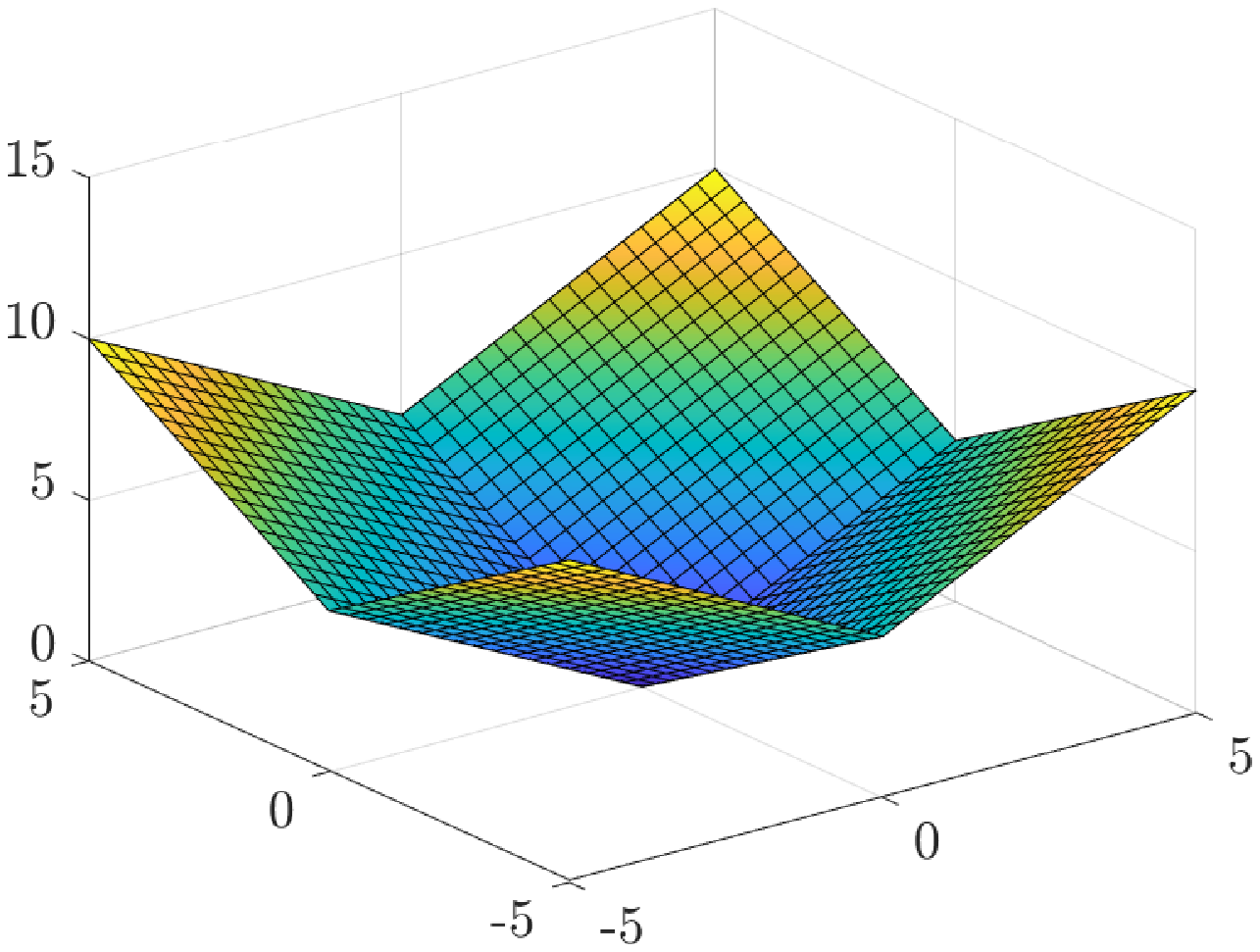} 
    \end{minipage}
    \\ \hline
     XSY random & $\sum_{i=1}^{d} \eta_i \vert x_i \vert^{i}, \qquad \eta_i \sim \mathcal{U}(0,1)$  & $[-5,5]^d$ & $ (0,\dots,0)$ & 0 &  \begin{minipage}{.25\textwidth}  \centering \vspace{5pt}
      \includegraphics[scale=0.20]{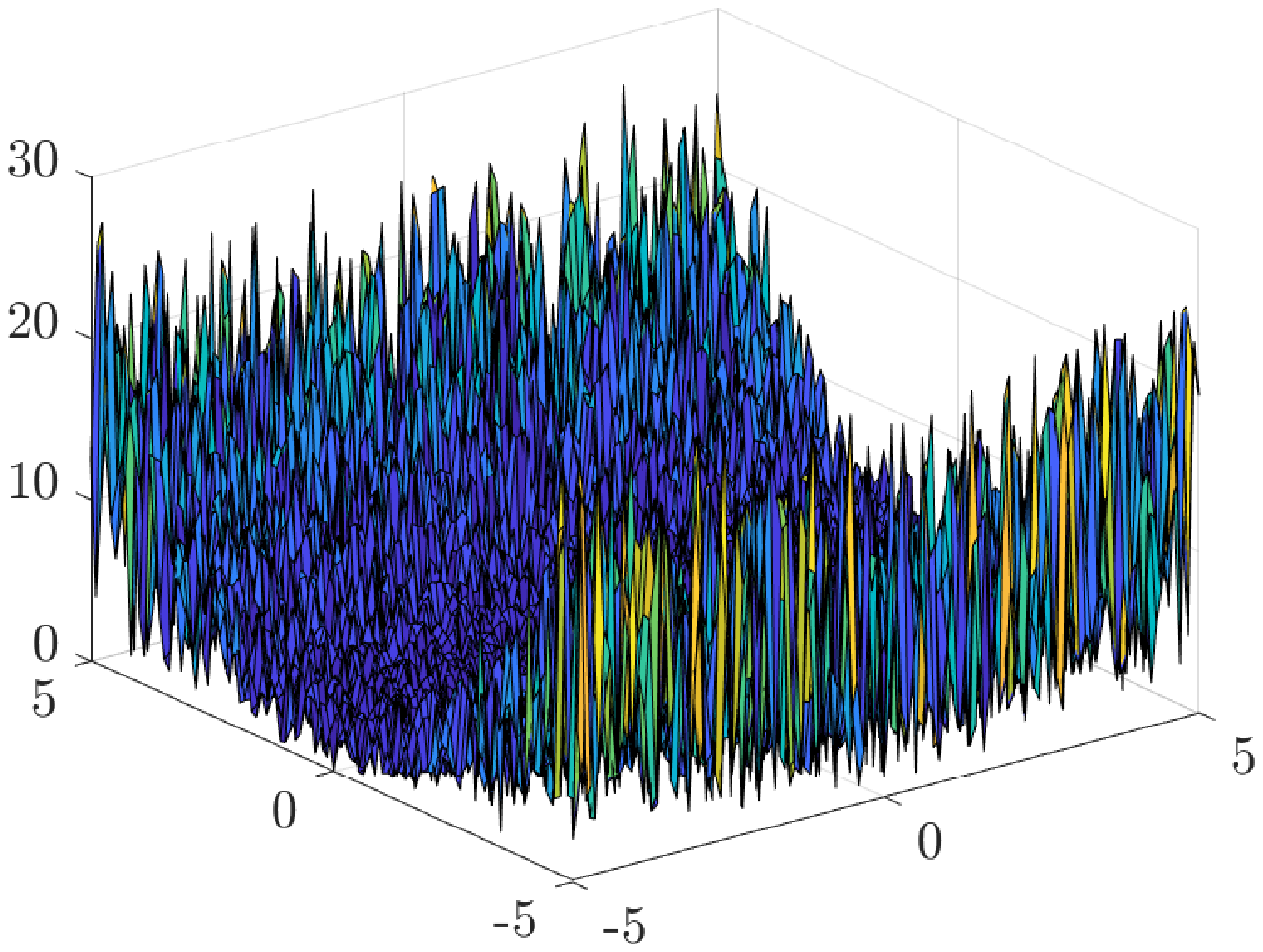}
    \end{minipage}
    \\ \hline
     XSY 4 & $\left( \sum_{i=1}^{d} \sin^2(x_i) - e^{ \ -\sum_{i=1}^{d} (x_i)^2} \right) \, e^{ \ -\sum_{i=1}^{d} \sin^2{\sqrt{\vert x_i \vert}}}   $  & $[-10,10]^d$ & $ (0,\dots,0)$ & $-1$ &  \begin{minipage}{.25\textwidth} \centering \vspace{5pt}
      \includegraphics[scale=0.20]{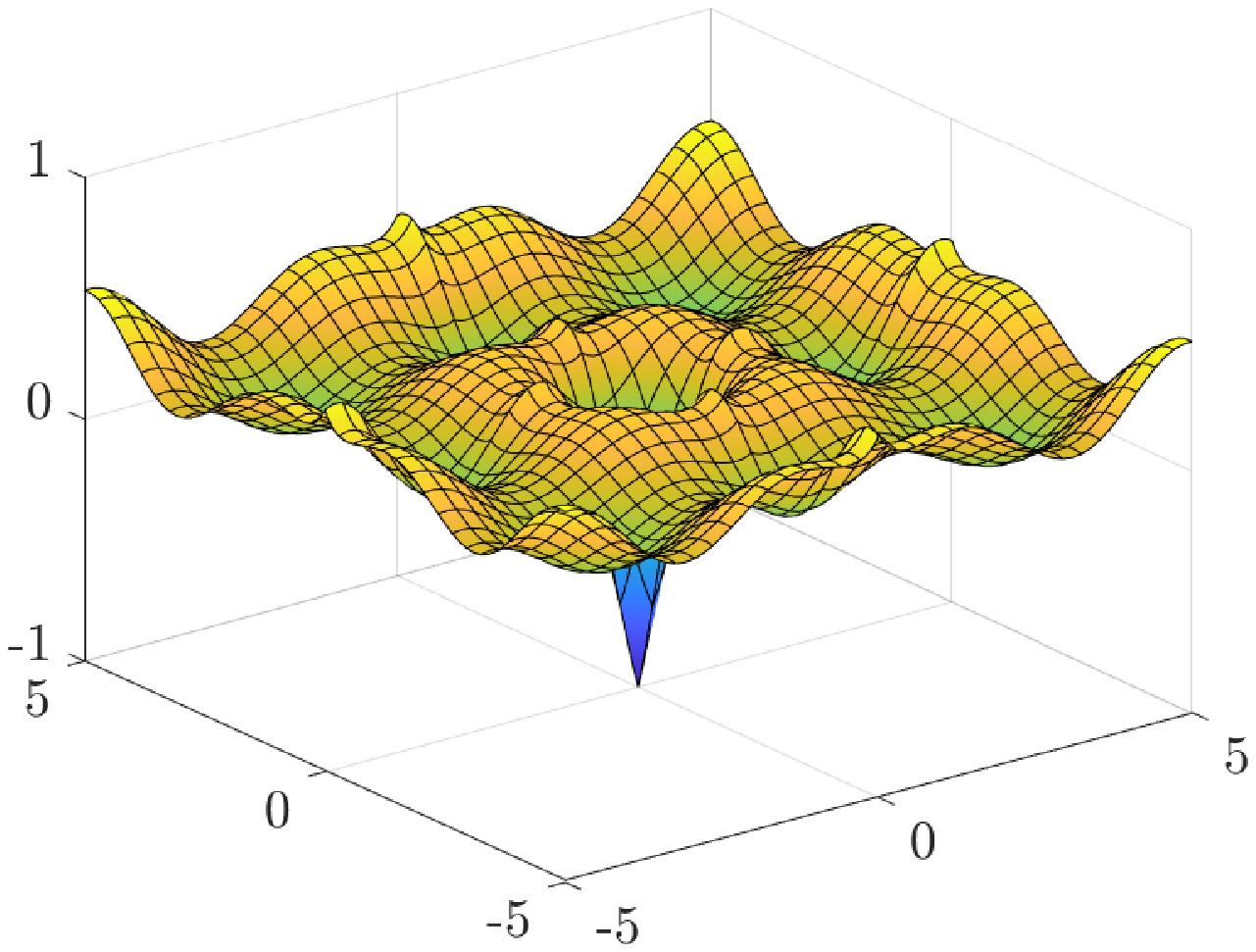} 
    \end{minipage}
    \\ \hline
    \hline
  \end{tabular}}
 \end{center}
 \caption{Prototype test functions for global optimization.}
 \label{Tabfun}
 }
\end{table}

\subsubsection{Comparison on prototype functions}

In the last test case we analyze the performance of the methods by solving simultaneously a set of different optimization functions considered in their standard search domains \cite{JYZ} (see Table \ref{Tabfun}). Here, instead of trying to find an optimal set of parameters for each function we use the same parameters for all functions. Furthermore, in order to identify a comparable set of optimization parameters for the different functions, we found it particularly effective to rescale all functions from their classical domain to the same reference domain. In our experiment we generalized the notion of success criteria by introducing the following definitions
\begin{itemize}
\item  the \textit{success rate}, computed averaging over $n_r$ runs and using as convergence criterion 
\[
\begin{split}
\Vert\bar\P_\alpha^{n_*}-x^{\ast}\Vert_{\infty}< \delta_{err}\quad {\rm or} \quad  |\TE(\bar\P_\alpha^{n_*})-\TE(x^{\ast})|< \delta_{fun}
  \end{split}
  \label{ConvCrit2}
\]
where $x^{\ast}$ is the minimum and $n_*$ the final time.
\item  The \textit{average function value} $\mathcal{F}_{avg}$, computed averaging the function value $\TE(\bar\P_\alpha^{n_*})$ over $n_r$ runs.
\end{itemize}
\begin{figure}[tb]
\begin{minipage}{\linewidth}
\centering
\subcaptionbox{$\Vert\bar\P_\alpha^{n_*}-x^{\ast}\Vert_{\infty}$ with $\xi = 0$}{\includegraphics[scale=0.38]{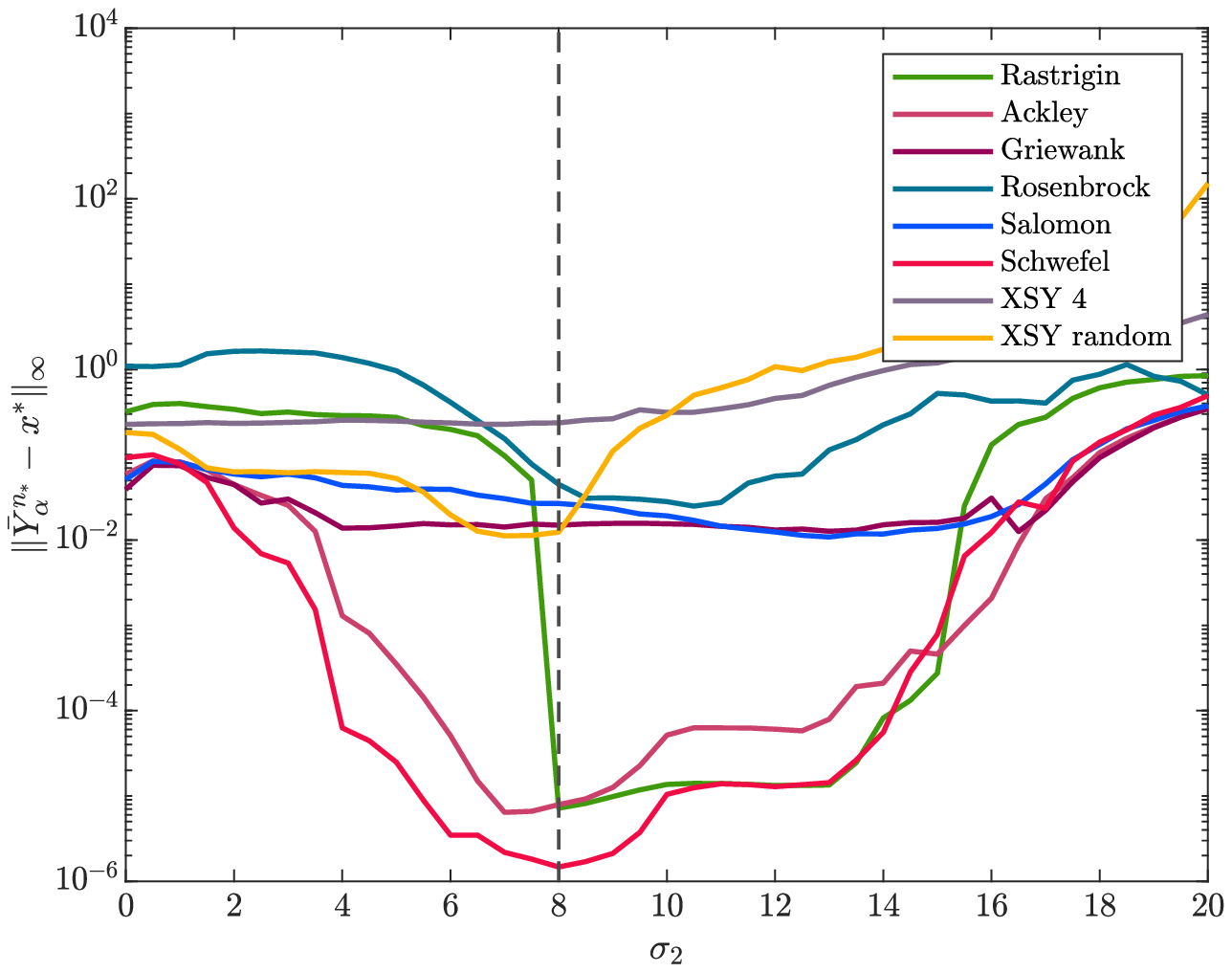}}
\subcaptionbox{$\Vert\bar\P_\alpha^{n_*}-x^{\ast}\Vert_{\infty}$  with $\xi = 0.25$}{\includegraphics[scale= 0.38]{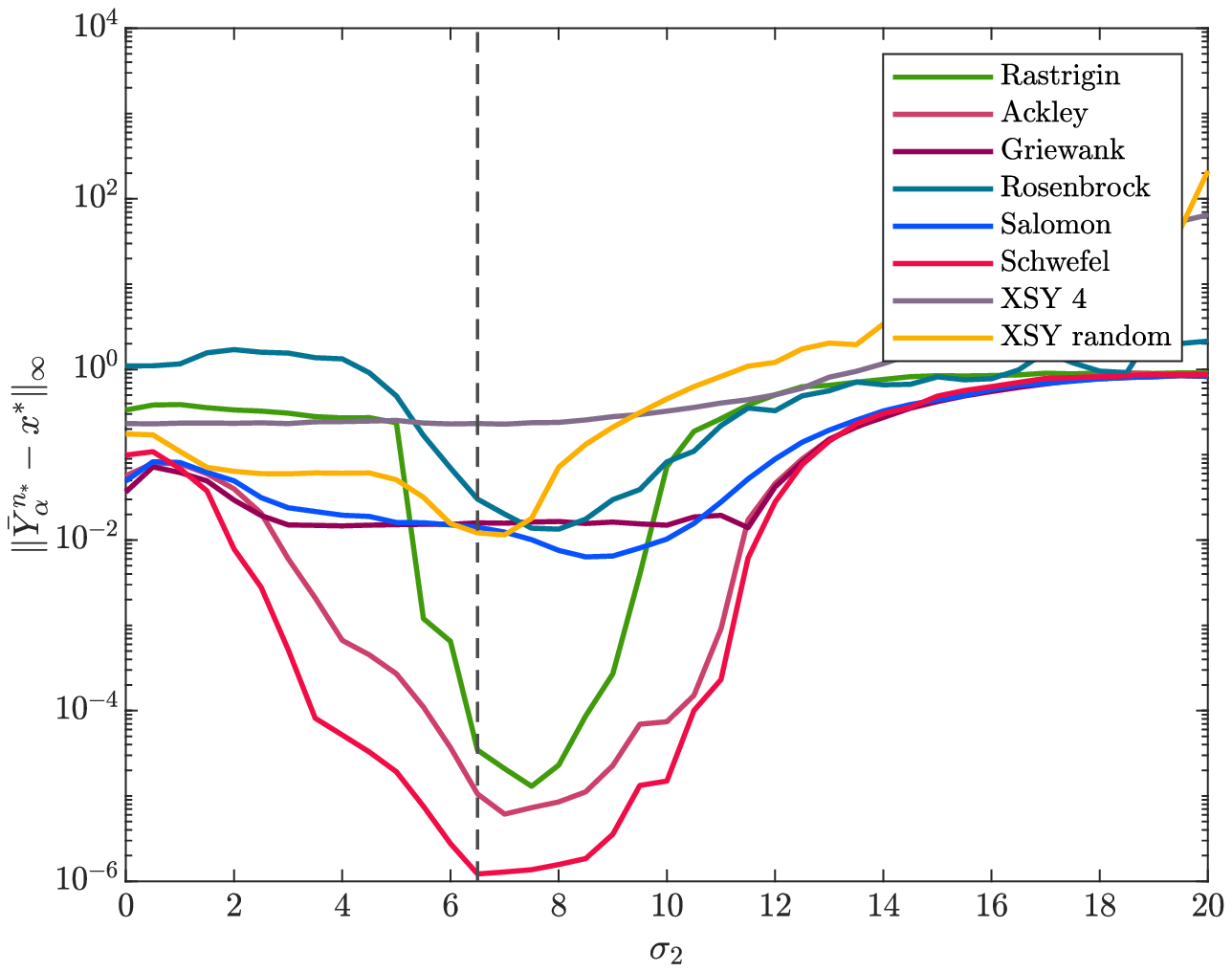}} \\
\subcaptionbox{$\mathcal{F}_{avg}$ with $\xi = 0$}{\includegraphics[scale=0.38]{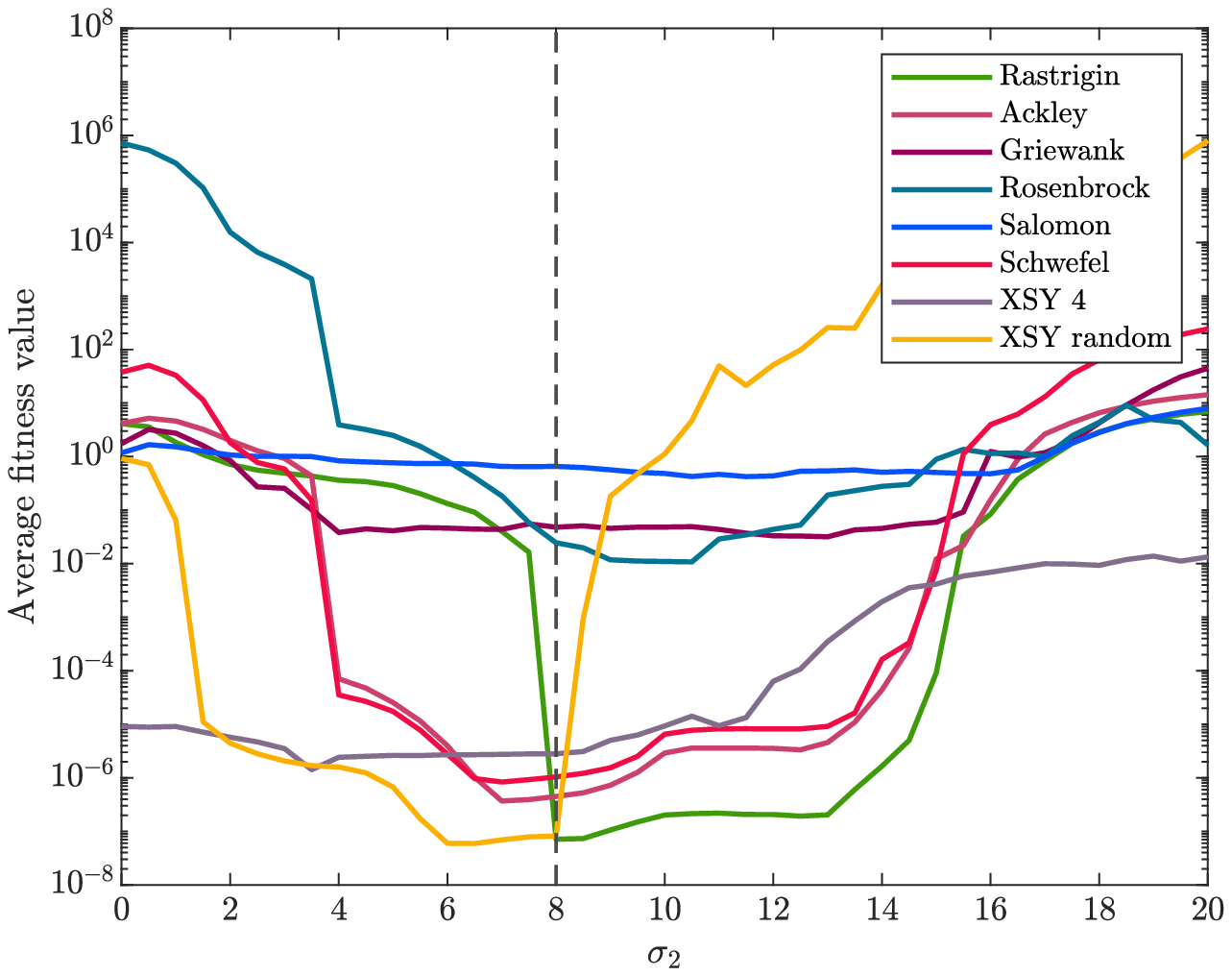}}
\subcaptionbox{$\mathcal{F}_{avg}$ with $\xi = 0.25$}{\includegraphics[scale= 0.38]{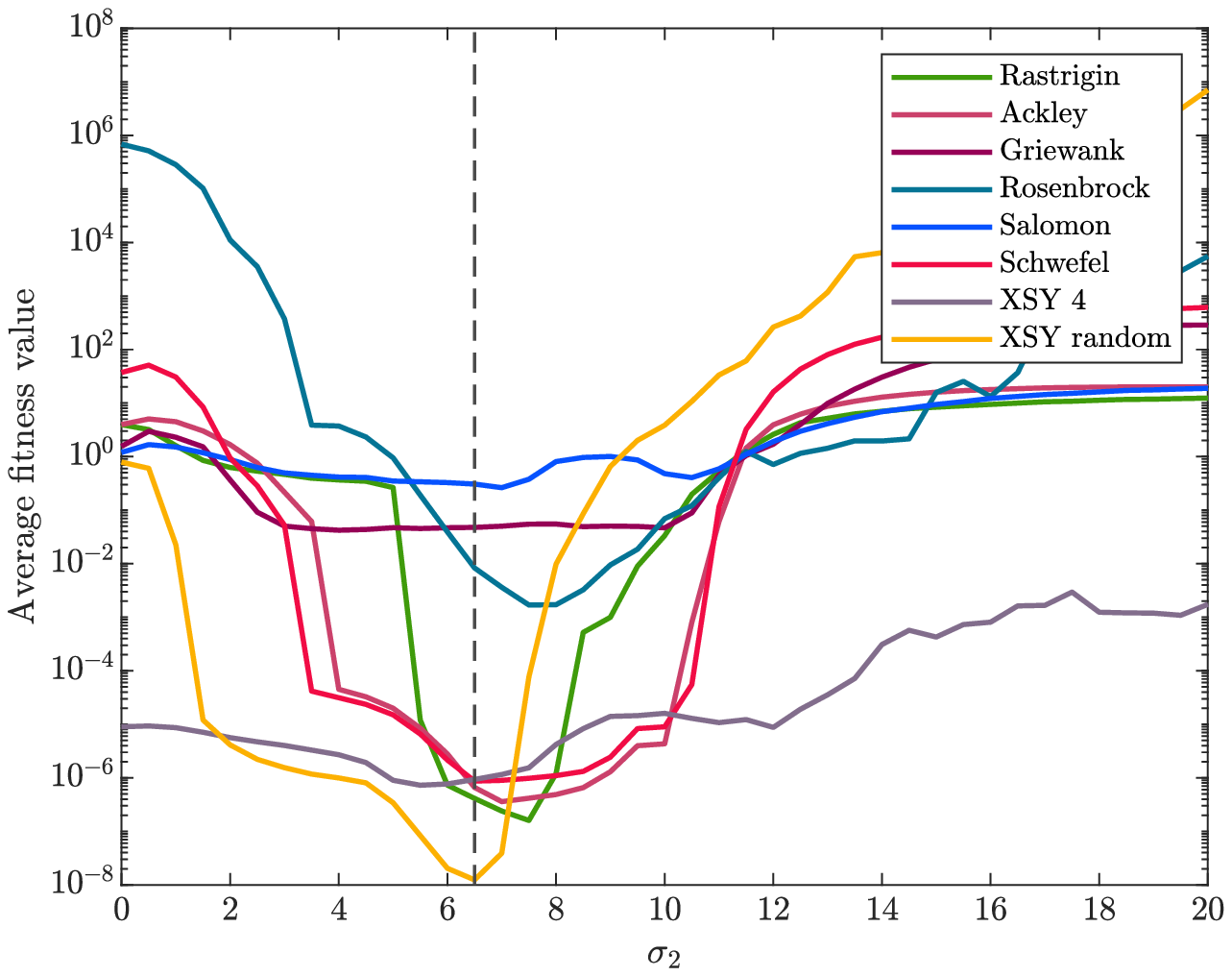}} \\
\caption{SD-PSO with memory ($m=0$). Behavior of the average error (top) and fitness value (bottom) for different values of $\sigma_2$. Here $\sigma_1 = \xi \cdot \sigma_2$, $\lambda_1 = \xi \cdot \lambda_2$, $\lambda_2=1$, $\Delta t = 0.01$, $\nu=50$, $\beta = 3 \times 10^3$ and $\alpha = 5 \times 10^4$. The dashed vertical lines are the estimated optimal values.}
\label{Fig8}
\end{minipage}
\end{figure}
In our simulations, we set $[-1,1]^d$ as the reference domain and translate the functions so that all have a minimum value of $\TE(x^*)=0$. 
We selected $\delta_{err}=0.1$, $\delta_{fun}=0.01$, $n_r=500$ and $n_{max}=10^4$.
We let most parameters fixed as in previous test case, namely $\alpha = 5 \times 10^4$, $\beta = 3 \times 10^3$, $\nu=50$. Additionally we keep $m=0$, $\Delta t = 0.01$, and for a given value of $\xi=0$ (absence of local best) and $\xi=0.25$ (local best weighted $1/4$ of global best) estimate the value for $\sigma_2$ in order to maximize the average convergence rate among all functions. This has been done with simple variations of step $0.5$ for $\sigma_2$ in the simulations, according to results in Figure \ref{Fig8} where we considered the behavior of the average error and fitness value for different values of $\sigma_2$ calculated over $n_r$ runs.


\begin{table}[tb]
\begin{center}\resizebox{4.5in}{2.3in}{
{\renewcommand{\arraystretch}{1.2}
\begin{tabular}{l|lccc|ccc}
 & \multicolumn{4}{c}{Case $\xi = 0$, $\sigma_2 = 8.0$} &\multicolumn{3}{c}{ \hspace{5pt} Case $\xi = 0.25$, $\sigma_2 = 6.5$ \hspace{5pt} } \\
\hline
\hline
& & $N=50$ & $N=100$ & $N=200$ & $N=50$ & $N=100$ & $N=200$ \\  
\hline
{\rm \textbf{Ackley}}& Rate  &100.0\% & 100.0\% & 100.0\%  &100.0\% & 100.0\% & 100.0\% \\
& Error &  9.44e-05 & 3.57e-05 & 1.48e-05 & 9.25e-06 & 4.40e-06 & 2.02e-06\\
& $\mathcal{F}_{avg}$ &2.61e-05 & 1.04e-05& 8.49e-06 &2.65e-05 & 1.26e-05 & 5.78e-06\\
& $n_{iter}$ &1012.5 & 847.9 & 736.2 &1033.4 & 874.3 & 764.0\\
\hline
{\rm \textbf{Griewank}}& Rate  &100.0\% & 100.0\% & 100.0\% &100.0\% & 100.0\% & 100.0\% \\
 & Error & 2.28e-02 & 2.24e-02 & 2.19e-02 &  2.27e-02 & 2.16e-02 & 2.24e-02\\
& $\mathcal{F}_{avg}$ &5.57e-02 & 5.21e-02 & 4.26e-02 &5.25e-02 & 4.93e-02 & 2.28e-02 \\
& $n_{iter}$ &1010.8 & 861.6 & 761.7 &1006.3 & 734.7 & 626.6 \\
\hline
{\rm \textbf{Rastrigin}}& Rate &34.0\% & 70.7\% & 95.0\% &9.0\% & 26.4\% & 42.0\% \\
& Error &1.78e-05 & 1.89e-05 & 2.05e-05 &  3.01e-05 & 3.12e-05 & 3.03e-05\\
& $\mathcal{F}_{avg}$ &9.32e-08 & 9.68e-08 & 9.95e-08 &2.41e-07 & 2.58e-07 & 2.44e-07 \\
& $n_{iter}$ &1308.5 & 1122.9 & 970.5& 1631.0 & 1483.0 & 1334.8 \\
\hline
{\rm \textbf{Rosenbrock}}& Rate  &49.3\% & 84.7\% & 100.0\% &87.3\% & 100.0\% & 100.0\% \\
& Error &  2.60e-02 & 3.44e-02 & 1.08e-02 &  4.87e-02 & 3.32e-02 & 6.92e-03\\
& $\mathcal{F}_{avg}$ &8.58e-02 & 1.25e-02 & 9.30e-03 &2.12e-02 & 8.01e-03 & 3.23e-04 \\
& $n_{iter}$ &8009.3 & 8392.8 & 7358.0 &9669.8 & 9553.8 & 7925.7\\
\hline
{\rm \textbf{Schwefel 2.20}}& Rate  &100.0\% & 100.0\% & 100.0\% &100.0\% & 100.0\% & 100.0\% \\
& Error &  2.11e-05 & 1.73e-06 & 7.32e-07 &  3.65e-06 & 1.63e-06 & 1.09e-06\\
& $\mathcal{F}_{avg}$ &2.93e-03 & 4.99e-04 & 2.18e-04 &5.14e-05 & 2.46e-05 & 8.01e-06 \\
& $n_{iter}$ &865.9 & 749.8 & 668.3 &863.2 & 747.0 & 665.8 \\
\hline
{\rm \textbf{Salomon}}& Rate  & 84.7\% & 98.7\% & 100.0\% &100.0\% & 100.0\% & 100.0\% \\
& Error & 8.94e-02 & 6.45e-02 & 4.99e-02 & 3.72e-02 & 3.21e-02 & 2.75e-02\\
& $\mathcal{F}_{avg}$ &8.96e-01 & 6.66e-01 & 5.24e-01&3.83e-01 & 3.21e-01 & 2.75e-01 \\
& $n_{iter}$ &1749.3 & 1657.9 & 1631.9 &2193.7 & 1749.7 & 1138.2\\
\hline
{\rm \textbf{XSY random \ }}& Rate  &90.0\% & 99.3\% & 100.0\% &100.0\% & 100.0\% & 100.0\% \\
& Error &  4.11e-02 & 2.26e-02 & 1.14e-02 &   2.45e-02 & 1.67e-02 & 1.66e-02\\
& $\mathcal{F}_{avg}$ &5.64e-07 & 9.60e-08 & 6.06e-08 &9.75e-09 & 7.26e-09 & 4.56e-09 \\
& $n_{iter}$ &10000.0 & 10000.0 & 10000.0 &10000.0 & 10000.0 & 10000.0\\
\hline
{\rm \textbf{XSY 4 \ }}& Rate  &100.0\% & 100.0\% & 100.0\% & 100.0\% & 100.0\% & 100.0\% \\
& Error &  1.09e+00 & 9.85e-01 & 9.70e-01 & 8.56e-01 & 8.19e-01 & 7.97e-01\\
& $\mathcal{F}_{avg}$ & 2.88e-05 & 2.57e-05 & 7.44e-05 & 1.69e-07 & 1.42e-07 & 1.41e-07 \\
& $n_{iter}$ &9682.5 & 9018.1 & 8861.6 &10000.0 & 10000.0 & 10000.0\\

\hline  
\hline    
\end{tabular}}}
\end{center}
\caption{SD-PSO with memory ($m=0$) for $\lambda_1 = \xi \cdot \lambda_2$, $\sigma_1 = \xi \cdot \sigma_2$, $\lambda_2=1$, $\Delta t = 0.01$, $\nu=50$, $\beta = 3 \times 10^3$, $\alpha = 5 \times 10^4$.}
\label{tab:TestFunctions}
\end{table}

The results in Table \ref{tab:TestFunctions} confirm the potential of the method in identifying correctly the global minima for different heterogeneous test functions. Overall, with the exception of the Rastrigin function for which  the local best produces a reduction in the convergence rate using this set of parameters, the importance of the local best is evident. In particular, the presence of the local best yields a reduction in the number of iterations for the Griewank, the Rosenbrock and the Salomon functions and an increase in the convergence rate for the XSY random and XSY4 functions. Except for the Griewank and Solomon functions, the computed value of the objective function is consistently close to zero and improves by increasing the number of particles. 
Finally, we emphasize that it was beyond the scope of this survey to discuss additional practical improvements to the algorithms that can be adopted to improve the success rate and the efficiency, like the use of random batch methods \cite{AlPa,carrillo2019consensus,JLJ}, particle reduction techniques \cite{fhps20-2,fornasier2021anisotropic} and parameters adaptivity \cite{poli2007particle,wang2017particle}. 
We refer to \cite{grassi2021consensus} for further details on these implementation aspects.

\section{Concluding remarks and research directions}
PSO methods represent a particularly prominent category within global optimization methods that do not make use of the gradient of the objective function. The popularity of these methods is related to the versatility and robustness of the algorithms, the good scalability that allows dealing with high-dimensional problems, and the ability to identify the global minimum effectively even in the case of non-convex and possibly non-smooth functions. Despite this, a complete mathematical theory related to the derivation of such methods and their global convergence properties is still lacking.

In this work, relying on some recent results \cite{Grassi2021PSO,huang2021mean1,cipriani2021zero,HJK,grassi2021consensus}, we have made an important step towards the construction of a general mathematical theory for the rigorous analysis and the understanding of PSO methods. 
The starting point of our analysis is a generalization of PSO methods in the context of second-order stochastic differential equations. In addition to the continuous formulation of PSO algorithms this novel class of methods generalizes the particle optimization process by making the alignment and exploration coefficients, based on the corresponding drift and diffusion dynamics, independent.

In the mean-field limit, using a regularized version of these SD-PSO systems, we obtained a Vlasov-Fokker-Planck type equation describing the MF-PSO dynamics. In addition, we rigorously studied the behavior of the system for small values of the inertia parameter showing how in such a limit the MF-PSO dynamics converges to a generalization of CBO models containing the local best. The latter result allowed us to clarify the relationships between these two classes of meta-heuristic optimization methods. A convergence result to the global minimum for a wide class of objective function is then proved in the case where the dynamic does not take into account memory effects. 
A complete gallery of numerical examples illustrate on the one hand the theoretical results obtained and on the other hand how the new class of SD-PSO methods potentially presents several advantages over traditional PSO in terms of convergence speed and solution stability. 

These results open important perspectives in the area of mathematical understanding of particle swarming optimization methods and in the construction of new algorithms. Among the many research directions some, not exhaustive, are summarized below.
\begin{enumerate}
\item[-] The majority of PSO applications are limited to single objective and unconstrained optimization problems. Therefore, the development of methods capable to deal with multi-objective and constrained optimization problems is a challenging and interesting area of research.
\smallskip
\item[-] Most of the convergence results for mean-field PSO and CBO models refer to the global best only. Generalization of these results to include the effect of the local best and its role should be studied. Convergence rate estimates of practical interest are still limited and further analysis is necessary.
\smallskip
\item[-] Similarly to classical PSO algorithms, the computational parameters are usually determined according to specific problems and require considerable application experience and numerous experimental tests. The identification of optimal parameters and the implementation of adaptive techniques for their determination is thus fundamental for many applications.
\end{enumerate}

 \section*{Acknowledgments}
This work has been written within the activities of GNCS group of INdAM (National Institute of High Mathematics). The support of MIUR-PRIN Project 2017, No. 2017KKJP4X "Innovative numerical methods for evolutionary partial differential equations and applications" and of the ESF PhD grant "Mathematical and statistical methods for machine learning in biomedical and socio-sanitary applications" is acknowledged. H. H. is partially supported by the Pacific Institute for the Mathematical Sciences (PIMS) postdoc fellowship. J. Q. is partially supported  by the National Science and Engineering Research Council of Canada (NSERC) and by the start-up funds from the University of Calgary.

\bibliographystyle{ims-rv-van} 
\bibliography{biblio} 

\end{document}